\documentclass[11pt,preprint]{article}

\usepackage{amsmath, latexsym, amsfonts, amssymb, amsthm}
\usepackage{graphicx, color,hyperref,dsfont,epsfig,caption,wrapfig,subfig}
\usepackage{pstricks}
\usepackage{tikz}
\usepackage{movie15}
\hypersetup{
    linktoc=page,
    linkcolor=red,          
    citecolor=blue,        
    filecolor=blue,      
    urlcolor=cyan,
    colorlinks=true           
}

\numberwithin{equation}{section}

\newtheorem{theorem}{Theorem}[section]
\newtheorem{lemma}[theorem]{Lemma}
\newtheorem{proposition}[theorem]{Proposition}
\newtheorem{corollary}[theorem]{Corollary}

\newtheorem{remark}[theorem]{Remark}

\newcounter{conj}
\newtheorem{conjecture}[conj]{Conjecture}

\theoremstyle{definition}

\renewcommand{\tilde}{\widetilde}          
\DeclareMathSymbol{\leqslant}{\mathalpha}{AMSa}{"36} 
\DeclareMathSymbol{\geqslant}{\mathalpha}{AMSa}{"3E} 
\DeclareMathSymbol{\eset}{\mathalpha}{AMSb}{"3F}     
\renewcommand{\leq}{\;\leqslant\;}                   
\renewcommand{\geq}{\;\geqslant\;}                   

\newcommand{\C}{\mathbb{C}}
\newcommand{\D}{\mathbb{D}}
\newcommand{\R}{\mathbb{R}}
\newcommand{\Z}{\mathbb{Z}}
\renewcommand{\H}{\mathbb{H}}
\newcommand{\N}{\mathbb{N}}

\newcommand{\E}{\mathds{E}}
\renewcommand{\P}{\mathds{P}}

\usepackage{xparse}

\DeclareDocumentCommand \Pmp { m m o} {
\IfNoValueTF{#3}
{P_{#1}^{#2}}
{P_{#1}^{#2}\left(#3\right)}
}

\DeclareDocumentCommand \Emp { m m o} {
\IfNoValueTF{#3}
{E_{#1}^{#2}}
{E_{#1}^{#2}\left[#3\right]}
}

\DeclareDocumentCommand \Pbr { m m m m o } {
\IfNoValueTF{#5}
{P_{#1}^{#2\stackrel{#4}{\rightarrow} #3}}
{P_{#1}^{#2\stackrel{#4}{\rightarrow} #3}\left(#5\right)}
}

\DeclareDocumentCommand \Ebr { m m m m o } {
\IfNoValueTF{#5}
{E_{#1}^{#2\stackrel{#4}{\rightarrow} #3}}
{E_{#1}^{#2\stackrel{#4}{\rightarrow} #3}\left[#5\right]}
}

\def\S{\mathbb{S}}

\def\bi{\begin{itemize}}
\def\ei{\end{itemize}}

\def\bnum{\begin{enumerate}}
\def\enum{\end{enumerate}}
\def\<#1{\langle #1 \rangle}

\title{Liouville Quantum Gravity on the unit disk}

\author{ Yichao Huang \footnote{ENS Ulm, DMA, 45 rue d'Ulm,  75005 Paris, France.} , R\'emi Rhodes \footnote{Universit{\'e} Paris-Est Marne la Vall\'ee, LAMA, Champs sur Marne, France.}, Vincent Vargas \footnote{ENS Ulm, DMA, 45 rue d'Ulm,  75005 Paris, France.} }

\date{\vspace{-5ex}}

\begin{document}

\maketitle

 \begin{abstract}
Our purpose is to pursue the rigorous construction of Liouville Quantum Field Theory  on Riemann surfaces initiated by F. David, A. Kupiainen and the last two authors in the context of the Riemann sphere and inspired by the 1981 seminal work by Polyakov. In this paper, we investigate the case of simply connected domains with boundary.  We also make precise conjectures about the relationship of this theory to scaling limits of random planar maps with boundary conformally embedded onto the disk.  
\end{abstract}
 
\noindent{\bf Key words or phrases:}  Liouville Quantum Gravity, quantum field theory, Gaussian multiplicative chaos, KPZ formula, KPZ scaling laws, Polyakov formula, conformal anomaly.


\tableofcontents

\section{Introduction}
Let us begin this introduction with a soft attempt of explanation for mathematicians of what is Liouville Quantum Field Theory (LQFT). This theory may be better understood if we first briefly recall the Feynman path integral representation  of Brownian motion on $\R^d$.  Denoting by $\Sigma$ the space of  paths $\sigma: [0,T]\to \R^d$ starting from $\sigma(0)=0$, we define the action functional on $\Sigma$ by
\begin{equation}\label{BM}
\forall \sigma \in\Sigma,\quad S_{{\rm BM}}(\sigma)=\frac{1}{2}\int_0^T|\dot{\sigma}(r)|^2\,dr.
\end{equation}
It is nowadays rather well understood that Brownian motion, call it $B$, can be understood in terms of Feynman path integrals via the relation
\begin{equation}\label{BMpath}
\E[F((B_s)_{s\leq T})]=\frac{1}{Z}\int_{\Sigma}F(\sigma)e^{-S_{BM}(\sigma)}D\sigma
\end{equation}
 where $D\sigma $ stands for a formal uniform measure on $\Sigma$ and $Z$ a renormalization constant. Brownian motion is also often said to be the canonical uniform random path in $\R^d$: this terminology is due to the fact the Brownian motion is the scaling limit of the simple random walk.

The reader may try to guess what could be the above picture if, instead of ``canonical random path'', we ask for a ``canonical random Riemann surface''. The answer is LQFT\footnote{Liouville Quantum Field Theory (LQFT) and Liouville Quantum Gravity (LQG) are similar for the unit disk but they differ on higher genus surfaces, see \cite{GRV} for references and discussions.}.  As in the case of Brownian motion, there are two ways to give sense to this theory: directly in the continuum in terms of Feynmann surface integrals or as scaling limit of suitable discrete models called Random Planar Maps (RPM). This picture is nowadays well understood in the physics literature since the pioneering work by Polyakov \cite{Pol}.  The reader is referred to \cite{witten,nakayama} for physics reviews, to \cite{Pol,cf:Da,DistKa,cf:KPZ} for founding papers in physics and to \cite{DKRV} for a brief introduction for mathematicians and a rigorous construction on the Riemann sphere.

In this paper, we will construct LQFT on Riemann surfaces with boundary directly in the continuum in the spirit of   Feynman surface integrals. More precisely, we consider   a (strict) simply connected domain $D$ of $\R^2$ with a simple boundary  equipped with a Riemannian metric $g$. Similar to the action \eqref{BM} for Brownian motion, we must consider  the Liouville action functional on such a Riemannian manifold. It is defined for each function $X: \overline{D}\to\R$ by
\begin{equation}\label{actionfalse}
S(X,g):= \frac{1}{4\pi}\int_{D}\big(|\partial^{g}X |^2+QR_{g} X  + 4\pi \mu e^{\gamma X  }\big)\,\lambda_{g} 
+\frac{1}{2\pi}\int_{\partial D}\big(QK_{g} X  +   2\pi \mu_\partial e^{\frac{\gamma}{2} X  }\big)\,\lambda_{\partial g} 
\end{equation}
where $\partial^{g}$, $R_{g} $, $K_{g} $, $\lambda_{g}$ and $\lambda_{\partial g}$ respectively stand for the gradient, Ricci scalar curvature, geodesic curvature (along the boundary), volume form and line element along $ \partial D$ in the metric $g$: see section \ref{defmetric} for the definitions. The parameters $\mu,\mu_\partial\geq  0$ (with $\mu+\mu_\partial>0$) are respectively the bulk and boundary cosmological constants  and  $Q,\gamma$ are real parameters.

\medskip
Before going into further details of the quantum field theory, let us first make a detour in Riemannian geometry to explain why the roots of LQFT are deeply connected to the theory of uniformization of Riemann surfaces. Indeed, a fundamental problem in geometry is to uniformize the surface $( \overline{D},g)$: this means that we look for a metric $g'$ on $ \overline{D}$ conformally equivalent to $g$, i.e. $g'=e^{u}g$ for some smooth function $u$ on $ \overline{D}$, with constant Ricci scalar curvature  in $D$ and constant geodesic curvature on $\partial D$.  Under appropriate assumptions, the unknown function $u$ is a minimizer of the Liouville action functional \eqref{actionfalse}. Indeed, for the particular value
\begin{equation}\label{Qconf}
 Q=\frac{2}{\gamma},
 \end{equation}
the saddle points  $X$ of this functional  with Neumann boundary condition $\partial_{n_{g}}  (\frac{\gamma}{2} X)+K_g=- \frac{\pi \mu_{\partial} \gamma^2}{2} e^{\frac{\gamma}{2}X}$, where $\partial_{n_{g}}$ stands for the Neumann operator along $\partial D$, solve (if exists) the celebrated {\it Liouville equation} 
\begin{equation}\label{eqliouv}
-\triangle_g (\gamma X)+R_g=-2\pi  \mu \gamma^2 e^{\gamma X} \quad \text{on }D,\quad  \partial_{n_{g}}  (\frac{\gamma}{2} X)+K_g=- \frac{\pi \mu_{\partial} \gamma^2}{2} e^{\frac{\gamma}{2}X}     \quad \text{on }\partial D.
\end{equation}
Setting $u=\gamma X$   and defining a new metric $g'=e^{u}g$, the metric $g'$ satisfies the relations
$$ R_{g'}=-2\pi \mu \gamma^2 \quad \text{and }\quad  K_{g'} =-\frac{ \pi \mu_\partial \gamma^2}{2},$$
hence providing a solution to the uniformization problem of the Riemann surface $( \overline{D},g)$. Let us further mention that, for the value of $Q$ given by \eqref{Qconf}, this theory is {\it conformally invariant}: this means that if we choose a conformal map $\psi: \, \tilde{D} \mapsto D$ then the couple $(X,g)$ solves \eqref{eqliouv} on $D$ if and only if $(X \circ \psi +Q\ln|\psi'|, g \circ \psi)$ solves \eqref{eqliouv} on $\tilde{D}$ \footnote{Let us prove this for the Neumann boundary condition; the other equation can be dealt with similarly. Since $\psi$ is an isometry from $(\tilde{D}, g \circ \psi |\psi'|^2)$ to $(D,g)$, we have $K_{g \circ \psi |\psi'|^2}= K_g \circ \psi$. Now applying formula \eqref{geod} which is valid in great generality, we get that $K_{g \circ \psi}= |\psi'| (K_g \circ \psi- \frac{1}{(g \circ \psi)^{1/2} |\psi'|} \ln |\psi'|) $. Hence we get that $$ \partial_{n_{g \circ \psi}}  (\frac{\gamma}{2} ( X \circ \psi + Q \ln |\psi'|))+K_{g \circ \psi}= |\psi'| (\frac{1}{g \circ \psi} \frac{\partial ( \frac{\gamma}{2}X)}{\partial n} \circ \psi +K_g \circ \psi )=  - \frac{\pi \mu_{\partial} \gamma^2}{2} e^{\frac{\gamma}{2}(X \circ \psi + Q \ln |\psi'|)}  $$ }. These are the foundations of the theory of uniformization of surfaces with boundary in $2d$, also called  {\it Classical Liouville field theory}.  


\medskip
In quantum (or probabilistic) Liouville field theory, one looks for the construction of a random field $X$ 
with law given heuristically  in terms of a  functional integral 
\begin{equation}\label{pathintegral}
\E[F(X)]=Z^{-1}\int F(X)e^{-S(X,g)}DX 
\end{equation}
where $Z$ is a normalization constant and $DX$ stands for a formal uniform measure on some space of maps  $X: \overline{D}\to\R$. This expression is in the same spirit as for the Brownian motion \eqref{BMpath}. This formalism describes the law of the log-conformal factor $X$ of a formal random metric of the form $e^{\gamma X}g$ on $D$. Of course, this description is purely formal and giving a mathematical description of this picture is a longstanding problem since the work of Polyakov \cite{Pol}.  It turns out that for the particular values 
$$\gamma\in ]0,2],\quad Q=\frac{2}{\gamma}+\frac{\gamma}{2},$$ this field theory is expected to become a {\it Conformal Field Theory} (see \cite{gaw} for a background on this topic). 
The aim of this paper is to make rigorous sense of the above heuristic picture and thereby to define a canonical random field $X$ inspired by Feynman surface integrals. A noticeable difference with the example of Brownian motion where there is only one canonical random path (up to reparametrization) is that there is a whole family of canonical random Riemann surfaces indexed by a single parameter $\gamma\in]0,2]$. Conformal Field Theories are characterized by their central charge $c\in\R$ that reflects the way the theory reacts to conformal changes of the background metric $g$ defined on $D$ (see Section~\ref{sec:weyl}). For the LQFT, we will establish that the central charge  is $c=1+6 Q^2$: thus it  can range continuously in the interval $[25,+\infty[$ and this is one of the interesting features of this theory.  We will also study the conformal covariance (KPZ formula) and $\mu,\mu_\partial$-dependence of this theory. Once constructed, the Liouville (random) field $X$ allows us to define  the  Liouville measure, which can be thought of as the volume form associated to the random metric tensor  $e^{\gamma X}g$. We will  state a precise mathematical conjecture  on the relationship between the Liouville measure and the scaling limit of random planar maps with a simple boundary conformally embedded onto the unit disk.
 
To conclude,  let us stress that the thread of the paper is inspired by \cite{DKRV} where the authors developped LQFT on the Riemann sphere. The main input is here to understand the phenomena related to the presence of a boundary; in particular, part of the construction relies on the theory of Gaussian multiplicative chaos (GMC) and the presence of the boundary requires to integrate against GMC measures functions that are not integrable with respect to Lebesgue measure when approaching the boundary (these technical difficulties do not appear in the case of the sphere \cite{DKRV} where there is no boundary): see Proposition~\ref{finitemeas} for instance.  

\subsection{On the difference between the David-Kupiainen-Rhodes-Vargas approach and the Duplantier-Miller-Sheffield approach}

There is a conceptual difference between the approach of  Duplantier-Miller-Sheffield developped in \cite{DMS} and the approach developed independently by David-Kupiainen and the last two authors of this paper in the work \cite{DKRV} where  was developped LQFT on the Riemann sphere. The point of view and the objects defined in both approaches are different though one can relate both approaches in a specific case as we now describe in the case of the sphere. 

\subsubsection{The case of the Riemann sphere}

In the work of David-Kupiainen-Rhodes-Vargas \cite{DKRV}, the authors construct the correlation functions of LQFT on the Riemann sphere $\S^2=\C\cup \lbrace \infty \rbrace$ and show that these correlation functions satisfy the axioms of a Conformal Field theory (CFT): conformal covariance (KPZ formula), Weyl anomaly, etc... The theory is indexed just like the disk by a parameter $\gamma \in ]0,2]$ (with $Q=\frac{\gamma}{2}+\frac{2}{\gamma}$) and a positive cosmological constant $\mu>0$ (recall that in the case of the disk, the theory requires two cosmological constants: one for the bulk and one for the boundary). The theory is now based on the following action 
\begin{equation}\label{Liouvillesphere}
S_{\S^2}(X,g):= \frac{1}{4\pi}\int_{\S^2}\big(|\partial^{g}X |^2+QR_{g} X  + 4\pi \mu e^{\gamma X  }\big)\,\lambda_{g} 
\end{equation}
where this time $\partial^{g}$, $R_{g} $, $\lambda_{g}$ respectively stand for the gradient, Ricci scalar curvature, volume form in the metric $g$. The symmetry of the theory is completely determined by $\gamma$, however the constant $\mu>0$ is essential for the existence of the theory and in particular the correlation functions. As an output of the construction, one can define (see Les Houches lecture notes \cite{Houches} for a simple introduction to the theory):
\begin{itemize}
\item
Correlation functions at points $(z_i)_{1 \leq i \leq n}$ and with weights $(\alpha_i)_{1\leq i\leq n}$ of ``variables'' $e^{\alpha_i X(z_i)}$ (called vertex operators in the physics literature) which correspond to a rigorous definition of the following heuristic path integral formulation
\begin{equation}\label{Correlation}
\langle \prod_{1 \leq i \leq n}  e^{\alpha_i X(z_i)}\rangle:= \int  \prod_{1 \leq i \leq n}  e^{\alpha_i X(z_i)}  e^{-S_{\S^2}(X,g)} DX.  
\end{equation}
Of course, to make sense of \eqref{Correlation}, one needs to regularize $X$ and take the limit as the regularization step goes to $0$.
These correlation functions exist if the $\alpha_i$ satisfy the so-called Seiberg bounds:
\begin{equation}\label{Seib}
\sum^n_{i=1} \alpha_i >2Q, \quad \text{and} \quad \alpha_i <Q, \; \forall i.
\end{equation}
\item
Random measures we will denote $Z_{(\alpha_i,z_i)_{1\leq i \leq n}}$ (when the $\alpha_i$ satisfy the bounds \eqref{Seib}) and unit volume random measures we will denote $Z_{(\alpha_i,z_i)_{1 \leq i \leq n}}^1$ (the unit volume measures can be defined under less restrictive conditions than \eqref{Seib}: see \cite{Houches}). These are the so-called Liouville measures of the theory. The random measure $Z_{(\alpha_i,z_i)_{1 \leq i \leq n}}$ is a rigorous construction of $e^{\gamma X} \lambda_{g}$ under the probability measure
\begin{equation*}
F \mapsto \frac{\langle  F \: \prod_{1 \leq i \leq n}  e^{\alpha_i X(z_i)}\rangle }{\langle \prod_{1 \leq i \leq n}  e^{\alpha_i X(z_i)}\rangle}.
\end{equation*}  
Most of these measures can be related (conjecturally) to planar maps; however, we will only discuss the simplest measure among these measures, namely $Z_{(\gamma,0), (\gamma,1), (\gamma,\infty)}^1$, which is related to the random measure constructed in the work of Duplantier-Miller-Sheffield \cite{DMS}.
\end{itemize}

One important aspect of the construction is that it provides very explicit expressions of the correlations and the distributions for the measures in terms of products of fractional moments of appropriate GMC measures: see Les Houches lecture notes \cite{Houches}. Also, the moments of the measures can be expressed in terms of the correlation functions. More precisely, one has the following formula for a measurable set $B \subset \S^2$ and any integer $p \geq 1$ 
\begin{equation}\label{MomLiou}
\E  \left [      \left ( Z_{(\alpha_i,z_i)_{1 \leq i \leq n}}   (B) \right )^p    \right  ]  = \frac{ \int_B \cdots \int_B  \langle  \prod_{1 \leq j \leq p} e^{\gamma X(x_j)}     \prod_{1 \leq i \leq n}   e^{\alpha_i X(z_i)}   \rangle      \lambda_g( d x_1)  \cdots \lambda_g( d x_p) }{ \langle       \prod_{1 \leq i \leq n}   e^{\gamma X(z_i)}   \rangle  }.
\end{equation}   
Hence, it is an essential program to compute these correlation functions since they determine for instance the moments of the measures $Z_{(\alpha_i,z_i)_{1 \leq i \leq n}}$.

\vspace{0.2 cm}

In the case of the sphere, the work of Duplantier-Miller-Sheffield \cite{DMS} constructed the so-called unit area quantum sphere we will denote $\mu_{DMS,\gamma}$ and which depends on $\gamma \in (0,2)$. In this setting, there is no cosmological constant and therefore no correlation functions strictly speaking. The unit area quantum sphere is in fact an equivalence class of random distributions with two marked points $0$ and $\infty$ (to be precise, one can also construct an equivalence class with one marked point but we will not discuss this case). More specifically, the authors first define an equivalence class of random surfaces with the following definition: a random distribution $h_1$ is equivalent to $h_2$ if and only if there is a (possibly random) M\"obius transform $\phi$ that fixes $0$ and $\infty$ such that the following holds in distribution
\begin{equation}\label{relationquantumsphere}
h_2=h_1 \circ \phi + Q \ln |\phi'|. \footnote{The definition of the unit area quantum sphere is in fact a bit more general as one can consider other marked points than $0$ and $\infty$ and hence more general conformal maps (which do not necessarily fix the points $0$ and $\infty$).}
\end{equation}    
The unit area quantum sphere $\mu_{DMS,\gamma}$ is then defined by the equivalence class of a special distribution $h^{\star}$ where the radial part is sampled according to a special Bessel bridge and the non radial part is sampled according to the non radial part of a standard full plane GFF (to be precise one works in the cylinder $\R \times [0,2 \pi]$ which is conformally equivalent to $\C$): see \cite{DMS} for the exact definition of $h^{\star}$. Let us mention that $h^{\star}$ is such that the radial part of $h^{\star}$ attains its maximum on the circle of radius $1$ (in the cylinder coordinates). In this setting, all the distributions $h$ in  $\mu_{DMS,\gamma}$ define random measures by the relation  
\begin{equation}\label{defmeasurequantumsphere}
\mu_{h}= \lim_{\epsilon\to 0} \epsilon^{\gamma^2/2} e^{\gamma h_\epsilon(z)} dz  
\end{equation} 
where $h_\epsilon(z)$ is a circle average of $h$ with center $z$ and radius $\epsilon$. The measures are related by 
\begin{equation}\label{relationquantumspherebis}
\mu_{h_2} = \mu_{h_1} \circ \phi
\end{equation} 
if $h_1$ and $h_2$ are related by \eqref{relationquantumsphere}. In what follows, we will identify the unit area quantum sphere with the equivalence class of the measure $\mu_{h^{\star}}$ with respect to relation \eqref{relationquantumspherebis} (rather than the equivalence class of the distribution $h^{\star}$ with respect to \eqref{relationquantumsphere}): hence, we will say that two random measures represent the same quantum surface if they are of the form $\mu_{h_1}$ and $\mu_{h_2}$ with $h_1$ and $h_2$ related by \eqref{relationquantumsphere}. Therefore, in conclusion, with this slightly different definition,  the unit area quantum sphere $\mu_{DMS,\gamma}$ is an equivalence class of measures such that $\mu_{h^{\star}} \in \mu_{DMS,\gamma}$.  

The work \cite{DMS} is interesting because it couples measures in $\mu_{DMS,\gamma}$ to space filling variants of SLE curves: this provides an interesting framework to relate the measures in $\mu_{DMS,\gamma}$ to decorated planar maps. We now describe the relation between the two approaches in the next subsection.

\subsubsection*{Historics on the conjectured scaling limit of finite volume planar maps}

In a previous paper \cite{Sheff}, Sheffield constructed a candidate for the scaling limit of the volume form of infinite volume planar maps (i.e. the non compact case). However, he left open the construction of a candidate for the scaling limit of the volume form of finite volume planar maps (i.e. the compact case). In particular, in the case of the sphere, following the work \cite{Sheff}, it was clearly not expected among probabilists that there could be a rather explicit candidate to the following question:


\vspace{0.2 cm}

{\bf Question}:  if you fix three points $z_1,z_2,z_3$ in the sphere $\S^2$, what is the scaling limit of the volume form of large finite planar maps (equipped with a natural conformal structure) embedded in the sphere where you send three points chosen at random on the map to the three fixed points $z_1,z_2,z_3$? \footnote{As stated, this question is not quite precise  because one has to give a definition of what we mean by "natural conformal structure": we refer to the Les Houches notes \cite{Houches} for a complete and precise exposition of the above question.}

\vspace{0.2 cm}
Such an explicit candidate was in fact constructed in \cite{DKRV}: it is the measure we denote $Z_{(\gamma,0), (\gamma,1), (\gamma,\infty)}^1$ (in the case $z_1=0, z_2=1, z_3=\infty$). More precisely, after the work \cite{Sheff}, the independent works \cite{DKRV} and \cite{DMS} appeared simultaneously: both works provide a description of the conjectured scaling limit of large planar maps embedded in the sphere (in \cite{DMS}, the authors also consider the situation of the disk and the relation of these disk measures to the ones considered in this paper is expected to be similar to the relation between the measures on the sphere considered in \cite{DKRV} and \cite{DMS}), namely the measure $Z_{(\gamma,0), (\gamma,1), (\gamma,\infty)}^1$ in \cite{DKRV}  and the equivalence class of the measure $\mu_{h^{\star}}$ in \cite{DMS}, i.e. the unit area quantum sphere $\mu_{DMS,\gamma}$.  

The main result of the recent work by Aru-Huang-Sun \cite{AHS} is to link both approaches: more precisely, they show that $\mu_{h^{\star}}$ and $Z_{(\gamma,0), (\gamma,1), (\gamma,\infty)}$ are equal seen as quantum surfaces with two marked points $0$ and $\infty$. Since the two measures define the same quantum surface, one can relate both measures by a relation of the form \eqref{relationquantumspherebis}. 

On the one hand, one can perform the following procedure: choose a point $z$ at random according to the measure $\mu_{h^{\star}}$ and consider the image of this measure by the conformal map $\phi_z$ which sends $z$ to $1$ and fixes $0$ and $\infty$. Of course, in this setting, we have $\phi_z(x)= \frac{x}{z}$; this procedure defines a random measure we will denote $\mu_{DMS,\gamma}^{3}$. In more mathematical terms, the construction of $\mu_{DMS,\gamma}^{3}$ is: 
\begin{equation}\label{quantumsphere}
\mu_{DMS,\gamma}^{3}= \mu_{h^{\star}} \circ \phi_z^{-1}, \; \; z \sim \mu_{h^{\star}}
\end{equation}
where $z \sim \mu_{h^{\star}}$ means you sample $z$ along $\mu_{h^{\star}}$. The work \cite{AHS} establishes that $\mu_{DMS,\gamma}^{3}=Z_{(\gamma,0), (\gamma,1), (\gamma,\infty)}^1$ in distribution. The equality in distribution between these two measures is in fact non trivial to prove because the work \cite{DKRV} provides an explicit and tractable formula for the distribution $Z_{(\gamma,0), (\gamma,1), (\gamma,\infty)}^1$ (see expression (3.15) and (3.16) in Les Houches notes \cite{Houches}) whereas $\mu_{DMS,\gamma}^{3}$ is defined by a non explicit procedure (this is because in definition \eqref{quantumsphere} the point $z$ is a random variable correlated to $\mu_{h^{\star}}$).

On the other hand, one can show that $Z_{(\gamma,0), (\gamma,1), (\gamma,\infty)}^1$ can be defined as a limit of the form $\lim\limits_{\epsilon\to 0} \epsilon^{\gamma^2/2} e^{\gamma \tilde{h}_\epsilon(z)} dz$ where $\tilde{h}$ is a random distribution. Then one can consider the (random)  M\"obius transform $\phi$ which fixes $0$ and $\infty$ such that the radial part of $\tilde{h} \circ \phi + Q \ln |\phi'|$ (mapped to the cylinder $\R \times [0,2 \pi]$) has its maximum on the circle of radius $1$. We then have in distribution
\begin{equation}\label{quantumspherebis}
\mu_{h^{\star}}= \mu_{\tilde{h}} \circ \phi.
\end{equation}



\subsubsection*{Other topologies} 

Finally, let us mention that both approaches can be extended to other topologies than the sphere and the disk. As a matter of fact, the approach \cite{DMS} in the case of the disk and sphere is an extension to the compact case of the initial approach in \cite{Sheff} where was developped the theory for the half plane and the full plane. However, the approach \cite{DMS} has not been extended to the case of higher genus surfaces. The approach of \cite{DKRV} can be extended to compact Riemann surfaces of genus ${\bf g} \geq 1$: see \cite{DRV} and \cite{GRV}. However, since the approach of \cite{DKRV} is based on first defining correlation functions it seems unadapted to the non compact setting where such correlation functions do not necessarily exist.

\subsection*{Acknowledgements}
The authors would like to thank F. David and A. Kupiainen for interesting discussions on this topic. The authors would also like to thank the anonymous referees for making comments which hopefully enabled to improve the readability of the present paper. In particular, the question raised by one of the referee motivated our writing of Section~\ref{momentestimates} which establishes the unit volume Seiberg bound.

\section{Background and preliminary results}
In order to facilitate the reading of the manuscript, we gather in this section the basics in Riemannian geometry and probability theory that we will use throughout the paper. 

\subsection{Metrics on the unit disk}\label{defmetric}
Let us denote by $\D$ the unit disk in the complex plane $\C$ and $\partial \D$ its boundary. We consider the standard Laplace-Beltrami operator $\triangle$, the standard gradient $\partial$ and Lebesgue measure $d\lambda$ on $\D$, as well as the standard Neumann operator $\partial_n$ and Lebesgue measure $d\lambda_\partial$ on $\partial \D$, the operators being defined with respect to the Euclidean metric if no index is given. More generally, we say that a metric  $g=g(x)dx^2$  on the   unit disk is conformally equivalent to the Euclidean metric if $g(x)=e^{u(x)}  $ for some function $u:\bar{\D}:\to\R$ of class $C^1(\D)\cap C^0(\bar{\D})$ such that
\begin{equation}\label{cond:phi}
\int_\D|\partial u|^2\,d\lambda<+\infty.
\end{equation}
Notice that we use the same notation $g$ for the metric tensor and the function which defines it but this should not lead to confusions. In that case, the Laplace-Beltrami operator $\triangle_g$ and Neumann operator $\partial_{n_g}$ in the metric $g$ are given by
$$ \triangle_g=g^{-1}\triangle,\quad \text{ and }\quad \partial_{n_g}=g^{-1/2}\partial_{n}.$$
We denote respectively by $R_g$ and $K_g$ the Ricci scalar curvature  and geodesic curvature $K_g$ in the metric $g$.
If $g'=e^{\varphi}g$ is another metric on the unit disk conformally equivalent to the flat metric, we get the following rules for the changes of (geodesic) curvature under such a conformal change of metrics
\begin{align}
R_{g'}=&e^{-\varphi}(R_g-\triangle_g \varphi)& \text{ on }\D,\label{curv}\\
  K_{g'}=&e^{-\varphi/2}(K_g+\partial_{n_g}\varphi/2)& \text{ on }\partial \D.\label{geod}
\end{align}
For instance, when equipped with the Euclidean metric, the unit disk has Ricci scalar curvature $0$ and geodesic curvature $1$ along its boundary. Combining these data with the rules \eqref{curv}+\eqref{geod}, one can recover the explicit expressions of $R_g$ and $K_g$ for any metric $g$ conformally equivalent to the Euclidean metric.
We will also consider the volume form $\lambda_g$ on $\D$, the line element $\lambda_{\partial g}$ on  $\partial \D$, and the gradient $\partial^g$   associated to the metric $g$.

\medskip
Let us further recall the Gauss-Bonnet theorem
\begin{equation}\label{GB}
\int_{\D}R_g\,d\lambda_g+2\int_{\partial\D}K_g \,d\lambda_{\partial g}=4\pi\chi(\D),
\end{equation}
where  $\chi(\D)$ is the Euler characteristics of the disk (that is $\chi(\D)=1$), and the Green-Riemann formula 
\begin{equation}\label{GR}
\int_\D \psi\triangle_g\varphi \,d\lambda_g+\int_\D\partial^g\varphi\cdot\partial^g\psi \,d\lambda_g=\int_{\partial \D}\partial_{n_g}\varphi\psi\,d\lambda_{\partial g}.
\end{equation}

We will denote by $m_{\nu}(f)$ and $m_{\partial \nu}(f)$ the mean value of $f$ respectively in the disk $\D$ or the boundary $\partial \D$ with respect to a measure $\nu$ on $\D$ or $\partial \D$, that is
$$m_\nu(f)=\frac{1}{\nu(\D)}\int_\D f\,d\nu,\quad m_{\partial \nu}(f)=\frac{1}{\nu(\partial\D)}\int_{\partial\D} f\,d\nu.$$
If the measure $\nu$ is the volume form (or the line element on $\partial\D$) of some metric $g$, we will use the notation $m_g(f)$ (or $m_{\partial g}(f)$). When no reference to the metric $g$ is given ($m(f)$ or $m_{\partial}(f)$) this means that we work with the Euclidean metric.

\medskip
The Sobolev space $H^1(\D)$ is defined as the closure of the space of smooth functions on $\bar{\D}$ with respect to the inner product
$$\int_\D( fh +\partial f\cdot\partial h)\,d\lambda.$$
We denote by $H^{-1}(\D)$ its dual.

\medskip
Finally, we introduce the Green function $G$ of the Neumann problem on $\D$
\begin{equation}\label{GreenN}
G(x,y)=\ln \frac{1}{|x-y| |1-x \bar{y}|}.
\end{equation}
It is the unique function satisfying  
\begin{enumerate}
\item $x\mapsto G(x,y)$ is harmonic on $\D\setminus\{y\}$,
\item $x\mapsto G(x,y)+\ln|y-x|$ is harmonic on $\D$ for all $y\in\bar{\D}$,
\item $\partial_nG(x,y)=-1$ for $x\in\partial\D$, $y\in\D$,
\item $G(x,y)=G(y,x)$ for $x,y\in\D$ and $x\not= y$,
\item $m_{\partial}G(x,\cdot)=0$ for all $x\in \D$.
\end{enumerate}
Recall that \eqref{GR} combined with the properties of $G$ implies that for all $f\in C^2(\D)\cap C^1(\bar{\D})$
\begin{equation}\label{relgreen}
-2\pi (f(x)-m_{\partial \D}(f))=\int_{\D}G(x,y)\triangle f(y)\,\lambda(dy)-\int_{\partial\D}G(x,y)  \partial_n f(y)\,\lambda_{\partial}(dy) .
\end{equation}
It is quite important to observe here that $G$ is positive  definite on $\D$.

\subsection{M\"obius transforms of the unit disk}
The M\"obius transforms of the unit disk are given by $\psi(x)= e^{i \alpha}  \frac{x-a}{1-\bar{a}  x}$ with $|a|<1$. Recall that
\begin{equation*}
\psi'(x)= e^{i \alpha}  \frac{1-|a|^2}{(1-\bar{a}x)^2}
\end{equation*}
from which one gets
\begin{equation}\label{eq:mobiusformula}
\psi(y)-\psi(x)=(\psi'(y))^{1/2} (\psi'(x))^{1/2}(y-x), \; \; \; 1-\psi(x)\overline{\psi(y)}=(\psi'(x))^{1/2}(\psi'(y))^{1/2}(1-x\overline{y}).
\end{equation}
The Green function for the Neumann problem defined above thus verifies
\begin{equation}\label{greenpsi}
G(\psi(x),\psi(y))=G(x,y)-\ln|\psi'(x)|-\ln|\psi'(y)|.
\end{equation}

\subsection{Gaussian Free Field with Neumann boundary conditions}
We   consider on $\D$ a Gaussian Free Field (GFF) $X_{\partial \D}$ with Neumann boundary conditions  and vanishing mean along the boundary, namely   $m_{\partial}(X_{\partial \D})=0$ (see \cite{dubedat,She07} for more details about GFF).
This field is a Gaussian centered distribution (in the sense of Schwartz) with covariance kernel given by the Green function of the Neumann problem with vanishing mean along the boundary
\begin{equation}
\E[X_{\partial \D}(x)X_{\partial \D}(y)] =G(x,y).
\end{equation}
It can be shown that this Gaussian random distribution (in the sense of Schwartz) lives almost surely in $H^{-1}(\D)$ (same argument as in \cite{dubedat}).\\
~\\
As a distribution, the field $X_{\partial \D}$ cannot be understood as a fairly defined function. To remedy this problem, we will need to consider some regularizations of this field in order to deal with nice (random) functions. Thus, we introduce the regularized field $X_{\partial \D,\epsilon}$ as follows. For $\epsilon>0$, we let $l_\epsilon(x)$ be the length of the arc $A_\epsilon(x)=\{z\in\D;|z-x|=\epsilon\}$ (computed with the Euclidean line element $ds$ on the boundary of the disk centered at $x$ and radius $\epsilon$). Then we set
\begin{equation*}
X_{\partial \D,\epsilon}(x)= \frac{1}{l_\epsilon(x)}  \int_{A_\epsilon(x)}  X_{\partial \D}(s) ds.
\end{equation*}
A similar regularization was considered in \cite{cf:DuSh} and the reader can check that this field has a locally H\"older version both in the variables $x$ and $\epsilon$.  
Let us mention that we have the following two options: either $x\in\D$ and then for $\epsilon<{\rm dist}(x,\partial\D)$ we obtain
\begin{equation*}
X_{\partial \D,\epsilon}(x)=\frac{1}{2 \pi  }  \int_{0}^{2\pi}  X_{\partial \D}(x+\epsilon e^{i\theta}) d\theta,
\end{equation*}
or $x\in\partial\D$ and  then $X_{\partial \D,\epsilon}(x)$ is intuitively the same as above except that we integrate along the ``half-circle'' centered at $x$ with radius $\epsilon$ contained in $\D$.

\begin{proposition}\label{circlegreen}
Let us denote by $g_P$ the Poincar\'e metric over the unit disk
\begin{equation}\label{poincare}
g_P=\frac{1}{(1-|x|^2)^2}dx^2.
\end{equation}
We claim\\
1) As $\epsilon\to 0$, the convergence $\E[X_{\partial \D,\epsilon} (x)^2]+\ln \epsilon\to  \frac{1}{2}\ln g_P(x) $ holds uniformly over the compact subsets of $\D$.\\
2) As $\epsilon\to 0$,  the convergence $\E[X_{\partial \D,\epsilon} (x)^2]+2\ln \epsilon\to  -1 $ holds uniformly over  $\partial \D$.\\
3) Consider a M\"obius transform $\psi$ of the disk. Denote by $X_{\partial \D}\circ \psi_\epsilon$ the $\epsilon$-circle average of the field $X_{\partial \D}\circ \psi$. Then as $\epsilon\to 0$, we have the convergence $$\E[X_{\partial \D}\circ \psi_\epsilon(x)^2]+\ln \epsilon\to \frac{1}{2}\ln g_P(\psi(x))-2\ln|\psi'(x)|$$ uniformly over the compact subsets of $\D$ and the convergence $$\E[X_{\partial \D}\circ \psi_\epsilon(x)^2]+2\ln \epsilon\to -1-2\ln|\psi'(x)|$$ uniformly over   $\partial\D$.
 
\end{proposition}

\noindent {\it Proof.} To prove the first statement results, apply the $\epsilon$-circle average regularization to the Green function $G $ in \eqref{GreenN} and use the fact that the following integral vanishes $$\int_0^{2\pi}\int_0^{2\pi}\ln\frac{1}{|e^{i\theta}-e^{i\theta'}|}\,d\theta d\theta'=0$$ to get  the uniform convergence over compact subsets of $\E[X_{\partial,\epsilon} (x)^2]+\ln \epsilon$   towards $$x\mapsto\frac{1}{2}\ln g_P(x).$$
The strategy is similar for the second statement except that one gets $\pi^{-2}$ times the integral $\int_0^{\pi}\int_0^{\pi} \ln\frac{1}{|e^{i\theta}-e^{i\theta'}|}\,d\theta d\theta'$, which does not vanish anymore and yields the constant $-1$. 
The third claim results from \eqref{greenpsi}.\qed

\subsection{Gaussian Multiplicative Chaos}

Gaussian multiplicative chaos theory was introduced in \cite{cf:Kah}. The reader is referred to  \cite{review} for a review on the topic.  Here, we deal with convolution of the GFF so that as a straightforward combination of the main result in  \cite{shamov} and Proposition \ref{law}, we claim

\begin{proposition}\label{law}
For $\gamma\in[0,2[$ and $\lambda$, $\lambda_\partial$  the volume form and line element on $\D$, $\partial \D$ of the Euclidean metric, the random measures $e^{\gamma X_{\partial\D}}d\lambda$, $e^{\frac{\gamma}{2} X_{\partial\D}}d\lambda_\partial$ are defined as the limits in probability
\begin{equation*}
e^{\gamma X_{\partial\D}}d\lambda=\lim\limits_{\epsilon\to 0}\epsilon^{\frac{\gamma^2}{2}}e^{\gamma X_{\partial\D,\epsilon}}d\lambda\quad \quad e^{\frac{\gamma }{2}X_{\partial\D}}d\lambda_{\partial \D}=\lim\limits_{\epsilon\to 0}\epsilon^{\frac{\gamma^2}{4}}e^{\frac{\gamma}{2} X_{\partial\D,\epsilon}}d\lambda_\partial
\end{equation*}
in the sense of weak convergence of measures over $\D$, $\partial \D$. These limiting measures are non trivial and are two standard Gaussian Multiplicative Chaos (GMC) on $\D$, $\partial \D$, namely
\begin{equation*}
e^{\gamma X_{\partial\D}}d\lambda= e^{\gamma X_{\partial\D}(x)-\frac{\gamma^2}{2}\E[X_{\partial\D}(x)^2]}g_P(x)^{\frac{\gamma^2}{4}}\lambda(dx)\quad \quad e^{\frac{\gamma }{2}X_{\partial\D}}d\lambda_{\partial \D}=e^{-\frac{\gamma^2}{8}}e^{\frac{\gamma}{2}X_{\partial\D}(x)-\frac{\gamma^2}{8}\E[X_{\partial\D}(x)^2]}\lambda_\partial(dx).
\end{equation*}
\end{proposition}

\noindent Actually, the main issue is to show that these measures give almost surely finite mass respectively to the disk and its boundary. This turns out to be obvious for the boundary measure as the expectation of the total mass of $\partial\D$ is finite. Concerning the bulk  measure, this statement is not straightforward: observe for instance that the expectation is infinite
$$\E[\int_\D e^{\gamma X_{\partial\D}}d\lambda]=\int_\D g_P(x)^{\frac{\gamma^2}{4}}\lambda(dx)$$
as soon as $\gamma^2\geq 2$. Yet, we show in the following proposition that the random variable $\int_\D e^{\gamma X_{\partial\D}}d\lambda$ is almost surely finite for all values of $\gamma\in ]0,2[$.  
\begin{proposition}\label{finitemeas}
For $\gamma\in]0,2[$, the quantities below are almost surely finite
\begin{equation*}
\int_\D e^{\gamma X_{\partial\D}}d\lambda \quad \text{ and }\quad \int_{\partial\D }e^{\frac{\gamma }{2}X_{\partial\D}}d\lambda_{\partial \D}.
\end{equation*}
\end{proposition}

\begin{proof} As explained above, we only need to focus on the bulk measure. Observe first that its expectation is finite in the case $\gamma^2<2$. For $\gamma^2\geq 2$ (in fact the argument below works for $\gamma>1$), we prove that it has moments of small order $\alpha>0$, which entails the a.s. finiteness of the total mass of the interior of the disk.\\
Recall the sub-additivity inequality for $\alpha\in]0,1[$: if $(a_j)_{1\leq j\leq n}$ are positive real numbers then
$$(a_1+\dots+a_n)^\alpha\leq a_1^\alpha+\dots+a_n^\alpha.$$
Therefore we can write
\begin{align*}
 \mathds{E}\Big[\Big(\int_\mathds{D}&e^{\gamma X_{\partial\D}(x)-\frac{\gamma^2}{2}  \E[X_{\partial\D}(x)^2] } \frac{1}{(1-|x|^2)^{\gamma^2/2}})\lambda(dx)\Big)^\alpha\Big]\\
=&\mathds{E}\Big[\Big(\sum\limits_{n\in\mathds{N}}\int_{1-2^{-n} \leq |x|^2 \leq 1-2^{-n-1} }e^{\gamma X_{\partial\D}(x)-\frac{\gamma^2}{2}  \E[X_{\partial\D}(x)^2] } \frac{1}{(1-|x|^2)^{\gamma^2/2}}\lambda(dx)\Big)^\alpha\Big]\\
\leq& \sum\limits_{n\in\mathds{N}}2^{n\alpha \frac{\gamma^2}{2}}\mathds{E}\Big[ \Big(\int_{1-2^{-n} \leq |x|^2 \leq 1-2^{-n-1} } e^{\gamma X_{\partial\D}(x)-\frac{\gamma^2}{2}  \E[X_{\partial\D}(x)^2] }\lambda(dx)\Big)^\alpha\Big].
\end{align*}
\noindent Now we trade the GFF $X_{\partial\D}$ for a log-correlated field that possesses a nicer structure of correlations with the help of Kahane's convexity inequality \cite{cf:Kah}. More precisely, we consider any log-correlated field on $\R^2$ with a white noise decomposition and invariant under rotation. For instance, let us consider a star scale invariant kernel with compact support (see \cite{Rnew1}): we choose a positive definite isotropic positive  function $k$ with compact support of class $C^2$ and we set
$$K_\epsilon(x)=\int_1^{\epsilon^{-1}}\frac{k(ux)}{u}\,du.$$
We consider a family of Gaussian processes $(Y_\epsilon(x))_\epsilon$ such that (see \cite{Rnew1} for  the details of the construction of such fields)
$$\forall x,y\in\R^2,\quad \E[Y_\epsilon(x)Y_{\epsilon'}(y)]=K_{\max(\epsilon,\epsilon')}(x-y).$$
The reader may check that for all $r,r'\in ]0,1]$ such that $ 1-2^{-n} \leq r^2,r^{'2} \leq 1-2^{-n-1} $  and $\theta,\theta'\in[0,2\pi]$
$$\E[X_{\partial\D}(re^{i\theta})X_{\partial\D}(r'e^{i\theta'})]\geq 2 \E[Y_{2^{-n}}(e^{i\theta})Y_{2^{-n}}(e^{i\theta'})]-A$$ for some constant $A$ independent of $n,\theta$.
This inequality of covariances allows us to use Kahane's convexity inequality (see \cite{cf:Kah} or \cite[Theorem 2.1]{review}). Indeed,  because the map $x\mapsto x^\alpha$ is concave, we have for some standard Gaussian random variable $N$ independent of everything 
\begin{align*}
 \mathds{E}\Big[ \Big(e^{\gamma A^{1/2 }N-A\gamma^2/2}\int_{1-2^{-n} \leq |x|^2 \leq 1-2^{-n-1}} &e^{\gamma X_{\partial\D}-\frac{\gamma^2}{2}  \E[X_{\partial\D}^2] }d\lambda\Big)^\alpha\Big]\\
 &\leq  \mathds{E}\Big[ \Big(\int_0^{2\pi}\int_{(1-2^{-n})^{1/2}}^{(1-2^{-n-1})^{1/2}}  e^{\gamma \sqrt{2}Y_{2^{-n}}(e^{i\theta})- \gamma^2 \E[Y_{2^{-n}}(e^{i\theta})^2] }drd\theta\Big)^\alpha\Big]\\
  &=  C2^{-n\alpha}\mathds{E}\Big[ \Big(\int_0^{2\pi}   e^{\gamma \sqrt{2}Y_{2^{-n}}(e^{i\theta})- \gamma^2 \E[Y_{2^{-n}}(e^{i\theta})^2] }d\theta\Big)^\alpha\Big]
\end{align*}
for some constant $C$ independent of everything.
By  using the comparison to Mandelbrot's multiplicative cascades  as explained in \cite[Appendix B.1]{Rnew7} to use a moment estimate in \cite[Proposition 2.1 and the remark just after]{madaule}, we have  that for any $\alpha<\gamma^{-1}$ and some other constant $C>0$
$$\sup_{n}\mathds{E}\Big[ \Big(n^{\frac{3\gamma}{2}}2^{n(\gamma-1)^2}\int_0^{2\pi}   e^{\gamma \sqrt{2}Y_{2^{-n}}(e^{i\theta})- \gamma^2 \E[Y_{2^{-n}}(e^{i\theta})^2] } d\theta\Big)^\alpha\Big]\leq C.$$
Combining we get (up to changing the value of $C$ to absorb the constant $\E[e^{\alpha\gamma A^{1/2 }N-\alpha A\gamma^2/2}]$)
\begin{align*}
\mathds{E}\Big[\Big(\int_\mathds{D}&e^{\gamma X_{\partial\D}(x)-\frac{\gamma^2}{2}  \E[X_{\partial\D}(x)^2] } \frac{1}{(1-|x|^2)^{\gamma^2/2}})\lambda(dx)\Big)^\alpha\Big]\leq C \sum\limits_{n\in\mathds{N}}  2^{n\alpha \big(\frac{\gamma^2}{2}-1-(\gamma-1)^2\big)}n^{-\frac{3\gamma}{2}\alpha},
\end{align*}
which is finite (with $\alpha\in(0,\gamma^{-1})$) when $\gamma\in ]1,2[$ because $\frac{\gamma^2}{2}-1-(\gamma-1)^2<0$ when $\gamma\in(0,2)$.
\end{proof}

\section{Liouville Quantum Gravity on the disk}
We are now in a position to give a precise definition of the LQFT on the disk with marked points: $n$ points in the bulk $\D$ and $n'$ points on the boundary $\partial\D$.  In what follows, we will first give a necessary and sufficient condition (the Seiberg bounds) on these marked points in order that the LQFT is well defined.  This will allow us to give the definitions of the Liouville field and measure. Finally, we will explain how these objects behave under conformal changes of background metrics and conformal reparametrization of the domain. Basically, the approach is the same as in \cite{DKRV} but there are some technical differences in order to treat the interactions bulk/boundary.

\subsection{Definition and existence of the partition function}
LQFT on the disk will be defined in terms of three parameters $\gamma$,  $\mu$, $\mu_{\partial}$, respectively the coupling constant and the bulk/boundary cosmological constants, together with prescribed marked points. In this section, we will assume  that the parameters  $\gamma$,  $\mu$, $\mu_{\partial}$ satisfy
\begin{equation}\label{def:param}
\gamma\in ]0,2[,\quad  \mu,\mu_{\partial  }\geq 0\quad \text{ and }\quad \mu+\mu_{\partial  }>0. 
\end{equation}
Concerning the marked points,  we fix a set of $n$ points  $(z_i)_{1\leq i\leq n}$  in the interior of $\D$ together with $n$ weights  $(\alpha_i)_{1\leq i\leq n}\in \R^n$ and  $n'$ points $(s_j)_{1\leq j\leq n'}$ on the boundary   $\partial\D$ together with $n'$ weights  $(\beta_j)_{1\leq j\leq n'}\in \R^{n'}$. The family   $(z_i,\alpha_i)_i$ will be called bulk marked points  and  the family $(s_j,\beta_j)_j$   boundary marked points. 
 
\medskip 
Consider any metric $g=e^{\varphi}dx^2$ on the unit disk conformally equivalent to the Euclidean metric  in the sense of \eqref{cond:phi}.\\
Our purpose is now to define the partition function $\Pi_{ \gamma,\mu_{\partial  },\mu}^{(z_i,\alpha_i)_i,(s_j,\beta_j)_j}  (\epsilon,g,F)$ of LQFT applied to a functional $F$. This partition function formally corresponds to the Feynmann surface integral \eqref{pathintegral} with action \eqref{actionfalse}. Yet, a rigorous approach requires the  regularization procedure. This is the reason why we define the regularized partition function  for all $\epsilon\in]0,1]$ and  bounded continuous functional $F$ on $H^{-1}(\D)$ by
\begin{align}\label{eq:defPigeps}
&\Pi_{ \gamma,\mu_{\partial  },\mu}^{(z_i,\alpha_i)_i,(s_j,\beta_j)_j}  (\epsilon,g,F)\\
 =&e^{  \frac{1}{96\pi}\Big(\int_{\D}|\partial \ln g|^2\,d\lambda+\int_{\partial \D}4\ln g\,d\lambda_{\partial} \Big)}  \int_\R\E\Big[F( X_{\partial  \D}+c+Q/2\ln g)\prod_i \epsilon^{\frac{\alpha_i^2}{2}}e^{\alpha_i (c+X_{\partial\D,\epsilon}+Q/2\ln g)(z_i)} \nonumber\\
&\times\prod_j \epsilon^{\frac{\beta_j^2}{4}}e^{\frac{\beta_j }{2}(c+X_{\partial\D,\epsilon}+Q/2\ln g)(s_j)}\exp\Big( - \frac{Q}{4\pi}\int_{\D}R_{g} (c+X_{\partial  \D})  \,d\lambda_g - \mu e^{\gamma c}\epsilon^{\frac{\gamma^2}{2}}\int_{\D}e^{\gamma (X_{\partial \D,\epsilon}+ Q/2\ln g) }\,d\lambda \Big)\nonumber \\
&\exp\Big( - \frac{Q}{2\pi}\int_{\partial \D}K_{g} (c+X_{\partial\D})  \,d\lambda_{\partial g} - \mu_\partial e^{\frac{\gamma}{2} c}\epsilon^{\frac{\gamma^2}{4}}\int_{\partial\D}e^{\frac{\gamma}{2} (X_{\partial\D,\epsilon}+ Q/2\ln g) }\,d\lambda_{\partial } \Big) \Big]\,dc.\nonumber %
\end{align}
The constant $c$ which is integrated against the Lebesgue measure $dc$ is crucial in the definition and adds extra symmetry. In particular, one has the following equality in distribution for all M\"obius transform $\psi$ when $c$ is distributed according to the Lebesgue measure 
\begin{equation}\label{magicinvariance}
X_{\partial \D} \circ \psi +c  \overset{(Law)}{=}  X_{\partial \D}+c 
\end{equation}
To prove identity \eqref{magicinvariance}, recalll that $X_{\partial \D} \circ \psi -\frac{1}{2 \pi} \int_{\partial \D}  X_{\partial \D} \circ \psi \:  d \lambda_{\partial} \overset{(Law)}{=}  X_{\partial \D}$ and then use that the Lebesgue measure is invariant under translation. Identity \eqref{magicinvariance} in distribution is essential in proving the conformal invariance properties of the theory. Now, the first natural question is to inquire whether the  limit 
\begin{align}\label{eq:defPig}
 \Pi_{ \gamma,\mu_{\partial  },\mu}^{(z_i,\alpha_i)_i,(s_j,\beta_j)_j}  (g,F) := \lim_{\epsilon\to 0}\Pi_{ \gamma,\mu_{\partial  },\mu}^{(z_i,\alpha_i)_i,(s_j,\beta_j)_j}  (\epsilon,g,F). 
\end{align}
exists and is not trivial. Existence and non triviality will be phrased in terms of   the following three conditions 
\begin{align} 
&\sum_i\alpha_i+\frac{1}{2}\sum_j \beta_j>Q\label{bounds1}, \\
&\forall i\quad \alpha_i<Q,\label{bounds2}\\
&\forall j\quad \beta_j<Q.\label{bounds3}
\end{align}
We claim:
\begin{theorem}{\bf (Seiberg bounds)}\label{th:seiberg}
We have the following alternatives
\begin{enumerate}
\item Assume $\mu>0$ and $\mu_{\partial}\geq 0$. The partition function $\Pi_{\gamma,\mu_{\partial},\mu}^{(z_i,\alpha_i)_i,(s_j,\beta_j)_j}(g,1)$     converges and is non trivial if and only if  \eqref{bounds1}+\eqref{bounds2}+\eqref{bounds3} hold. 
\item Assume $\mu=0$ and $\mu_{\partial}>0$. The partition function $\Pi_{\gamma,\mu_{\partial},\mu}^{(z_i,\alpha_i)_i,(s_j,\beta_j)_j}(g,1)$      converges and is non trivial if and only if  \eqref{bounds1} +\eqref{bounds3} hold. 
\item In all other cases, we have $$\Pi_{\gamma,\mu_{\partial},\mu}^{(z_i,\alpha_i)_i,(s_j,\beta_j)_j}(g,1)=0\quad \text{ or }\quad  \Pi_{\gamma,\mu_{\partial},\mu}^{(z_i,\alpha_i)_i,(s_j,\beta_j)_j}(g,1)=+\infty.$$
\end{enumerate}
\end{theorem}

Along the computations involved in Theorem~\ref{th:seiberg}, we get the  expression below for the partition function when the metric $g$ is the Euclidean metric. Notice that considering the only Euclidean metric is not a restriction because we will see later that there is an explicit procedure to express the partition function in any background metric $g$ in terms of that in the Euclidean metric (Weyl anomaly, Subsection~\ref{sec:weyl}). 

\begin{proposition}{\bf (Partition function)}\label{prop:part}
Assume $g$ is the Euclidean metric $dx^2$. Then, in each case of Theorem~\ref{th:seiberg} ensuring existence and non triviality, we have
\begin{align}\label{reduced}
 \Pi_{ \gamma,\mu_{\partial  },\mu}^{(z_i,\alpha_i)_i,(s_j,\beta_j)_j}  (dx^2,F)=& \Big(\prod_i g_P(z_i)^{-\frac{\alpha_i^2}{4}}\Big)e^{C({\bf z},{\bf s})}\int_\R e^{\big(\sum_i\alpha_i+\sum_j\frac{\beta_j}{2}-Q\big)c}\\
& \E\Big[F( X_{\partial  \D}+H+c)   \exp\Big( - \mu e^{\gamma c} \int_{\D}e^{\gamma H}e^{\gamma X_{\partial \D}  }\,d\lambda   - \mu_\partial e^{\frac{\gamma}{2} c} \int_{\partial\D}e^{\frac{\gamma }{2}H}e^{\frac{\gamma}{2}  X_{\partial\D }  }\,d\lambda_{\partial } \Big) \Big]\,dc ,\nonumber
\end{align}
where 
\begin{align*}
H(x)=&\sum_i\alpha_iG(x,z_i)+\sum_j\frac{\beta_j}{2}G(x,s_j),\\
C({\bf z},{\bf s})=&\sum_{i<i'}\alpha_i\alpha_{i'}G(z_i,z_{i'})+\sum_{j<j'}\frac{\beta_j\beta_{j'}}{4}G(s_j,s_{j'})+\sum_{i,j}\frac{\alpha_i\beta_{j}}{2}G(z_i,s_{j})-\sum_j\frac{\beta_j^2}{8}.
\end{align*}
\end{proposition}

\medskip 
\noindent {\it Proof of Theorem~\ref{th:seiberg} and Proposition~\ref{prop:part}.} We begin with the Seiberg bound.  Because the conformal factor $\varphi$ of $g=e^\varphi dx^2$ is assumed to be smooth (i.e. of class $C^1$), we can assume without loss of generality that $\varphi=0$. The main lines of the argument will be similar to \cite[Section 3]{DKRV}, up to a few modifications that we explain below. First observe that Propositions~\ref{law} and \ref{finitemeas} ensure that the interaction terms
$$\lim_{\epsilon\to 0}  \epsilon^{\frac{\gamma^2}{4}}\int_{\partial\D}e^{\frac{\gamma}{2} X_{\partial\D,\epsilon} }\,d\lambda_{\partial }\quad \text{ and }\lim_{\epsilon\to 0} \epsilon^{\frac{\gamma^2}{2}}\int_{\D}e^{\gamma X_{\partial \D,\epsilon}  }\,d\lambda $$
are non trivial provided that $\gamma\in]0,2[$. Hence, following \cite[Section 3]{DKRV}, $\Pi_{\gamma,\mu_{\partial},\mu}^{(z_i,\alpha_i)_i,(s_j,\beta_j)_j}(g,1)<+\infty$ if and only if \eqref{bounds1} holds: roughly speaking, recall that basically this amounts to claiming that the integral ($A,A'$ are two strictly positive constants)
$$\int_\R e^{\big(\sum_i\alpha_i+\frac{1}{2}\sum_j\beta_j-Q\big)c} e^{-\mu e^{\gamma c}A-\mu_\partial e^{\frac{\gamma}{2} c}A'}\,dc$$ is converging if and only if \eqref{bounds1} holds.

Recall then that the remaining part of the proof in \cite[Section 3]{DKRV} consists in determining when a marked point causes the blowing up of the interaction measure, in which case $ \Pi_{ \gamma,\mu_{\partial  },\mu}^{(z_i,\alpha_i)_i,(s_j,\beta_j)_j}  (dx^2,F)=0$.  The reason why a marked point may cause the blowing up of the interaction measure is that these marked points are handled with the Girsanov transform and this amounts to determining whether the bulk/boundary measures integrates some singularities of the type $\frac{1}{|x-z_i|^{\alpha_i\gamma}}$ or $\frac{1}{|x-s_j|^{\frac{\beta_j}{2}\gamma}}$. This is what we study in more details below.

Here we have two types of marked points (in the bulk or along the boundary) and two interaction measures: boundary $e^{\frac{\gamma}{2} X_{\partial\D} }\,d\lambda_{\partial }$ or bulk $e^{\gamma X_{\partial \D}  }\,d\lambda$. A marked point $(z_i,\alpha_i)$ in the bulk questions whether the bulk measure integrates the singularity $x\mapsto e^{\alpha_i \gamma G(x,z_i)}$. This is exactly the same situation as in \cite[Section 3]{DKRV}. Therefore the conclusion is the same: $\alpha_i$ must be strictly less than $Q$. The same argument settles the case of the effect of a boundary marked point $(s_j,\beta_j)$ on the boundary measure: $\beta_i$ must be strictly less than $Q$.  

What is not treated in \cite[Section 3]{DKRV} is the effect of boundary marked points on the bulk measure: namely we have to determine when the measure $e^{\gamma X_{\partial \D}  }\,d\lambda$ integrates the singularity $x\mapsto e^{\frac{\beta_j }{2}\gamma G(x,s_j)}$ for some $s_j$ belonging to the boundary $\partial\D$. Observe that the situation is more complicated as the behavior of the bulk measure is  highly perturbed when approaching the boundary: recalling the expression of the bulk measure in Proposition \ref{law}, we see that on the one hand the deterministic density $g_P(x)^{\frac{\gamma^2}{4}}$ blows up along the boundary and on the other hand the field $X_{\partial \D}$ acquires more and more correlations, which become maximal along the boundary: as $x$ approaches the boundary, $G(x,y)$ tends to behave like $2\ln\frac{1}{|x-y|}$ rather than $\ln\frac{1}{|x-y|}$. 

Let us now analyze the situation. We want to prove that the singularity is integrable if and only if $\beta_j<Q$. Without loss of generality, we assume that $s_j=1$. In what follows, $C$ stands for some generic constant, which may change along the lines and does not depend on relevant quantities.

Let us first assume that the singularity is integrable, more precisely for some $\delta$ fixed small enough
\begin{equation}\label{limcond}
\lim_{\epsilon\to 0}\int_{\D\cap B(1,\delta)}e^{\frac{\beta_j }{2}\gamma G_\epsilon(\cdot,1)}\epsilon^{\frac{\gamma^2}{2}} e^{\gamma X_{\partial \D,\epsilon}}\,d\lambda<+\infty
 \end{equation}  
where $$G_\epsilon (x,y)=\E[X_{\partial \D,\epsilon}(x)X_{\partial \D,\epsilon}(y)].$$
For each $\epsilon>0$ small enough, we denote by $D_\epsilon$ the small disk centered at $1-2\epsilon$  with radius $\epsilon$. Notice that for $\epsilon$ small enough, this disk is contained in   $B(1,\delta)\cap\D$. Therefore, we have the obvious relation
\begin{align*}
\int_{\D\cap B(1,\delta)}e^{\frac{\beta_j }{2}\gamma G_\epsilon(\cdot,1)}\epsilon^{\frac{\gamma^2}{2}} e^{\gamma X_{\partial \D,\epsilon}}\,d\lambda\geq &\int_{D_\epsilon}e^{\frac{\beta_j }{2}\gamma G_\epsilon(\cdot,1)} e^{\gamma X_{\partial \D,\epsilon}-\frac{\gamma^2}{2}\E[X_{\partial \D,\epsilon}^2]}e^{\frac{\gamma^2}{2}(\E[X_{\partial \D,\epsilon}^2]-\ln \frac{1}{\epsilon})}\,d\lambda.
\end{align*}
It is then plain to check that, for some constant $C$ independent of $\epsilon$ and uniformly with respect to $x\in D_\epsilon$,
$$ |\E[X_{\partial \D,\epsilon}(x)^2]- 2\ln \frac{1}{\epsilon}| \leq C,\quad |G_\epsilon (x,1)- 2\ln\frac{1}{\epsilon}| \leq C.$$
We deduce
\begin{align*}
\int_{\D\cap B(1,\delta)}e^{\frac{\beta_j }{2}\gamma G_\epsilon(\cdot,1)}\epsilon^{\frac{\gamma^2}{2}} e^{\gamma X_{\partial \D,\epsilon}}\,d\lambda\geq  & C \epsilon^{- \beta_j  \gamma-\frac{\gamma^2}{2}}\int_{D_\epsilon}  e^{\gamma X_{\partial \D,\epsilon}-\frac{\gamma^2}{2}\E[ X_{\partial \D,\epsilon}^2]} \,d\lambda.
\end{align*}
If we can establish the following estimate
\begin{equation}\label{limsup}
\text{in probability},\quad \limsup_{\epsilon\to 0}\epsilon^{-2-\gamma^2}\int_{D_\epsilon}  e^{\gamma X_{\partial \D,\epsilon}-\frac{\gamma^2}{2}\E[ X_{\partial \D,\epsilon}^2]} \,d\lambda=+\infty,
\end{equation}
we deduce that necessarily $\beta_j<  Q$ in order for \eqref{limcond} to hold.

To establish \eqref{limsup}, observe (see Subsection~\ref{estimations}) that, for some deterministic constant $C$ independent of $\epsilon$,
$$\sup_{\epsilon>0}\sup_{x\in D_\epsilon}|G_\epsilon(x,x)+2\ln \epsilon|<+\infty,$$ in such a way that
 \begin{equation}\label{minor}
 \int_{D_\epsilon}  e^{\gamma X_{\partial \D,\epsilon}-\frac{\gamma^2}{2}\E[ X_{\partial \D,\epsilon}^2]} \,d\lambda\geq C \epsilon^{2+\gamma^2} e^{\gamma X_{\partial \D,\epsilon}(1)}e^{\min_{x\in D_\epsilon} X_{\partial \D,\epsilon}(x)- X_{\partial \D,\epsilon}(1)}.
 \end{equation}  
Next, we estimate the $\min$ in the above expression. Observe that ($D(2,1)$ stands for the disk centered at $2$ with radius $1$)
$$\min_{x\in D_\epsilon} X_{\partial \D,\epsilon}(x)- X_{\partial \D,\epsilon}(1)=\min_{u\in D(2,1)} Y_\epsilon(u)$$ where the Gaussian process $Y_\epsilon $ is defined by
$$Y_\epsilon(u)=X_{\partial \D,\epsilon}(1-\epsilon u)-X_{\partial \D,\epsilon}(1).$$
The key point is to estimate the fluctuations of the Gaussian process $Y_\epsilon$. The reader may check (see Subsection~\ref{estimations})  that the variance of $Y_\epsilon(2)$ is bounded independently of $\epsilon$ and that for all $z,z'\in D(2,1)$
$$   \E[(Y_\epsilon(z) -Y_\epsilon(z') )^2]\leq C|z-z'|,
$$
uniformly in $0<\epsilon\leq 1$. Recall the Kolmogorov criterion
\begin{theorem}{\bf (Kolmogorov criterion)} Let $X$ be a continous stochastic process on $D(1,2)$.
If, for some $\beta,\alpha,C>0$: 
$$\forall x,z\in D(1,2),\quad \E[|X_x-X_z|^q]\leq C|x-z|^{2+\beta}.$$
For all $\delta\in ]0,\frac{\beta}{q}[$, we set $L=\sup_{x\not = z}\frac{|X_x-X_z|}{|x-z|^\delta}$. Then, for all $p<q$, $\E[L^\beta]\leq 1+\frac{C p2^{\beta-q\delta}}{(q-p)(2^{\beta-q\delta}-1)}$.
\end{theorem}
One can then deduce that the family of processes $(Y_\epsilon)_\epsilon$ is tight in the space of continuous functions over $D(2,1)$ for the topology of uniform convergence. We deduce that for each subsequence, we can find $R$ large enough such that $\min_{x\in D_\epsilon} X_{\partial \D,\epsilon}(x)- X_{\partial \D,\epsilon}(1)\geq -R$ with probability arbitrarily close to $1$. Finally, we observe that the process $\epsilon\mapsto X_{\partial \D,\epsilon}(1)$ behaves like a Brownian motion at time $2\ln \frac{1}{\epsilon}$ (see \cite[section 6.1]{cf:DuSh}), we can use the law of the iterated logarithm   in \eqref{minor} to complete the proof of \eqref{limsup}.
 
\medskip
Now it remains to show that the condition $\beta_j<Q$ is sufficient to have integrability, that is $\int_{\D\cap B(1,\delta)}e^{\gamma\frac{\beta_j}{2}G(\cdot,1)} e^{\gamma X_{\partial \D }-\frac{\gamma^2}{2}\E[X_{\partial \D }^2]} g_P^{\frac{\gamma^2}{4}}\,d\lambda<+\infty$. To simplify a bit the notations we will prove the following equivalent statement
\begin{equation}
\int_{\H\cap B(0,1)}e^{\gamma\frac{\beta_j}{2}G_\H(z,0)} e^{\gamma X(z) -\frac{\gamma^2}{2}\E[X^2(z)]} \frac{1}{{\rm Im}(z)^{\frac{\gamma^2}{2}}} \, \lambda(dz)<+\infty
\end{equation}
where $X$ is the GFF on $\H$ defined by $X=X_{\partial\D}\circ \psi$ where $\psi(z)=\frac{z-i}{z+i} $ is the Cayley transform 
mapping the upper half-plane onto the unit disk 
 and $G_\H$ its Green function, that is $G_\H(x,y)=G(\psi(x),\psi(y))$. A simple check shows that on the ball $B(0,1)$ we have
 $$G_\H(x,y)=\ln\frac{1}{|x-y||x-\bar{y}|}+g(x,y)$$ where $g$ is a continuous bounded function. It is then also easy to see that for $x,y\in B(0,1)\cap \H$ and all $r\in]0,1[$
\begin{equation}\label{Kah1}
G_\H(r x ,ry)\geq G_\H(x,y)+2\ln\frac{1}{r}-C
\end{equation}
 where $C$ is some fixed positive constant.


The same argument as in Proposition \ref{finitemeas} shows that the quantity
$$\E\Big[\Big(\int_{\H\cap B(0,1)} e^{\gamma X(z) -\frac{\gamma^2}{2}\E[X^2(z)]} \frac{1}{{\rm Im}(z)^{\frac{\gamma^2}{2}}} \, \lambda(dz)\Big)^{\alpha}\Big]$$ is finite for $\alpha<\gamma^{-1}$. Then, for $r<1$, we can make a change of variables $r u=z$ and then combine the relation \eqref{Kah1} with Kahane's convexity inequality  \cite{cf:Kah} (see also \cite[Theorem 2.1]{review})
to deduce  (for some irrelevant constant $C$ which may change along lines)
\begin{align*}
\E\Big[\Big(&\int_{\H\cap B(0,r)} e^{\gamma X(z) -\frac{\gamma^2}{2}\E[X^2(z)]} \frac{1}{{\rm Im}(z)^{\frac{\gamma^2}{2}}} \, \lambda(dz)\Big)^{\alpha}\Big]\\
=&r^{2\alpha}\E\Big[\Big(\int_{\H\cap B(0,1)} e^{\gamma X(r u) -\frac{\gamma^2}{2}\E[X^2(r u)]} \frac{1}{{\rm Im}(ru)^{\frac{\gamma^2}{2}}} \, \lambda(du)\Big)^{\alpha}\Big]\\
=&C r^{(2-\frac{\gamma^2}{2})\alpha}\E\Big[\Big(\int_{\H\cap B(0,1)} e^{\gamma X(z) -\frac{\gamma^2}{2}\E[X^2(z)]} \frac{1}{{\rm Im}(z)^{\frac{\gamma^2}{2}}} \, \lambda(dz)\Big)^{\alpha}\Big]\E\big[e^{\alpha \gamma Z_{r}-\frac{\alpha^2\gamma^2}{2}2\ln\frac{1}{r}}\big]
\end{align*}
where $Z_r$ is a centered Gaussian random variable with variance $2\ln\frac{1}{r}$ and independent of everything.
Hence, for all $r<1$
\begin{align*}
\E\Big[\Big( \int_{\H\cap B(0,r)} e^{\gamma X(z) -\frac{\gamma^2}{2}\E[X^2(z)]} \frac{1}{{\rm Im}(z)^{\frac{\gamma^2}{2}}} \, \lambda(dz)\Big)^{\alpha}\Big] \leq & C r^{(2+\frac{\gamma^2}{2})\alpha- \alpha^2\gamma^2}.
\end{align*}
Let $\eta>0$. By using the Markov inequality and the above relation, we obtain
\begin{align*}
\P\big(\int_{\H\cap B(0,r)} &e^{\gamma X(z) -\frac{\gamma^2}{2}\E[X^2(z)]} \frac{1}{{\rm Im}(z)^{\frac{\gamma^2}{2}}} \, \lambda(dz)>r^{2+\frac{\gamma^2}{2}-\eta}\big)
\\
\leq & r^{-\alpha(2+\frac{\gamma^2}{2}-\eta)}\E\Big[\Big( \int_{\H\cap B(0,r)} e^{\gamma X(z) -\frac{\gamma^2}{2}\E[X^2(z)]} \frac{1}{{\rm Im}(z)^{\frac{\gamma^2}{2}}} \, \lambda(dz)\Big)^{\alpha}\Big]
\\
\leq &C r^{\eta\alpha-\alpha^2\gamma^2}.
\end{align*}
Choosing $\alpha>0$ small enough so as to get $\eta\alpha-\alpha^2\gamma^2>0$. We can then use the Borel-Cantelli lemma to deduce that there exists a random constant $R$, which is finite almost surely, such that
\begin{equation}\label{Kah2}
\sup_{r\in]0,1]}r^{-(2+\frac{\gamma^2}{2}-\eta)}\int_{\H\cap B(0,r)}  e^{\gamma X(z) -\frac{\gamma^2}{2}\E[X^2(z)]} \frac{1}{{\rm Im}(z)^{\frac{\gamma^2}{2}}} \, \lambda(dz)\leq R.
\end{equation}
Now we introduce the annuli for $n\geq 0$ 
$$A_{n}=\{z\in \H;    2^{-n-1}\,\leq  |z|\leq 2^{-n}\}.$$
We get from \eqref{Kah2}
\begin{align*}
\int_{\H\cap B(0,1)}&e^{\gamma\frac{\beta_j}{2}G_\H(z,0)} e^{\gamma X(z) -\frac{\gamma^2}{2}\E[X^2(z)]} \frac{1}{{\rm Im}(z)^{\frac{\gamma^2}{2}}} \, \lambda(dz)
\\
=&\sum_{n\geq 0}\int_{A_n}e^{\gamma\frac{\beta_j}{2}G_\H(z,0)} e^{\gamma X(z) -\frac{\gamma^2}{2}\E[X^2(z)]} \frac{1}{{\rm Im}(z)^{\frac{\gamma^2}{2}}} \, \lambda(dz)
\\
\leq & C\sum_{n\geq 0} 2^{\gamma \beta_j n} \int_{\H\cap B(0,2^{-n})}e^{\gamma X(z) -\frac{\gamma^2}{2}\E[X^2(z)]} \frac{1}{{\rm Im}(z)^{\frac{\gamma^2}{2}}} \, \lambda(dz)\\
\leq & CR\sum_{n\geq 0} 2^{\gamma \beta_j n}  2^{-n(2+\frac{\gamma^2}{2}-\eta)}.
\end{align*}
The proof of Theorem \ref{th:seiberg} is complete provided that we choose $0<\eta<\gamma (Q-\beta_j)$. Once the Seiberg bounds are established, the computation of the partition function (i.e. Proposition \ref{prop:part}) follows the same lines as in \cite[Theorem 3.2]{DKRV}.\qed

\subsection{Definitions of the Liouville field, Liouville measure and boundary Liouville measure} 

As long as one of the two conditions of Theorem \ref{th:seiberg} is satisfied,  one may define the joint law of the Liouville field $\phi$ together with the Liouville measure $Z(\cdot )$ and boundary Liouville measure $Z_{\partial}(\cdot )$. In spirit, the situation is that the convergence of the partition function entails that we get a non trivial probability law for the field $\phi =c+X_{\partial \D}+\frac{Q}{2}\ln g$ under the probability measure defined by the partition function. This field formally corresponds to the log-conformal factor of some random metric $e^{\gamma\phi }g$ conformally equivalent to $g$. Yet, observe that the field $\phi$ is in $H^{-1}$ almost surely so that a rigorous description of this metric is not straightforward, at least clearly not standard. The Liouville measure that we construct below is a random measure that can be thought of as the volume form of this formal metric tensor whereas the boundary Liouville measure corresponds to the line element along the boundary. Let us mention that we could construct as well the Liouville Brownian motion by using the construction made in \cite{GRV,GRV-FD} but a rigorous construction of a distance function associated to the metric tensor $e^{\gamma\phi }g$ remains an open question.

\medskip
Given a measured space $E$, we denote by $\mathcal{R}(E)$ the space of Radon measures on $E$ equipped with the topology of weak convergence. The joint law of $(\phi,Z,Z_\partial)$ is defined for all continuous bounded functional $F$ on $H^{-1}(\bar{\D})\times \mathcal{R}(\D)\times \mathcal{R}(\partial\D)$ by
 \begin{align*}
 &  \E_{\gamma,\mu_{\partial  },\mu,g}^{(z_i,\alpha_i)_i,(s_j,\beta_j)_j} \big[F(\phi,Z,Z_{\partial})\big]  \\
=& \frac{e^{  \frac{1}{96\pi}\Big(\int_{\D}|\partial \ln g|^2\,d\lambda+\int_{\partial \D}4\ln g\,d\lambda_{\partial} \Big)} }{\Pi_{\gamma,\mu_{\partial  },\mu}^{(z_i,\alpha_i)_i,(s_j,\beta_j)_j}(g,1)}  \lim_{\epsilon\to 0}\int_\R  \prod_i \epsilon^{\frac{\alpha_i^2}{2}}e^{\alpha_i  (X_{\partial \D,\epsilon}+Q/2\ln g)(z_i)}\prod_j \epsilon^{\frac{\beta_j^2}{4}}e^{\frac{\beta_j }{2} (X_{\partial \D,\epsilon} +Q/2\ln g)(s_j)} \nonumber\\\ &E\Big[F\Big( X_{\partial \D}+c+ Q/2\ln g ,e^{\gamma c}\epsilon^{\frac{\gamma^2}{2}} e^{\gamma (X_{\partial \D,\epsilon} +Q/2\ln g) } \,d\lambda, e^{\frac{\gamma}{2} c}\epsilon^{\frac{\gamma^2}{4}} e^{\frac{\gamma}{2}( X_{\partial \D,\epsilon}  +Q/2\ln g)} \,d\lambda_\partial  \Big)\nonumber \\
&\exp\Big( - \frac{Q}{4\pi}\int_{\D}R_{g} (c+X_{\partial  \D})  \,d\lambda_g - \mu e^{\gamma c}\epsilon^{\frac{\gamma^2}{2}}\int_{\D}e^{\gamma (X_{\partial \D,\epsilon}+ Q/2\ln g) }\,d\lambda \Big)\nonumber \\
&\exp\Big( - \frac{Q}{2\pi}\int_{\partial \D}K_{g} (c+X_{\partial\D})  \,d\lambda_{\partial g} - \mu_\partial e^{\frac{\gamma}{2} c}\epsilon^{\frac{\gamma^2}{4}}\int_{\partial\D}e^{\frac{\gamma}{2} (X_{\partial\D,\epsilon}+ Q/2\ln g) }\,d\lambda_{\partial } \Big) \Big]\,dc\nonumber .%
\end{align*}
 We denote by $\P_{\gamma,\mu_{\partial  },\mu,g}^{(z_i,\alpha_i)_i,(s_j,\beta_j)_j} $ the associated probability measure. In the following subsections, we will mention several interesting properties satisfied by these objects.

\subsection{Conformal changes of metric and Weyl anomaly}\label{sec:weyl}
Here we want to determine the dependence of the partition function \eqref{eq:defPig}  (as well as the Liouville field/measures) on the metric $g$  conformally equivalent to the Euclidean metric. In fact, this dependence enables to determine the central charge of the theory: 
\begin{theorem}{\bf (Weyl anomaly)}\label{th:CI}
\begin{enumerate}
\item 
Given two metrics $g,g'$ conformally equivalent to the flat metric and $g'=e^{\varphi}g$, we have
$$\ln\frac{\Pi_{\gamma,\mu_{\partial},\mu}^{(z_i,\alpha_i)_i,(s_j,\beta_j)_j}  (g',F)}{  \Pi_{\gamma,\mu_{\partial},\mu}^{(z_i,\alpha_i)_i,(s_j,\beta_j)_j}  (g,F)}= \frac{1+6Q^2}{96\pi}\Big(\int_{\D}|\partial^g \varphi|^2\,d\lambda_g+\int_\D 2R_g\varphi\,d\lambda_g+4\int_{\partial \D} K_g \varphi\,d\lambda_{\partial_g}\Big).$$
\item The law of the triple $(\phi,Z,Z_{\partial})$ under  $\P_{\gamma,\mu_{\partial  },\mu,g}^{(z_i,\alpha_i)_i,(s_j,\beta_j)_j} $ does not depend on the metric $g$ in the conformal equivalence class of the Euclidean metric.
\end{enumerate}
\end{theorem}

In the language of CFT, the above theorem states that the central charge of LQFT is $1+6 Q^2$: see the lecture notes \cite{gaw}.

\noindent {\it Proof.}   In \eqref{eq:defPig}, we use the Girsanov transform to the exponential term
$$\exp\Big( - \frac{Q}{4\pi}\int_{\D}R_{g}  X_{\partial \D } \,d\lambda_g   - \frac{Q}{2\pi}\int_{\partial \D}K_{g}  X_{\partial \D }  \,d\lambda_{\partial g}  \Big),$$
which has the effect of shifting the field $X$ by
$$-  \frac{Q}{4\pi}\int_{\D}R_{g}  G_{\partial \D}(\cdot,z) \,\lambda_g(dz)   - \frac{Q}{2\pi}\int_{\partial \D}   G_{\partial \D}(\cdot,z) K_g\,d\lambda_{\partial g}. $$
Then we use the rules \eqref{curv}+\eqref{geod}+\eqref{relgreen} to see that this shift is equal to
$$-  \frac{Q}{2}(\ln g-m_{\partial}(\ln g)).$$
Due to the Girsanov renormalization, the whole partition function will be multiplied by the exponential of the variance of the field $\frac{Q}{4\pi}\int_{\D}R_{g}  X_{\partial \D } \,d\lambda_g  + \frac{Q}{2\pi}\int_{\partial \D}K_{g}  X_{\partial \D }  \,d\lambda_{\partial g} $, which can be computed with \eqref{curv}+\eqref{geod}+\eqref{relgreen} and is given by
$$\frac{Q^2}{16\pi}\int_\D|\partial \ln g|^2\,d\lambda.$$
Hence, by making the changes of variables $v=c+\frac{Q}{2}m_{\partial}(\ln g)$, we get
\begin{align*} 
\Pi_{\gamma,\mu_{\partial  },\mu}^{(z_i,\alpha_i)_i,(s_j,\beta_j)_j}&(g,F)\\
 =&e^{\frac{6 Q^2}{96\pi}\int_\D|\partial \ln g|^2\,d\lambda+\frac{Q^2}{2}m_{\partial}(\ln g)} e^{\frac{1 }{96\pi}\big(\int_{\R^2}|\partial \ln g|^2\,d\lambda+\int_{\partial}4\ln g\,d\lambda_{\partial}\big)}\\
 & \lim_{\epsilon\to 0}\int_\R e^{\big(\sum_i\alpha_i+\frac{1}{2}\sum_j\beta_j-Q\big)v} \E\Big[F( X_{\partial \D}+v,e^{\gamma v}e^{\gamma X_{\partial \D}}\,d\lambda,e^{\frac{\gamma}{2} v}e^{\frac{\gamma}{2} X_{\partial \D}}\,d\lambda_\partial )\nonumber\\
 &\prod_i \epsilon^{\frac{\alpha_i^2}{2}}e^{\alpha_i   X_{\partial \D,\epsilon} (z_i)} \prod_j \epsilon^{\frac{\beta_j^2}{4}}e^{\frac{\beta_j }{2} X_{\partial \D,\epsilon} (s_j)}\\
& \exp\Big(  - \mu e^{\gamma v}\epsilon^{\frac{\gamma^2}{2}}\int_{\D}e^{\gamma  X_{\partial \D,\epsilon} }\,d\lambda    - \mu_{\partial} e^{\frac{\gamma}{2}v}\epsilon^{\frac{\gamma^2}{4}}\int_{\partial\D}e^{\frac{\gamma}{2} X_{\partial \D,\epsilon} }\,d\lambda_{\partial  } \Big) \Big]\,dc\nonumber \\
=&e^{\frac{1+6Q^2}{96\pi}\Big(\int_{\D}|\partial \ln g|^2\,d\lambda+\int_{\partial  \D}4\ln g\,d\lambda_{\partial}\Big)}  \Pi_{\gamma,\mu_{\partial  },\mu}^{(z_i,\alpha_i)_i,(s_j,\beta_j)_j}(dx^2,F)\nonumber.
 \end{align*}   
To complete the proof for two metrics $g,g'$ conformally equivalent to the Euclidean metric, say $g'=e^\varphi g$, we apply twice the above result to get
\begin{align*} 
\ln&\frac{\Pi_{\gamma,\mu_{\partial  },\mu}^{(z_i,\alpha_i)_i,(s_j,\beta_j)_j} (g',F)}{\Pi_{\gamma,\mu_{\partial  },\mu}^{(z_i,\alpha_i)_i,(s_j,\beta_j)_j} (g,F)}\nonumber\\
=& \frac{1+6Q^2}{96\pi}\Big(\int_{\D}|\partial \ln g'|^2\,d\lambda+\int_{\partial  \D}4\ln g'\,d\lambda_{\partial}-\int_{\D}|\partial \ln g|^2\,d\lambda-\int_{\partial  \D}4\ln g\,d\lambda_{\partial}\Big) \nonumber\\
=& \frac{1+6Q^2}{96\pi}\Big(\int_{\D}|\partial \varphi|^2\,d\lambda+2\int_{\D}\partial \varphi\cdot \partial \ln g\,d\lambda+\int_{\partial  \D}4\varphi\,d\lambda_{\partial} \Big).\nonumber
 \end{align*}  
Now we use \eqref{GR}+\eqref{curv} to get
\begin{align*} 
\ln&\frac{\Pi_{\gamma,\mu_{\partial  },\mu}^{(z_i,\alpha_i)_i,(s_j,\beta_j)_j} (g',F)}{\Pi_{\gamma,\mu_{\partial  },\mu}^{(z_i,\alpha_i)_i,(s_j,\beta_j)_j} (g,F)}= \frac{1+6Q^2}{96\pi}\Big(\int_{\D}|\partial \varphi|^2\,d\lambda+2\int_{\D} \varphi R_g\,d\lambda_g+4\int_{\partial  \D} \varphi(1+\frac{1}{2}\partial_n\ln g)\,d\lambda_{\partial} \Big). 
 \end{align*}
We complete the proof with \eqref{geod}.\qed

\subsection{Conformal covariance and KPZ formula} 

 Now we want to establish the conformal covariance of the partition function, i.e. to determine its behavior under the action  of M\"obius transforms on the marked points. We focus here on the case when the background metric is the Euclidean one: as shown by the Weyl anomaly (Theorem \ref{th:CI}), this is not a restriction. One thus looks at
\begin{align}\label{flatmobius}
&\Pi_{\gamma,\mu_{\partial},\mu}^{(\psi(z_i),\alpha_i)_i,(\psi(s_j),\beta_j)_j}(dx^2,F)\\
=&\lim_{\epsilon\to 0}\int_\R e^{\big(\sum_i\alpha_i+\frac{1}{2}\sum_j\beta_j-Q\big)c} \E\Big[F( X_{\partial\D,\epsilon}+c)\prod_i\epsilon^{\frac{\alpha_i^2}{2}}e^{\alpha_i X_{\partial \D,\epsilon}(\psi(z_i))} \prod_j \epsilon^{\frac{\beta_j^2}{4}}e^{\frac{\beta_j }{2} X_{\partial \D,\epsilon} (\psi(s_j))} \nonumber\\
 &\exp\Big(-\mu e^{\gamma c}\epsilon^{\frac{\gamma^2}{2}}\int_{\D}e^{\gamma  X_{\partial \D,\epsilon} }\,d\lambda    - \mu_{\partial} e^{\frac{\gamma}{2} c}\epsilon^{\frac{\gamma^2}{4}}\int_{\partial\D}e^{\frac{\gamma}{2} X_{\partial \D,\epsilon} }\,d\lambda_{\partial } \Big) \Big]\,dc.\nonumber
\end{align}
where $\psi$ is a M\"obius transform of the unit disk. 

We use the following convention for the rest of this section. If $M$ is a measure on a measurable space $E$ and $\psi:E\to E$ is a bi-measurable bijection then the measure $M\circ \psi $ is defined by the relation $M\circ \psi(A)=M(\psi (A))$ for all measurable set $A\subset E$.
\begin{theorem}
Let  $\psi$ be a M\"obius transform  of the disk. Then
$$\Pi_{\gamma,\mu_\partial,\mu}^{(\psi(z_i),\alpha_i)_i,(\psi(s_j),\beta_j)_j} (dx^2,1)=\prod_i|\psi'(z_i)|^{-2\Delta_{\alpha_i}}\prod_j|\psi'(s_j)|^{-\Delta_{\beta_j}}\Pi_{\gamma,\mu_\partial,\mu}^{(z_i,\alpha_i)_i,(s_j,\beta_j)_j} (dx^2,1)$$
where the conformal weights $\triangle_\alpha$ are defined by $$\triangle_\alpha= \frac{\alpha}{2}(Q-\frac{\alpha}{2}).$$
Furthermore the law of the triple  $(\phi,Z,Z_{\partial})$ under $\P_{\gamma,\mu_{\partial  },\mu,dx^2}^{(z_i,\alpha_i)_i,(s_j,\beta_j)_j}$  is the same as that of the triple  $\big(\phi\circ \psi+Q\ln|\psi'|,Z\circ \psi , Z_\partial\circ \psi \big)$ under $\P_{\gamma,\mu_{\partial  },\mu,dx^2}^{(\psi(z_i),\alpha_i)_i,(\psi(s_j),\beta_j)_j} $.
\end{theorem}
\begin{proof} To facilitate the comprehension, we take only into consideration the law of the Liouville field and we leave to the reader the details of the whole proof for the triple $(\phi,Z,Z_{\partial})$.\\
We first study the behavior of the measure under the M\"obius transform $\phi$:
\begin{lemma}\label{mobiusmeasure}
For any $f\in C^2(\overline{\mathds{D}})$, we have
\begin{align*}
&(X_{\partial\mathds{D}}\circ\psi,\lim\limits_{\epsilon\to 0}\int_{\mathds{D}}f \epsilon^{\frac{\gamma^2}{2}}e^{\gamma X_{\partial\mathds{D},\epsilon}}d\lambda,\lim\limits_{\epsilon\to 0}\int_{\mathds{D}}f \epsilon^{\frac{\gamma^2}{4}}e^{\frac{\gamma}{2}X_{\partial\mathds{D},\epsilon}}d\lambda_{\partial})\\
\overset{law}{=}&(X_{\partial\mathds{D}}+m_\partial(X_{\partial\mathds{D}}\circ\psi),\lim\limits_{\epsilon\to 0}\int_{\mathds{D}}f\circ\psi e^{\gamma(X_{\partial\mathds{D},\epsilon}+m_\partial(X_{\partial\mathds{D}}\circ\psi))}|\psi'|^{Q\gamma}d\lambda,\lim\limits_{\epsilon\to 0}\int_{\mathds{D}}f\circ\psi e^{\frac{\gamma}{2}(X_{\partial\mathds{D},\epsilon}+m_\partial(X_{\partial\mathds{D}}\circ\psi))}|\psi'|^{\frac{Q\gamma}{2}}d\lambda_\partial)
\end{align*}
\end{lemma}
\noindent {\it Proof of Lemma \ref{mobiusmeasure}.} Using Proposition \ref{circlegreen}, we have that
\begin{align*}
\lim\limits_{\epsilon\to 0}\mathds{E}[X_{\partial\mathds{D},\epsilon}(\psi(x))^2]-\mathds{E}[(X_{\partial\mathds{D}}\circ\psi)_{\frac{\epsilon}{|\phi'(x)|}}(x)^2]=0
\end{align*}
on $\mathds{D}$ and on $\partial\mathds{D}$.\\
As $|\phi'(x)|$ is always larger than a constant that is strictly positive, we can use the result in \cite{shamov} to show that the measures
$$(\frac{\epsilon}{|\phi'|})^{\frac{\gamma^2}{2}}e^{\gamma X_{\partial\mathds{D},\epsilon}\circ\psi}d\lambda$$
and
$$\epsilon^{\frac{\gamma^2}{2}}e^{\gamma(X_{\partial\mathds{D}}\circ\psi)_{\epsilon}}d\lambda$$
converge in probability to the same random measure on $\mathds{D}$.\\
Similarly,
$$(\frac{\epsilon}{|\phi'|})^{\frac{\gamma^2}{4}}e^{\frac{\gamma}{2} X_{\partial\mathds{D},\epsilon}\circ\psi}d\lambda_{\partial}$$
and
$$\epsilon^{\frac{\gamma^2}{4}}e^{\frac{\gamma}{2}(X_{\partial\mathds{D}}\circ\psi)_{\epsilon}}d\lambda_{\partial}$$
converge in probability to the same limit measure on $\partial\mathds{D}$.\\
We also have, by change of variables in the integrand
$$\int_{\mathds{D}}f~\epsilon^{\frac{\gamma^2}{2}}e^{\gamma X_{\partial\mathds{D},\epsilon}}d\lambda=\int_{\mathds{D}}f\circ\psi ~\epsilon^{\frac{\gamma^2}{2}}e^{\gamma X_{\partial\mathds{D},\epsilon}\circ\psi}|\psi'|^2 d\lambda=\int_{\mathds{D}}f\circ\psi ~(\frac{\epsilon}{|\psi'|})^{\frac{\gamma^2}{2}}e^{\gamma X_{\partial\mathds{D},\epsilon}\circ\psi}|\psi'|^{Q\gamma}d\lambda$$
and similarly
$$\int_{\mathds{D}}f~\epsilon^{\frac{\gamma^2}{4}}e^{\frac{\gamma}{2} X_{\partial\mathds{D},\epsilon}}d\lambda_\partial=\int_{\mathds{D}}f\circ\psi ~\epsilon^{\frac{\gamma^2}{4}}e^{\frac{\gamma}{2} X_{\partial\mathds{D},\epsilon}\circ\psi}|\psi'| d\lambda_\partial=\int_{\mathds{D}}f\circ\psi ~(\frac{\epsilon}{|\psi'|})^{\frac{\gamma^2}{4}}e^{\frac{\gamma}{2} X_{\partial\mathds{D},\epsilon}\circ\psi}|\psi'|^{\frac{Q\gamma}{2}}d\lambda_\partial$$
Combining the above arguments, we conclude the proof by recalling the change of metric formula
\begin{equation}\label{eq:changeofmetric}
X_{\partial\D} \circ \psi - m_\partial (X_{\partial\D} \circ \psi) \overset{law}{=} X_{\partial\D},
\end{equation}
which can be verified using the definition of $m_\partial$ and the Green function.\qed 

\medskip
Anticipating the formula (\ref{eq:changeofmetric}), we use the change of variables $\overline{c}=c+m_\partial(X_\partial \circ \psi)$ to write
\begin{align*}
&\Pi_{\gamma,\mu_{\partial},\mu}^{(\psi(z_i),\alpha_i)_i,(\psi(s_j),\beta_j)_j}(dx^2,F)\\
=&\lim_{\epsilon\to 0}\int_\R e^{\big(\sum_i\alpha_i+\frac{1}{2}\sum_j\beta_j-Q\big)(\overline{c}-m_\partial(X_{\partial\D}\circ\psi))} \E\Big[F( X_{\partial\D,\epsilon}+\overline{c}-m_\partial(X_{\partial\D}\circ\psi))\prod_i\epsilon^{\frac{\alpha_i^2}{2}}e^{\alpha_i X_{\partial \D,\epsilon}(\psi(z_i))} \\
 &\prod_j \epsilon^{\frac{\beta_j^2}{4}}e^{\frac{\beta_j }{2} X_{\partial \D,\epsilon} (\psi(s_j))}\exp\Big(-\mu e^{\gamma \overline{c}}\epsilon^{\frac{\gamma^2}{2}}\int_{\D}e^{\gamma  (X_{\partial \D,\epsilon} - m_\partial(X_{\partial\D}\circ\psi))}\,d\lambda    - \mu_{\partial} e^{\frac{\gamma}{2} \overline{c}}\epsilon^{\frac{\gamma^2}{4}}\int_{\D}e^{\frac{\gamma}{2} (X_{\partial \D,\epsilon} - m_\partial(X_{\partial\D}\circ\psi)) }\,d\lambda_{\partial g} \Big) \Big]\,dc.\nonumber
\end{align*}
We now apply the Girsanov transform to the factor $e^{Q m_\partial(X_{\partial\D}\circ\psi)}$. This will shift the law of the field $X_{\partial\D}$, which becomes
$$X_{\partial\D}+\frac{Q}{2\pi}\int_{\partial\D}G (\cdot,\psi(z))\lambda_\partial(dz)$$
We now introduce a useful constant in the following calculation  
$$D_\psi=\int_{\partial\D}\int_{\partial\D} G (\psi(y),\psi(z))\lambda_\partial(dy)\lambda_\partial(dz)=4\pi^2 \mathds{E}[m_\partial(X_{\partial\D}\circ\psi)^2].$$
We also introduce the function
$$H(y)=\int_{\partial\D} G(\psi(y),\psi(z))\lambda_\partial(dz)$$
so that $D_\psi=\int_{\partial\D} H(y)\lambda_\partial(dy)$. Recall that $\int_{\partial\D} G(y,z)\lambda_\partial(dz)=0$ for all $y$.

Under the Girsanov transform $X_{\partial\D}(x)-m_\partial(X_{\partial\D}\circ\psi)$ becomes $X_{\partial\D}(x)-m_\partial(X_{\partial\D}\circ\psi)+\frac{Q}{2\pi}H(\psi^{-1}(x))-\frac{Q}{4\pi^2}D_\psi$ and we get
\begin{align*}
&\Pi_{\gamma,\mu_{\partial},\mu}^{(\psi(z_i),\alpha_i)_i,(\psi(s_j),\beta_j)_j}(dx^2,F)\\
=&\lim_{\epsilon\to 0}e^{\frac{Q^2}{8\pi}D_\psi}\int_\R e^{\big(\sum_i\alpha_i+\frac{1}{2}\sum_j\beta_j-Q\big)\overline{c}} \E\Big[F( X_{\partial\D,\epsilon}+\overline{c}-m_\partial(X_{\partial\D}\circ\psi)+\frac{Q}{2\pi}H(\psi^{-1}(\cdot))-\frac{Q}{4\pi^2}D_\psi) \nonumber\\
 &\prod_i\epsilon^{\frac{\alpha_i^2}{2}}e^{\alpha_i(X_{\partial \D,\epsilon}(\psi(z_i))-m_\partial(X_{\partial\D}\circ\psi)+\frac{Q}{2\pi}H(z_i)-\frac{Q}{4\pi^2}D_\psi)} \prod_j \epsilon^{\frac{\beta_j^2}{4}}e^{\frac{\beta_j }{2}(X_{\partial \D,\epsilon}(\psi(s_j))-m_\partial(X_{\partial\D}\circ\psi)+\frac{Q}{2\pi}H(s_j)-\frac{Q}{4\pi^2}D_\psi)}\\
 &\exp\Big(-\mu e^{\gamma \overline{c}}\epsilon^{\frac{\gamma^2}{2}}\int_{\D}e^{\gamma  (X_{\partial\D}(x)-m_\partial(X_{\partial\D}\circ\psi)+\frac{Q}{2\pi}H(\psi^{-1}(x))-\frac{Q}{4\pi^2}D_\psi)}\,d\lambda   \\
 & - \mu_{\partial} e^{\frac{\gamma}{2} \overline{c}}\epsilon^{\frac{\gamma^2}{4}}\int_{\D}e^{\frac{\gamma}{2} (X_{\partial\D}(x)-m_\partial(X_{\partial\D}\circ\psi)+\frac{Q}{2\pi}H(\psi^{-1}(x))-\frac{Q}{4\pi^2}D_\psi) }\,d\lambda_{\partial g} \Big) \Big]\,dc.\nonumber
\end{align*}
Notice the relation (consequence of (\ref{greenpsi}))
$$\frac{Q}{2\pi}H(x)=Q\ln{\frac{1}{|\psi'(x)|}}+\frac{Q}{8\pi^2}D_\psi$$
the $D_\psi$ part with cancel out the first exponential term in the above expression when we do the change of variables $c=\overline{c}-\frac{Q}{8\pi^2}D_\psi$.\\
Now using \eqref{eq:changeofmetric}, \eqref{law} and Lemma \ref{mobiusmeasure}, we finally have
\begin{align*}
&\Pi_{\gamma,\mu_{\partial},\mu}^{(\psi(z_i),\alpha_i)_i,(\psi(s_j),\beta_j)_j}(dx^2,F)\\
=&\prod_i|\psi'(z_i)|^{\alpha^2/2-\alpha Q}\prod_j|\psi'(s_j)|^{\beta^2/4-\beta Q/2} \nonumber\\
 &\lim_{\epsilon\to 0}\int_\R e^{\big(\sum_i\alpha_i+\frac{1}{2}\sum_j\beta_j-Q\big)c} \E\Big[F( X_{\partial\D,\epsilon}\circ\psi^{-1}-Q\ln|\psi'(\psi^{-1}(\cdot))|+c)\prod_i\epsilon^{\frac{\alpha_i^2}{2}}e^{\alpha_i(X_{\partial \D,\epsilon}(z_i))} 
 \\
&\prod_j \epsilon^{\frac{\beta_j^2}{4}}e^{\frac{\beta_j }{2}(X_{\partial \D,\epsilon}(s_j))}
\exp\Big(-\mu e^{\gamma c}\epsilon^{\frac{\gamma^2}{2}}\int_{\D}e^{\gamma  X_{\partial \D,\epsilon}}\,d\lambda    - \mu_{\partial} e^{\frac{\gamma}{2} c}\epsilon^{\frac{\gamma^2}{4}}\int_{\D}e^{\frac{\gamma}{2} X_{\partial \D,\epsilon} }\,d\lambda_{\partial g} \Big) \Big]\,dc.\nonumber
\end{align*}
This completes the proof of the theorem.
\end{proof}

\subsection{Conformal changes of domains}
In this section, we explain how to construct the LQFT on domains that are conformally equivalent to the unit disk. Basically, the idea is to find a conformal map sending this domain to the unit disk and to use the conformal covariance property of the LQFT.

Let $D$ be a simply connected (strict) domain of $\C$, say with a $C^1$ Jordan boundary. From the Riemann mapping theorem, we can consider a conformal map $\psi:D\to \D$. If we further consider marked points $(z_i,\alpha_i)$ in $D$ and boundary marked points $(s_j,\beta_j)_j$ in $\partial D$, they will be sent respectively to $(\psi(z_i),\alpha_i)$ in $\D$ and to the boundary marked points $(\psi(s_j),\beta_j)_j$ in $\partial \D$.   Finally, the uniformization theorem tells us that there is no restriction if we assume that $D$ is equipped with a metric of the type $g_\psi=|\psi'|^2g(\psi)$ for some metric $g$ on $\D$. 

The Liouville partition function on $(D,g_\psi)$ applied to a functional $F$ reads
\begin{align}\label{eq:defD0}
&\Pi_{ \gamma,\mu_{\partial  },\mu}^{(z_i,\alpha_i)_i,(s_j,\beta_j)_j}  (D,g_\psi,F)\\
 =&e^{  \frac{1}{96\pi}\Big(\int_{\D}|\partial \ln g|^2\,d\lambda+\int_{\partial \D} 4 \ln g\,d\lambda_{\partial} \Big)} \lim_{\epsilon\to 0}\int_\R\E\Big[F( X_{\nu }+c+Q/2\ln g_\psi)\prod_i \epsilon^{\frac{\alpha_i^2}{2}}e^{\alpha_i (c+X_{\nu,\epsilon}+Q/2\ln g_\psi)(z_i)}\nonumber\\
&\prod_j \epsilon^{\frac{\beta_j^2}{4}}e^{\frac{\beta_j }{2}(c+X_{\nu,\epsilon}+Q/2\ln g_\psi)(s_j)}\exp\Big( - \frac{Q}{4\pi}\int_{D}R_{g_\psi} (c+X_{\nu })  \,d\lambda_{g_\psi} - \mu e^{\gamma c}\epsilon^{\frac{\gamma^2}{2}}\int_{D}e^{\gamma (X_{\nu,\epsilon}+ Q/2\ln g_\psi) }\,d\lambda \Big)\nonumber \\
&\exp\Big( - \frac{Q}{2\pi}\int_{\partial  D}K_{\partial g_\psi}  (c+X_{\nu})  \,d\lambda_{\partial g_\psi} - \mu_\partial \,e^{\frac{\gamma}{2} c}\epsilon^{\frac{\gamma^2}{4}}\int_{\partial D}e^{\frac{\gamma}{2} (X_{\nu,\epsilon}+ Q/2\ln g_\psi) }\,d\lambda_{\partial } \Big) \Big]\,dc,\nonumber %
\end{align}
where $X_\nu$ is a GFF on $D$ with Neumann boundary condition and vanishing $\nu$-mean. By shift invariance of the Lebesgue measure, the choice of $\nu$ is irrelevant  and it will be convenient to take $X_\nu= X_{\partial \D}\circ \psi$, which is free boundary GFF with vanishing mean for the line element on $\partial D$ in the metric $|\psi'|^2dx^2$ on $D$.  

\begin{proposition}\label{CCD}
Let $D$ be a simply connected (strict) domain of $\C$ with a  $C^1$ Jordan boundary. Then we have the relation
\begin{align*}
\Pi_{ \gamma,\mu_{\partial  },\mu}^{(z_i,\alpha_i)_i,(s_j,\beta_j)_j} & (D,g_\psi,F(\phi,Z,Z_{\partial}))\\=&\prod_i|\psi'(z_i)|^{2\triangle_{\alpha_i}} \prod_j|\psi'(s_j)|^{\triangle_{\beta_j}} \Pi_{ \gamma,\mu_{\partial  },\mu}^{(\psi(z_i),\alpha_i)_i,(\psi(s_j),\beta_j)_j}  (\D,g,F(\phi\circ \psi+Q\ln |\psi'|,Z\circ \psi,Z_{\partial}\circ \psi) ).
\end{align*}
In particular:
\begin{enumerate}
\item we have the following relation between the partition functions ($F=1$)
$$\Pi_{ \gamma,\mu_{\partial  },\mu}^{(z_i,\alpha_i)_i,(s_j,\beta_j)_j}  (D,g_\psi,1)=\prod_i|\psi'(z_i)|^{2\triangle_{\alpha_i}} \prod_j|\psi'(s_j)|^{\triangle_{\beta_j}} \Pi_{ \gamma,\mu_{\partial  },\mu}^{(\psi(z_i),\alpha_i)_i,(\psi(s_j),\beta_j)_j}  (\D,g,1).$$
\item The law of the triple $(\phi,Z,Z_{\partial})$ under $\P_{\gamma,\mu_{\partial},\mu,(D,g_\psi)}^{(z_i,\alpha_i)_i,(s_j,\beta_j)_j} $ is the same as $(\phi\circ \psi+Q\ln|\psi'|,Z,Z_{\partial})$ under $\P_{\gamma,\mu_{\partial},\mu,(\D,g)}^{(\psi(z_i),\alpha_i)_i,(\psi(s_j),\beta_j)_j}$.
\end{enumerate}
\end{proposition}

\noindent {\it Proof.} Again we only treat a functional $F$ depending only on the Liouville field for simplicity. Applying Lemma~\ref{mobiusmeasure} and using that $R_{g_\psi}(x)=R_g(\psi(x))$ and $K_{g_\psi}(x)= K_{g}(\psi(x))$ (because $\psi$ is a conformal map), \eqref{eq:defD0} is equal to
\begin{align}\label{eq:defD2}
&\Pi_{ \gamma,\mu_{\partial  },\mu}^{(z_i,\alpha_i)_i,(s_j,\beta_j)_j}  (D,g_\psi,F)\nonumber\\
 =&e^{  \frac{1}{96\pi}\Big(\int_{\D}|\partial \ln g|^2\,d\lambda+\int_{\partial \D}4\ln g\,d\lambda_{\partial} \Big)} \prod_i|\psi'(z_i)|^{Q\alpha_i-\frac{\alpha_i^2}{2}} \prod_j|\psi'(s_j)|^{Q\frac{\beta_j}{2}-\frac{\beta_j^2}{4}}\nonumber\\
 & \lim_{\epsilon\to 0}\int_\R\E\Big[F( (X_{\partial \D }+c+Q/2\ln g)\circ \psi+ Q\ln |\psi'|) \nonumber\\
 &\prod_i \epsilon^{\frac{\alpha_i^2}{2}}e^{\alpha_i (c+X_{\partial \D,\epsilon}+Q/2\ln g)(\psi(z_i))}  \prod_j \epsilon^{\frac{\beta_j^2}{4}}e^{\frac{\beta_j }{2}(c+X_{\partial \D,\epsilon}+Q/2\ln g)(\psi(s_j))}\nonumber\\
&\exp\Big( - \frac{Q}{4\pi}\int_{D}R_{g}(\psi)(c+X_{\partial \D }\circ \psi) g(\psi)|\psi'|^2 \,d\lambda  - \mu e^{\gamma c}\epsilon^{\frac{\gamma^2}{2}}\int_{\D}e^{\gamma (X_{\partial \D,\epsilon}+ Q/2\ln g )}\,d\lambda \Big) \nonumber\\
&\exp\Big( - \frac{Q}{2\pi}\int_{\partial D}K_{g} (\psi)(c+X_{\partial \D}\circ \psi) |\psi'| g^{1/2}(\psi)\,d\lambda_{\partial } - \mu_\partial e^{\frac{\gamma}{2} c}\epsilon^{\frac{\gamma^2}{4}}\int_{\partial \D}e^{\frac{\gamma}{2}( X_{\partial \D,\epsilon}+ Q/2\ln g) }\,d\lambda_{\partial } \Big) \Big]\,dc.\nonumber\\ 
= & \prod_i|\psi'(z_i)|^{2\triangle_{\alpha_i}} \prod_j|\psi'(s_j)|^{\triangle_{\beta_j}} \Pi_{ \gamma,\mu_{\partial  },\mu}^{(\psi(z_i),\alpha_i)_i,(\psi(s_j),\beta_j)_j}  (\D,g,F(\phi \circ \psi+Q\ln |\psi'|)).
\end{align} 
This completes the proof.\qed


\subsection{Law of the volume of space/boundary}\label{joint}
We want to express here the (joint) law of the volume of bulk/boundary on the unit disk equipped with the Euclidean metric. It will be convenient to express this law in terms of the  couple of random measures $(Z_0,Z^{\partial}_0)$ under $\P$  respectively defined on $\D$ and $\partial \D$ by (recall Proposition \ref{law})
\begin{equation}
Z_0=e^{\gamma H }e^{\gamma X_{\partial \D}}\,d\lambda,\quad Z^{\partial}_0=e^{ \frac{\gamma }{2}H} e^{\frac{\gamma }{2} X_{\partial \D}}\,d\lambda_{\partial}
\end{equation}
with
\begin{equation}
H(x)= \sum_i\alpha_iG(x,z_i) + \sum_j \frac{ \beta_j}{2} G(x,s_j).
\end{equation}
We further introduce the ratio
$$R=\frac{Z_0(\D)}{Z_0^\partial(\partial\D)^2}.$$
By definition of the law of the bulk/boundary Liouville measures, we have
 \begin{align*}
  &  \E_{\gamma,\mu_{\partial  },\mu,dx^2}^{(z_i,\alpha_i)_i,(s_j,\beta_j)_j}  \big[ F  (Z, Z_{\partial})\big]\\\
 =& (\Pi_{\gamma,\mu_{\partial  },\mu}^{(z_i,\alpha_i)_i,(s_j,\beta_j)_j}(dx^2,1))^{-1} \int_\R e^{\big(\sum_i\alpha_i+\frac{1}{2}\sum_j\beta_j-Q\big)c} \E\Big[F( e^{\gamma c}Z_0,e^{\frac{\gamma}{2} c}Z^{\partial}_0  \big) \nonumber\\
 & \exp\Big(  - \mu e^{\gamma c}Z_0(\D)   - \mu_{\partial} e^{\frac{\gamma}{2} c} Z_0^{\partial}(\partial\D)\Big)  \Big]\,dc\\
=&\frac{2}{\gamma} (\Pi_{\gamma,\mu_{\partial  },\mu}^{(z_i,\alpha_i)_i,(s_j,\beta_j)_j}(dx^2,1))^{-1} \int_0^{\infty} y^{ \frac{2}{\gamma} (\sum_i\alpha_i+\frac{1}{2}\sum_j\beta_j-Q ) -1 } \nonumber\\
& \E\Big[F\Big( y^2 R\frac{Z_0}{Z_0(\D)},y\frac{ Z^{\partial}_0}{Z^{\partial}_0(\partial\D) }  \Big)  \exp\Big(  - \mu y^2R   - \mu_{\partial}  y\Big) Z^{\partial}_0(\partial\D)^{- \frac{2}{\gamma} (\sum_i\alpha_i+\frac{1}{2}\sum_j\beta_j-Q )} \Big]\,dy.
\end{align*} 
This is the general formula. It may be useful to state as a particular example the case $\mu_\partial=0$ as it often arises in the study of random planar maps with a boundary.
\begin{corollary}\label{unitvolume}
Assume $\mu_\partial=0$ and $\mu>0$. The joint law of the bulk/boundary Liouville measures is given by 
\begin{align*}
&  \E_{\gamma,\mu_{\partial  }=0,\mu,dx^2}^{(z_i,\alpha_i)_i,(s_j,\beta_j)_j}  \big[ F  (Z, Z_{\partial})\big]\\
={}&\mathcal{Z}^{-1} \int_\R u^{\frac{1}{\gamma}\big(\sum_i\alpha_i+\frac{1}{2}\sum_j\beta_j-Q\big)-1}  \E\Big[F\Big(u\frac{Z_0}{Z_0(\D)},u^{\frac{1}{2}}\frac{ Z^{\partial}_0}{Z_0( \D)^{\frac{1}{2}} }  \Big)  Z_0( \D)^{- \frac{1}{\gamma} (\sum_i\alpha_i+\frac{1}{2}\sum_j\beta_j-Q )} \Big]e^{-\mu u}\,du.
\end{align*} 
where $\mathcal{Z}$ is a renormalization constant to get a probability measure. In particular, the law of the volume of space follows a Gamma law with parameters $\big(\frac{\sum_i\alpha_i+\frac{1}{2}\sum_j\beta_j-Q}{\gamma},\mu\big)$ and the random variable  $Z(\D)$ is independent of the random measures $(\frac{Z }{Z(\D) },\frac{Z_{\partial}}{Z(\D)^{\frac{1}{2}}})$.
\end{corollary}
\noindent If we further condition the total bulk measure to be $1$, the unit volume Liouville measure on the disk as in Corollary~\ref{unitvolume} can be defined when the following three conditions are satisfied
\begin{equation}
\forall i,\quad \alpha_i<Q,
\end{equation}
\begin{equation}
\forall j,\quad \beta_j<Q,
\end{equation}
\begin{equation}\label{unitvolumeseiberg}
Q-\sum_i\alpha_i-\frac{1}{2}\sum_j\beta_j<\frac{2}{\gamma}\wedge2\min\limits_{i}(Q-\alpha_i)\wedge\min\limits_{j}(Q-\beta_j),
\end{equation}
see Corollary~\ref{relaxseiberg} for a precise statement. This shows that the Seiberg bounds \eqref{bounds1}+\eqref{bounds2}+\eqref{bounds3} can be relaxed when conditioning on finite total volume.
\begin{remark}
We can also look at what happens when we set $\mu=0$ and condition the total boundary Liouville length measure to be $1$. In this case, one can treat bulk insertions and boundary insertions seperately as in \cite[Section 3.4]{DKRV} to obtain similar relaxed Seiberg bounds without additional technical difficulties. For completeness we state the result in the following corollary.
\end{remark}
\begin{corollary}\label{unitlength}
Assume $\mu=0$ and $\mu_\partial>0$. The joint law of the bulk/boundary Liouville measures is given by
\begin{align*}
&  \E_{\gamma,\mu_{\partial},\mu=0,dx^2}^{(z_i,\alpha_i)_i,(s_j,\beta_j)_j}  \big[ F  (Z, Z_{\partial})\big]\\
={}&\mathcal{Z}^{-1} \int_\R y^{\frac{2}{\gamma}\big(\sum_i\alpha_i+\frac{1}{2}\sum_j\beta_j-Q\big)-1}  \E\Big[F\Big(y^2\frac{Z_0}{Z_0^{\partial}(\partial\D)^2},y\frac{ Z^{\partial}_0}{Z_0^\partial(\partial\D)}  \Big)  Z_0^\partial(\partial\D)^{- \frac{2}{\gamma} (\sum_i\alpha_i+\frac{1}{2}\sum_j\beta_j-Q )} \Big]e^{-\mu u}\,dy.
\end{align*} 
where $\mathcal{Z}$ is a renormalization constant to get a probability measure. In particular, the law of the total length of the boundary follows a Gamma law with parameters $\big(2\gamma^{-1}\big(\sum_i\alpha_i+\frac{1}{2}\sum_j\beta_j-Q\big),\mu\big)$ and the random variable $Z_{\partial}(\partial\D)$ is independent of the random measures $(\frac{Z }{Z_\partial(\partial\D)^2 },\frac{Z_{\partial}}{Z_\partial(\partial\D)})$.\\
The unit boundary length Liouville measure on the disk can be defined under the following conditions:
\begin{equation}
\forall j,\quad \beta_j<Q,
\end{equation}
\begin{equation}\label{unitlengthseiberg}
Q-\sum_i\alpha_i-\frac{1}{2}\sum_j\beta_j<\frac{2}{\gamma}\wedge\min\limits_{j}(Q-\beta_j),
\end{equation}
\end{corollary}
\begin{remark}
Since the geometrical KPZ formula established in \cite{Rnew10}  has been established almost surely with respect to the GFF expectation, it holds for the Liouville measure in our context almost surely too. 
\end{remark}

\section{Liouville QFT at $\gamma=2$}\label{sec:crit}
Here we explain how to construct LQFT on the unit disk in  the important   case $\gamma=2$. The reason why this case is so specific is that it is no more superrenormalizable at small scales. In other words the interaction terms $e^{2 X_{\partial \D} } d\lambda$ or $e^{ X_{\partial \D } } d\lambda_\partial$ can no more be obtained as a  Wick ordering, i.e. a subcritical Gaussian multiplicative chaos: it corresponds to the phase transition in Gaussian multiplicative chaos theory. Indeed,   the standard renormalizations 
$$\epsilon^2 \, e^{2 X_{\partial \D,\epsilon} } d\lambda  \quad \text{ and }\quad  \epsilon  \,e^{ X_{\partial \D,\epsilon} } d\lambda_\partial$$
yield vanishing limiting measures. To get a non trivial limit, an extra push $\sqrt{\ln \frac{1}{\epsilon}}$ is necessary, which is called the Seneta-Heyde norming. For Gaussian multiplicative chaos, this has been investigated in \cite{Rnew12} for a white noise decomposition of the GFF, which does not exactly correspond to our framework as we work with convolution cutoff approximations. So, we explain in this section how to generalize the results in \cite{Rnew12} to convolutions.

We first claim
\begin{theorem}\label{seneta1} 
The family of boundary approximation measures on $\partial \D$
\begin{equation*}
\sqrt{\ln \frac{1}{\epsilon}} \,\epsilon \, e^{ X_{\partial \D,\epsilon}   } d\lambda_\partial
\end{equation*}
converges in probability as $\epsilon$ goes to $0$ towards a non trivial limiting measure, which has moments of order $q$ for all $q<1$.
\end{theorem}

\begin{theorem} \label{seneta2} 
The family of bulk approximation measures on $\D$
\begin{equation*}
\sqrt{\ln \frac{1}{\epsilon}}\, \epsilon^2\,  e^{ 2 X_{\partial \D,\epsilon}   } d\lambda
\end{equation*}
converges in probability as $\epsilon$ goes to $0$ towards a non trivial limiting measure, which has moments of order $q$ for all $q<1$.
\end{theorem}

\begin{remark}
Actually, our proof for the two above theorems establishes convergence in probability for a large class of cutoff approximations with mollifying family, not only the circle average family.  
\end{remark}

\noindent {\it Proof of Theorem \ref{seneta1}}.   The strategy is the following: first we show the convergence in probability of a specific family of white noise cutoff approximations. Then we will show that this entails the convergence in probability for a whole class of convolution approximations, including circle averages.

 Recall that if we consider a centered Gaussian distribution $X$  on the boundary of the unit disk with the following covariance structure
\begin{equation*}
\E[   X(e^{i \theta}) X(e^{i \theta})  ]= 2 \ln \frac{1}{|e^{i \theta}-e^{i \theta'}|},
\end{equation*}
then the law on the boundary of the   GFF $X_{\partial \D}$ is given by
\begin{equation*}
X_{\partial \D}= X-\frac{1}{2 \pi} \int_0^{2 \pi} X(e^{i \theta}) d \theta.
\end{equation*}
Our first step is to  construct $X_{\partial \D}$ as a function of some white noise $W$ and of a smooth Gaussian process $Y$.  This decomposition will be convenient to establish convergence in probability of the approximating  measures based on martingale techniques. We will recover the situation  of approximations based on convolution of  $X_{\partial \D}$   after that.

Recall the following decomposition (see \cite{vincent})
\begin{equation*}
\forall x\in \R^2,\quad \ln_+ \frac{1}{|x|}= 2 \int_0^1 (t-|x|^{\frac{1}{2}})_{+} \frac{dt}{t^2}+2 (1-|x|^{\frac{1}{2}})_{+}\,.
\end{equation*}
Now we construct two Gaussian distributions: the first one $\bar{X}$ will have the covariance structure of the first term in the above right-hand side and the second one $Y$ the second term. 

\begin{lemma}\label{decomp}
There exists a white noise $W$ on $[1,+\infty[\times \partial\D$ and a family of centered Gaussian processes $( \bar{X}_\epsilon)_{\epsilon\in]0,1]}$ on $\partial \D$, which are measurable functions of this white noise, such that
\begin{equation}
\forall 0<\epsilon <\epsilon'\leq 1,\quad \bar{X}_\epsilon-\bar{X}_{\epsilon'} \text{ is independent of } \sigma\{X_u(e^{i\theta});\epsilon'\leq u\leq 1,\theta\in[0,2\pi]\}
\end{equation}
and
\begin{equation}
\E[   \bar{X}_\epsilon (e^{i \theta}) \bar{X}_\epsilon (e^{i \theta'})   ]  = 2 \int_{\sqrt{\epsilon}}^1 (t-|e^{i \theta}-e^{i \theta'}|^{\frac{1}{2}})_{+}\frac{dt}{t^2}= \int_1^{\frac{1}{\epsilon}}  (1-|v(e^{i \theta}-e^{i \theta'})|^{\frac{1}{2}})_{+} \frac{dv}{v}.
\end{equation}
The limiting distribution $\bar{X}=\lim_{\epsilon\to 0} \bar{X}_\epsilon $ is a centered Gaussian distribution with covariance structure
\begin{equation*}
\E[   \bar{X}(e^{i \theta}) \bar{X}(e^{i \theta'})   ]  = 2 \int_0^1 (t-|e^{i \theta}-e^{i \theta'}|^{\frac{1}{2}})_{+}\frac{dt}{t^2}.
\end{equation*}
Finally, for any smooth function $R$ on $[1,+\infty[\times \partial\D$ with compact support, the function 
$$z\in\partial \D\mapsto T_\epsilon(R)(z):=\E[ \bar{X}_\epsilon (z)W(R)]$$ is a continuous function which converges uniformly as $\epsilon\to 0$ towards 
$$z\in\partial \D\mapsto T(R)(z):=\E[ \bar{X} (z)W(R)].$$
\end{lemma}

This lemma is proved in Appendix~\ref{proof:decomp}. Then we consider a centered Gaussian field $Y$ independent of $ ( \bar{X}_\epsilon)_{\epsilon\in]0,1]}$ with covariance given by
\begin{equation*}
\E[   Y (e^{i \theta}) Y (e^{i \theta'})   ]  =  (1-|e^{i \theta}-e^{i \theta'}|^{\frac{1}{2}})_{+}\,.
\end{equation*}
Recall that such a kernel is indeed positive definite \cite{golubov}. 

Now we can set 
\begin{equation*}
X_{\partial \D}= \bar{X}+Y-\frac{1}{2 \pi} \int_0^{2 \pi} ( \bar{X}(e^{i \theta})+Y( e^{i \theta} ) )d \theta. 
\end{equation*}
This is a construction of $X_{\partial \D}$ as a function of $(W,Y)$. Now we would like to use \cite{Rnew12} to show that the random measures
\begin{equation}\label{cvprobsen}
\sqrt{\ln \frac{1}{\epsilon}} \,\epsilon\, e^{  \bar{X}_\epsilon   } d \lambda_{\partial}
\end{equation}
converges in probability to a non trivial limiting random measure as $\epsilon\to 0$.  To this purpose, observe that the covariance $(k_\epsilon)_{\epsilon\in]0,1]}$ kernels of the family $( \bar{X}_\epsilon)_{\epsilon\in]0,1]}$ can be written as 
$$k_\epsilon(e^{i\theta},e^{i\theta'})=\int_1^{\frac{1}{\epsilon}}\frac{k(v,e^{i\theta},e^{i\theta'})}{v}\,dv\quad \text{ with }\quad k(v,e^{i\theta},e^{i\theta'})=(1-|v(e^{i \theta}-e^{i \theta'})|^{\frac{1}{2}})_{+}.$$
Such a kernel $k$ satisfies the properties
\begin{description}
\item[A.1] $k$ is nonnegative, continuous. 
\item[A.2] $k$ is H\"older on the diagonal, more precisely   $\forall \theta,\theta'$, $\forall v\geq 1$,  
$$  |k(v,e^{i\theta},e^{i\theta})-k(v,e^{i\theta},e^{i\theta'})|\leq   v^{1/2}|e^{i \theta}-e^{i \theta'}|^{1/2}$$
\item[A.3] $k$ satisfies the integrability condition 
$$\sup_{\theta,\theta'}\int_{\frac{1}{|e^{i\theta}-e^{i\theta'}|}}^{\infty}\frac{k(v,e^{i\theta},e^{i\theta'})}{v}\,dv<+\infty.$$ 
\item[A.4] for all $\epsilon\in]0,1]$, $\int_1^{ \frac{1}{\epsilon}}  \frac{k(v,e^{i\theta},e^{i\theta})}{v}\,dv=\ln\frac{1}{\epsilon}$, 
\item[A.5] $k(v,e^{i\theta},e^{i\theta'})=0$ for $|e^{i\theta}-e^{i\theta'}|\geq  v^{-1} $.
\end{description}
Observe in particular that  [A.2] implies that
$$|\ln\frac{1}{\epsilon}-k_\epsilon(e^{i\theta},e^{i\theta'})|\leq\int_{1}^{1/\epsilon} \frac{|e^{i\theta}-e^{i\theta'}|^{1/2}}{v^{1/2}}dv\leq C (|e^{i\theta}-e^{i\theta'}|/\epsilon)^{1/2}$$
for some constant $C$ (independent of $\epsilon$). In particular we have the property
\begin{equation}\label{LBMcrit1}
|e^{i\theta}-e^{i\theta'}|\leq \epsilon \quad \Rightarrow \quad |\ln\frac{1}{\epsilon}-k_\epsilon(e^{i\theta},e^{i\theta'})|\leq C.
\end{equation}
These properties are the only assumptions used in \cite{LBMcrit} to construct the derivative martingale  and in \cite{Rnew12} to prove the Seneta-Heyde norming. Therefore the family of random measures \eqref{cvprobsen} converges in probability towards a non trivial random measures, which has moments of order $q$ for all $q<1$ (see \cite{Rnew12}).

Hence, if $X_\epsilon= \bar{X}_\epsilon+Y-\frac{1}{2 \pi} \int_0^{2 \pi} ( \bar{X}_\epsilon(e^{i \theta})+Y( e^{i \theta} ) )d \theta$, then  
\begin{equation*}
M_\epsilon=\sqrt{\ln \frac{1}{\epsilon}} \,\epsilon\, e^{  X_\epsilon(e^{i \theta})   } d\lambda_{\partial}
\end{equation*}
converges in probability to a random measure $M'$ which is a measurable function of the white noise $W$ and the process $Y$, call it $F(W,Y)$.  

\medskip
Now, we show convergence in probability of $\sqrt{\ln \frac{1}{\epsilon}}\, \epsilon\, e^{ X_{\partial \D,\epsilon}  }  d\lambda_{\partial}$, where $X_{\partial \D,\epsilon}$ is  the circle average approximation  of $X_{\partial \D}$ towards the same limit $M'$. The ideas in the following stem from the techniques developed in \cite{review} along with some variant of Lemma~49 in \cite{shamov} (we will not recall Lemma~49 as our proof will be self contained).

For this, we introduce $X_{\partial \D}^1$, an independent copy of $X_{\partial \D}$, and $X_{\partial \D,\epsilon}^1$ its circle average approximation.  Let us define for $t\in[0,1]$ and $\theta\in [0,2\pi]$
$$
Z_\epsilon(t,e^{i\theta})=\sqrt{t}X_{\partial \D,\epsilon}^1(e^{i\theta})+\sqrt{1-t} X_\epsilon(e^{i\theta}).
$$       
Now, we set 
\begin{equation*}
M_{ \epsilon}^1=\sqrt{\ln \frac{1}{\epsilon}}\, \epsilon\, e^{  X_{\partial \D,\epsilon}^1(e^{i \theta})    } d\lambda_\partial.
\end{equation*}

We first show that $M_{ \epsilon}^1$ converges in distribution to $M'=F(W ,Y)$.  
From \cite[Proof of Theorem 2.1]{vincent}, one gets that for all $\alpha<1$
\begin{align*}
& \underset{\epsilon \to 0} {\overline{\lim}}  \big|\E[M_{\epsilon}^1(B)^\alpha] -\E[M_{\epsilon}(B)^\alpha]\big|  \\
& \leq c \frac{\alpha(1-\alpha)}{2}C_A \underset{\epsilon \to 0} {\overline{\lim}} \int_0^1\E\Big[\Big ( \sqrt{\ln \frac{1}{\epsilon}}  \int_{\partial \D} e^{Z_\epsilon(t,\cdot)-\frac{1}{2}\E[Z_\epsilon(t,\cdot)^2]}\,d\lambda_{\partial}  \Big )^\alpha\Big] \,dt \\
& +  c\,  \overline{C}_A  \underset{\epsilon \to 0} {\overline{\lim}} \int_0^1 \E \Big [\Big ( \sup_{0\leq i<\frac{1}{A\epsilon}}  \sqrt{\ln \frac{1}{\epsilon}}  \int_{2iA\epsilon}^{2(i+1)A\epsilon}e^{Z_\epsilon(t,e^{i\theta})-\frac{1}{2}\E[Z_\epsilon(t,e^{i\theta})^2]}  d\theta \Big)^\alpha     \Big]  dt,   
\end{align*}
where
\begin{equation*}
C_A=\underset{\epsilon \to 0}{\overline{\lim}} \sup_{|e^{i\theta}-e^{i\theta'}| \geq A \epsilon}  | \E[ X_{\partial \D,\epsilon}^1(e^{i\theta}) X_{\partial \D,\epsilon}^1(e^{i\theta'})   ] -  \E[ X_{\epsilon}(e^{i\theta}) X_{\epsilon}(e^{i\theta'})   ] |
\end{equation*}
and
\begin{equation*}
\overline{C}_A= \underset{\epsilon \to 0}{\overline{\lim}} \sup_{|e^{i\theta}-e^{i\theta'}| \leq A \epsilon}  | \E[ X_{\partial \D,\epsilon}^1(e^{i\theta}) X_{\partial \D,\epsilon}^1(e^{i\theta'})   ] -  \E[ X_{\epsilon}(e^{i\theta}) X_{\epsilon}(e^{i\theta'})   ] |.
\end{equation*}
The reader can check that $\overline{C}_A$ is bounded independently of $A$ and $\underset{A \to \infty}{\lim} C_A=0$. Since \\ $\E\big[ \left ( \sqrt{\ln \frac{1}{\epsilon}}  \int_0^1 e^{Z_\epsilon(t,u)-\frac{1}{2}\E[Z_\epsilon(t,u)^2]}\,du  \right )^\alpha\big] $ is also bounded  independently of everything (by comparison with Mandelbrot's multiplicative cascades as explained in the \cite[Appendix]{Rnew7} and \cite[Appendix B.4]{Rnew12}), we are done if we can show that for all $t\in [0,1]$
\begin{equation}\label{claim}
 \underset{\epsilon \to 0} {\overline{\lim}} \,\,\E \Big [\Big ( \sup_{0\leq i<\frac{1}{A\epsilon}}  \sqrt{\ln \frac{1}{\epsilon}}  \int_{2iA\epsilon}^{2(i+1)A\epsilon}e^{Z_\epsilon(t,e^{i\theta})-\frac{1}{2}\E[Z_\epsilon(t,e^{i\theta})^2]}  d\theta \Big)^\alpha     \Big] =0.
 \end{equation}
Notice that this quantity is less than
\begin{equation}\label{ascola}
\big(\ln \frac{1}{\epsilon}\big)^{\alpha/2}\epsilon^\alpha \E  \Big [  \Big  (    e^{\sup_{\theta \in [0,2\pi]} Z_\epsilon(t,e^{i\theta})-\frac{1}{2}\E[Z_\epsilon(t,e^{i\theta})^2]}   \Big )^\alpha     \Big ] .
\end{equation} 
To estimate this quantity, we use the main result  of \cite{acosta}: more precisely, setting 
$$m_\epsilon=2\ln\frac{1}{\epsilon}-\frac{3}{2}\ln\ln \frac{1}{\epsilon},$$
 we claim that there exist two constants $C,c>0$ such that for $\epsilon$ small enough
\begin{align*}
\forall x\geq 0,\quad \P\Big(\Big|\max_{\theta\in [0,2\pi]}Z_\epsilon(t,e^{i\theta})-m_\epsilon \Big|\geq x\Big)\leq Ce^{-cx}.
\end{align*}
In particular we get that for   $\alpha<c$  
\begin{equation*}
\sup_{\epsilon} \E  \Big [  \Big  (    e^{\sup_{\theta \in [0,2\pi]} Z_\epsilon(t,e^{i\theta})}   \Big )^\alpha     \Big ]<\infty.
\end{equation*}
Plugging this estimate into \eqref{ascola}, we see that the quantity \eqref{ascola} is less than
$$C' \big(\ln \frac{1}{\epsilon}\big)^{\alpha/2}\epsilon^{2\alpha}e^{\alpha m_\epsilon}=C' \big(\ln \frac{1}{\epsilon}\big)^{-\alpha}. $$
for some constant $C'>0$. This proves the claim \eqref{claim}, hence the convergence in law of the random measure $M_\epsilon^1$ towards  $M'=F(W,Y)$. 

Now we deduce that the family $(W,Y,M_\epsilon^1)_\epsilon$ converges in law. Take any  smooth function $R$ on $[1,+\infty[\times \partial\D$ with compact support, any  continuous function $g$ on $\partial\D$, any bounded continuous function $G$ on $\R$ and $u\in\R$. We have by using the Girsanov transform 
$$\E[e^{W(R)+uY}G(M_\epsilon^1(g))]=e^{\frac{1}{2}{\rm Var}[W(R)+uY]}\E[ G(M_\epsilon^1(e^{T_\epsilon(R)}g)]$$
where $T_\epsilon(R)$ is defined in Lemma \ref{decomp}. The quantity in the right-hand side converges as $\epsilon\to 0$ towards
 $$e^{\frac{1}{2}{\rm Var}[W(R)+uY]}\E[ G(M'(e^{T(R)}g)]=\E[e^{W(R)+uY}G(M'(g))].$$
Hence our claim about the convergence in law of the triple $(W,Y,M_\epsilon^1)_\epsilon$ towards $(W,Y,M'=F(W,Y))$ is proved.
 
 Now we consider the family $(W,Y, M_{\epsilon}^1,F(W,Y))_\epsilon$, which is tight. Even if it means extracting a subsequence, it converges in law towards some $(\mathcal{W}, \mathcal{Y}, \mathcal{M}, \bar{\mathcal{M}})$. We have just shown that the law of $(\mathcal{W}, \mathcal{Y}, \mathcal{M})$ is that of  $(\mathcal{W},\mathcal{Y}, F
(\mathcal{W},\mathcal{Y}))$, i.e. the same as the law of 
$(\mathcal{W},\mathcal{Y}, \bar{\mathcal{M}})$. Hence $\mathcal{M} =\bar{\mathcal{M}}$ almost surely. Therefore $M^1_{\epsilon}- F(W,Y)$ converges in law towards $0$, hence in probability.   Since the  convergence in probability of the family $(M^1_{\epsilon})_\epsilon$ implies the convergence of probability of every family $(\widehat{M}_{\epsilon})_\epsilon$ that has the same law as $(M^1_{\epsilon})_\epsilon$, the proof of Theorem \ref{seneta1} is complete.

Finally, one can notice that instead of $X_{\partial \D,\epsilon}$ we could have considered any smooth convolution approximation of $X$.     \qed

\bigskip

\noindent {\it Proof of Theorem \ref{seneta2}.} Let us consider the Poisson kernel on the unit disk 
$$\forall 0\leq r <1,\forall \theta\in [0,2\pi],\quad P_r(\theta)=\sum_{n\in\Z} r^{|n|}e^{in\theta}.$$
We can then consider the harmonic extension inside the unit disk of the trace of the GFF $X_{\partial \D}$ along the boundary
$$P_X(re^{i\theta})=\frac{1}{2\pi}\int_0^{2\pi}P_r(\theta-t)X_{\partial \D}(e^{i t})\,dt.$$
 It is plain to see that $P_X$ is a continuous Gaussian process inside the unit disk. If we set 
 $$X^{{\rm Dir}}=X_{\partial \D}-P_X,$$ one can check that we get a GFF with Dirichlet boundary condition in the unit disk. Therefore, by continuity of $P_X$ inside $\D$, the convergence in probability of the random measures $(\epsilon^2 \, e^{ 2 X_{\partial \D,\epsilon}(x) } d\lambda)_\epsilon$ boils down to showing the convergence in probability for the random measures
$$(\epsilon^2 \, e^{ 2 X^{{\rm Dir}}_\epsilon(x) } d\lambda)_\epsilon$$
where $(X^{{\rm Dir}}_\epsilon)_\epsilon$ stands for the circle average approximations of the GFF $X^{{\rm Dir}}$. Given the fact that the Seneta-Heyde norming has been proved in \cite{Rnew12} for a white noise decomposition of $X^{{\rm Dir}}$, we can use the same argument as in the proof of Theorem \ref{seneta1} to show that convergence for the white noise approximation family entails the convergence in probability for the circle average approximations.\qed

\medskip
From now on, the construction of the Liouville LQG on the unit disk for $\gamma=2$ follows the same lines as for $\gamma<2$ by taking the limit as $\epsilon\to 0$ of the quantity 
\begin{align}\label{eq:defPigcrit}
&\Pi_{2,\mu_{\partial  },\mu}^{(z_i,\alpha_i)_i,(s_j,\beta_j)_j}  (g,F)\\
 =&e^{  \frac{1}{96\pi}\Big(\int_{\D}|\partial \ln g|^2\,d\lambda+\int_{\partial \D}4\ln g\,d\lambda_{\partial} \Big)} \lim_{\epsilon\to 0}\int_\R\E\Big[F( X_{\partial  \D}+c+\ln g)\prod_i \epsilon^{\frac{\alpha_i^2}{2}}e^{\alpha_i (c+X_{\partial\D,\epsilon}+\ln g)(z_i)} \nonumber\\
&\prod_j \epsilon^{\frac{\beta_j^2}{4}}e^{\frac{\beta_j }{2}(c+X_{\partial\D,\epsilon}+\ln g)(s_j)}\exp\Big( - \frac{2}{4\pi}\int_{\D}R_{g} (c+X_{\partial  \D})  \,d\lambda_g - \mu e^{2 c}\sqrt{-\ln\epsilon}\epsilon^{2}\int_{\D}e^{2 (X_{\partial \D,\epsilon}+ \ln g) }\,d\lambda \Big)\nonumber \\
&\exp\Big( - \frac{2}{2\pi}\int_{\partial \D}K_{g} (c+X_{\partial\D})  \,d\lambda_{\partial g} - \mu_\partial e^{ c}\sqrt{-\ln\epsilon}\epsilon\int_{\partial\D}e^{  (X_{\partial\D,\epsilon}+  \ln g) }\,d\lambda_{\partial } \Big) \Big]\,dc.\nonumber %
\end{align}
defined for all continuous and bounded functional $F$ on $H^{-1}(\D)$. From now on, the properties of LQG  (and their proofs) on the disk for $\gamma=2$ are the same as for $\gamma<2$ except Proposition \ref{finitemeas}, which needs some extra care that we treat now.
  
 \begin{proposition}\label{finitemeascrit}
The quantities below are almost surely finite
\begin{equation*}
\int_\D e^{2 X_{\partial\D}}d\lambda \quad \text{ and }\quad \int_{\partial\D }e^{X_{\partial\D}}d\lambda_{\partial \D}.
\end{equation*}
\end{proposition}

\begin{proof} 
Recall the sub-additivity inequality for $\alpha\in]0,1[$: if $(a_j)_{1\leq j\leq n}$ are positive real numbers then
$$(a_1+\dots+a_n)^\alpha\leq a_1^\alpha+\dots+a_n^\alpha.$$
Now we use Kahane's convexity inequality \cite[Theorem 2.1]{review} to compare the Gaussian multiplicative chaos with standard dyadic lognormal cascade (once again we refer to \cite[Appendix B.1]{Rnew7} for full details). We consider the dyadic tree with i.i.d.  weights with Gaussian law $\mathcal{N}(0,\ln 2)$ on the edges of the tree and denote by $Y_j^n$ the sum of these weights starting from the root up to the dyadic indexed by $j$ at generation $n$. We denote by $(Z_j)_j$ an i.i.d sequence (independent of everything) standing for the mass of the dyadic cascade at criticality rooted at the dyadic $j$ at generation $n$. From \cite{madaule} these random variables have distribution tail $\P(Z_j >x) \leq \frac{C}{x}$ for some constant $C>0$, and $\E[Z_j^q]<\infty$ for $q<1$. Hence we get
\begin{align*}
 & \mathds{E}\Big[\Big(\int_\mathds{D} e^{2 X_{\partial\D}(x)-2  \E[X_{\partial\D}(x)^2] } \frac{1}{(1-|x|^2)^{2}})\lambda(dx)\Big)^\alpha\Big]\\
& \leq \sum\limits_{n\in\mathds{N}}2^{2 n\alpha }\mathds{E}\Big[ \Big(\int_{1-2^{-n} \leq |x|^2 \leq 1-2^{-n-1} } e^{2 X_{\partial\D}(x)-2  \E[X_{\partial\D}(x)^2] }\lambda(dx)\Big)^\alpha\Big] \\
& \leq  \sum\limits_{n\in\mathds{N}} \mathds{E}\Big[ \Big( \sum_{j=1}^{2^n} Z_j   e^{2 \sqrt{2} (Y_j^n-  \sqrt{2}  \ln 2 \, n  ) }  \Big)^\alpha\Big] \\
& =  \sum\limits_{n\in\mathds{N}} \frac{1}{n^{3 \alpha}} \mathds{E}\Big[ \Big( \sum_{j=1}^{2^n} Z_j  e^{2 \sqrt{2} (Y_j^n-  \sqrt{2}  \ln 2 \, n   +\frac{3}{2 \sqrt{2}} \ln n) }  \Big)^\alpha\Big]  .
\end{align*}
Let $\eta \in]0,1[$. By Jensen and for some constant $B>0$
\begin{align*}
   \mathds{E}\Big[ \Big( \sum_{j=1}^{2^n} Z_j  e^{2 \sqrt{2} (Y_j^n-  \sqrt{2}  \ln 2 \, n   +\frac{3}{2 \sqrt{2}} \ln n) }  \Big)^\alpha\Big]   &=  \mathds{E}\Big[ \Big( \sum_{j=1}^{2^n} \E[Z_j^{1-\eta}]  e^{2 \sqrt{2} (1-\eta)(Y_j^n-  \sqrt{2}  \ln 2 \, n   +\frac{3}{2 \sqrt{2}} \ln n) }  \Big)^{\frac{\alpha}{1-\eta}}\Big]  \\
   &\leq B \mathds{E}\Big[ \Big( \sum_{j=1}^{2^n}  e^{2 \sqrt{2} (1-\eta)(Y_j^n-  \sqrt{2}  \ln 2 \, n   +\frac{3}{2 \sqrt{2}} \ln n) }  \Big)^{\frac{\alpha}{1-\eta}}\Big]
\end{align*}
From \cite{madaule} again, this last expectation is bounded independently of $n$ provided that we choose $2\alpha(1-\eta)<1$. In that case, up to changing the value of $B$, we get
\begin{align*}
\mathds{E}\Big[\Big(\int_\mathds{D} e^{2 X_{\partial\D}(x)-2  \E[X_{\partial\D}(x)^2] } \frac{1}{(1-|x|^2)^{2}})\lambda(dx)\Big)^\alpha\Big]  \leq  B \sum\limits_{n\in\mathds{N}} \frac{1}{n^{3 \alpha}}   ,
\end{align*}
which can be obviously made finite provided that $\alpha >1/3$.
 \end{proof}

\section{Conjectures related to planar quadrangulations with boundary}

We consider $\overline{\mathcal{Q}}_{n,p}$ the set of quandrangulations of size $n$, i.e. with $n$ inner faces and a simple boundary of length $2p$ with one marked edge on the boundary and one marked face inside. Now to each quadrangulation $Q$ with a marked point inside and a marked point on the boundary (we choose at random a point in the marked face at a point on the marked edge), we associate a standard conformal structure (by gluing Euclidean squares along their edges as prescribed by the quadrangulation) and map it to the disk such that the interior point gets mapped to $0$ and the frontier point to $1$. We give volume $a^2$ to each quadrilateral and length $a$ to each edge on the boundary: we denote $\nu_{Q,a}$ the corresponding volume measure and $\nu_{Q,a}^{\partial}$ the corresponding boundary length measure. Recall that we have the following asymptotics as $n,p \to \infty$ with $\frac{n}{p^2}$ tending to some value (see appendix):
\begin{equation*}
|\overline{\mathcal{Q}}_{n,p}  |  \sim  e^{n \ln 12} e^{2p \ln \frac{3}{\sqrt{2}} } n^{-3/2}  \frac{\sqrt{3p}}{2 \pi}    e^{- \frac{9 (2p)^2}{16n}}
\end{equation*}
and we set $\mu^c=  \ln 12, \, \mu^c_\partial= \ln \frac{3}{\sqrt{2}}$ (these two constants are not universal as they depend on the class of map one considers, i.e. are different for triangulations, etc...).
Now, we consider the measures $(\nu_{a},\nu_{a}^{\partial})$ defined by the following expression for all $F$

\begin{equation*}
\E^{a}[  F( \nu_{a},\nu_{a}^{\partial} )  ]= \frac{1}{Z_{a}}\sum_{n,p} e^{-\bar{\mu} n} e^{-\bar{\mu}_\partial   2p}\sum_{Q \in \mathcal{Q}_{n,p}} F(\nu_{Q,a},\nu_{Q,a}^{\partial}  ),
\end{equation*} 
where the constants $\bar{\mu},\bar{\mu}_\partial$ are functions of $a>0$ defined by
$\bar{\mu}=\mu_c+a^2 \mu, \: \bar{\mu}_\partial =\mu^c_\partial +a\mu_\partial$ and $Z_a$ is a normalization constant.
We can now state a precise mathematical conjecture:

\begin{conjecture}\label{conjecturecartes}
The limit  in law  $\: \underset{a \to 0}{\lim} \: (\nu_{a},\nu_{a}^{\partial} ) $ exists in the product space of Radon measures equipped with the topology of weak convergence and is given (up to deterministic constants) by the Liouville measure of LQG with parameter $\gamma=\sqrt{\frac{8}{3}}$, appropriate cosmological constants and $\alpha_1=\gamma$, $\beta_1=\gamma$ and points $z_1=0$, $s_1=1$. 
\end{conjecture}

Here we give a few more details on the above conjecture. It states the existence of constants $\bar{C},\bar{c}>0$ such that 

\begin{equation}\label{conjprecise}
\underset{a \to 0}{\lim} \: \E^{a}[  F( \nu_{a},\nu_{a}^{\partial} )  ] = \E ^{(0,\gamma),(1,\gamma)}_{\gamma,\bar{C} \mu,\bar{c} \mu_{\partial} } [ F( \bar{C}Z,\bar{c}Z_\partial ) ]
\end{equation}
with $\gamma=\sqrt{\frac{8}{3}}$. 
Looking at Section~\ref{joint}, recall that we have for all $\gamma \in ]\sqrt{2},2[$ 
\begin{align}
 &  \E ^{(0,\gamma),(1,\gamma)}_{\gamma,\bar{C} \mu,\bar{c} \mu_{\partial} } [ F( \bar{C}Z,\bar{c}Z_\partial ) ]  \nonumber \\
 & =   \frac{1}{C_{\mu,\mu_{\partial} ,\gamma}}\int_\R e^{(\gamma-\frac{2}{\gamma})c} \E\Big[F( \bar{C}e^{\gamma c}Z_0 , \bar{c}e^{\frac{\gamma}{2} c}Z^{\partial}_0  \big) 
  \exp\Big(  - \bar{C} \mu e^{\gamma c}Z_0(\D)   - \bar{c} \mu_{\partial} e^{\frac{\gamma}{2} c} Z_0^{\partial}(\partial\D)\Big)  \Big]\,dc  \label{explicit}
 \end{align} 
where the couple $(Z_0,Z^{\partial}_0 )$ is defined by
\begin{equation*}
Z_0=e^{\gamma H }e^{\gamma X_{\partial \D}}\,d\lambda,\quad Z^{\partial}_0=e^{ \frac{\gamma }{2}H}e^{\frac{\gamma }{2} X_{\partial \D}}\,d\lambda_{\partial}
\end{equation*}
where the couple $(e^{\gamma X_{\partial \D}}\,d\lambda,e^{\frac{\gamma }{2} X_{\partial \D}}\,d\lambda_{\partial})$  is a standard couple of Gaussian chaos measures defined by a limiting procedure in Proposition~\ref{law}   and
\begin{equation*}
H(x)= \gamma G(x,0)+\frac{\gamma}{2} G(x,1).  
\end{equation*}
with $G$ the standard Green function in the disk (see \eqref{GreenN} for the definition).
The constants $\bar{C},\bar{c}$ are non universal in the sense that they depend on the class of planar map you consider. For instance, the constants $\bar{C},\bar{c}$ will be different if you consider triangulations instead of quadrangulations. It would be interesting to know these constants (in the case of quadrangulations say); however, we do not know how to compute them as it requires information on the joint law of $(Z_0(\D),Z^{\partial}_0(\partial \D))$.

It is known that the joint law of the total volume and the total boundary length of $(\nu_{a},\nu_{a}^{\partial})$  is given by the following density within the regime of conjecture \ref{conjecturecartes} (see Appendix~\ref{mapasymp})
\begin{equation}\label{distrimap}
\frac{1}{D_{\mu,\mu_{\partial}}} V^{-3/2} l^{1/2} e^{- \mu V} e^{-\mu_{\partial} l} e^{-\frac{9 l^2 }{16V}} dl dV. 
\end{equation}
In fact, the above distribution should be universal, i.e. should not depend on the planar map model, except for the $\frac{9}{16}$ constant in $e^{-\frac{9 l^2 }{16V}}$ which is specific to quadrangulations and in the case of triangulations (for instance) one should get a different constant than $\frac{9}{16}$.  One can in fact read on relations \eqref{explicit} and \eqref{distrimap} where the relation $\gamma=\sqrt{\frac{8}{3}}$ comes from. Indeed, for any function $G$, by making a simple change of variables $V=\bar{C} e^{\gamma c} Z_0(\D) $ in \eqref{explicit} we get that
\begin{equation*}
\E ^{(0,\gamma),(1,\gamma)}_{\gamma,\bar{C} \mu,\bar{c} \mu_{\partial} } [ G( \bar{C}Z(\D))  e^{\bar{c}\mu_{\partial} Z_\partial (\partial \D)  } ] = \frac{1}{\gamma C_{\mu,\mu_{\partial} ,\gamma} }  \E[   \frac{1}{( \bar{C}  Z_0(\D)   )^{1-2/\gamma^2}}    ]  \int_0^\infty G(V) V^{-\frac{2}{\gamma^2}}e^{-\mu V}  dV. 
\end{equation*} 
Similarly, one has 
\begin{equation*}
\frac{1}{D_{\mu,\mu_{\partial}}} \int_0^\infty   \int_0^\infty   \left (     G(V) e^{\mu_{\partial} l }  \right )   V^{-3/2} l^{1/2} e^{- \mu V} e^{-\mu_{\partial} l} e^{-\frac{9 l^2 }{16V}} dl dV= \frac{\int_0^\infty  \tilde{l}^{-1/4} e^{-\frac{9}{16}  \tilde{l}^2}   d\tilde{l}}{2 D_{\mu,\mu_{\partial}}}   \int_0^\infty  G(V)   V^{-\frac{3}{4}}e^{-\mu V} dV,
\end{equation*}
by using the change of variable $\tilde{l}=\frac{l^2}{V}$. 
This shows that the only possible choice for \eqref{conjprecise} to hold is $\gamma$ such that $\frac{2}{\gamma^2}= \frac{4}{3}$, i.e. $\gamma=\sqrt{\frac{8}{3}}$.

 One could also state similar conjectures with three distinct marked points on the boundary (instead of one interior marked point and one marked point on the boundary) or/and by conditioning on the measures to have fixed volume (instead of the Boltzmann weight setting of conjecture \ref{conjecturecartes}). One could also state similar conjectures where the quadrangulation is chosen according to the partition function of a model of statistical physics (at critical temperature): in that case, the value of $\gamma$ in conjecture \ref{conjecturecartes} will depend on the model and can be read on the asymptotics of the partition function of the quadrangulation (in a way similar to the way we derived the relation $\gamma=\sqrt{\frac{8}{3}}$ for uniform quadrangulations).

Finally, let us mention that variants of the measures defined by \eqref{explicit} (where you fix three points on the boundary and condition on the volume of the bulk measure or the boundary) should be related (in a similar way as the sphere case) to the unit area quantum disk and the unit boundary length quantum disk which appear in \cite{DMS}.

\section{Appendix}

\subsection{Asymptotics of quadrangulations with a boundary}\label{mapasymp}
Here we take material from \cite{Bouttier} (see also \cite{CurMie}). Let $\mathcal{Q}_{n,p}$ denote quandrangulations of size $n$ with a simple boundary of length $2p$ and a marked point on the frontier. Then we have
\begin{equation*}
|\mathcal{Q}_{n,p}  |= \frac{1}{3^p}  \frac{(3p)!}{p! (2p-1)!} 3^n \frac{(2n+p-1)!}{(n-p+1)! (n+2p)!}.
\end{equation*} 
We are interested in the asymptotics of $|\mathcal{Q}_{n,p}  |$ as $n,p \to \infty$ with $\frac{p^2}{n}$ fixed. Notice that we have within this asymptotic:
\begin{align*}
(2n+p-1)! \sim & \sqrt{2 \pi} 2^{2n+p-1}  e^{(2n+p-1) \ln n+p-1+\frac{p^2}{4n}-(2n+p-1)}  \sqrt{2n},\\
(n-p+1)! \sim& \sqrt{2 \pi} e^{(n-p+1) \ln n -p+1+\frac{p^2}{2n}-(n-p+1)}  \sqrt{n},\\
(n+2p)! \sim &\sqrt{2 \pi}  e^{(n+2p) \ln n +2p +\frac{2p^2}{n}-(n+2p)}  \sqrt{n}  .
\end{align*}
Hence, we get that
\begin{equation*}
\frac{(2n+p-1)!}{(n-p+1)! (n+2p)!}  \sim  \sqrt{\frac{1}{\pi}}n^{-5/2} 2^{2n+p-1}  e^{-\frac{9p^2}{4n}}.
\end{equation*}
Also,
\begin{equation*}
\frac{(3p)!}{p! (2p-1)!} \sim  \frac{\sqrt{3}}{{ \sqrt{\pi}}}  \sqrt{p}   (\frac{27}{4})^p
\end{equation*}
in such way that we get
\begin{equation*}
|\mathcal{Q}_{n,p}  |  \sim  12^n (\frac{9}{2})^p n^{-5/2}  \frac{\sqrt{3p}}{2 \pi}    e^{- \frac{9 p^2}{4n}}.
\end{equation*}
Finally, if $\overline{\mathcal{Q}}_{n,p}$ denotes the set of quandrangulations of size $n$ with a simple boundary of length $2p$ with one marked point on the frontier and one marked point inside then we get
\begin{equation*}
|\overline{\mathcal{Q}}_{n,p}  |  \sim  e^{n \ln 12} e^{2p \ln \frac{3}{\sqrt{2}} } n^{-3/2}  \frac{\sqrt{3p}}{2 \pi}    e^{- \frac{9 (2p)^2}{16n}}.
\end{equation*}

\subsection{Some auxiliary estimates}\label{estimations}
Here we give hints for some estimates used in the proof of Theorem \ref{th:seiberg} and Proposition \ref{prop:part}. We stick to the notations used in this proof.
\begin{lemma}
On boundary behavior of the regularized Green function $G_\epsilon$: remember that $D_\epsilon$ is the disk of radius $\epsilon$ centered at $1-2\epsilon$, we claim that
$$\sup\limits_{\epsilon>0}\sup\limits_{x\in D_\epsilon}|G_\epsilon(x,x)+2\ln\epsilon|<+\infty .$$
As a consequence, one sees that if $x\in D_\epsilon$,
$$ |\E[X_{\partial \D,\epsilon}(x)^2]-2\ln \frac{1}{\epsilon}|\leq C,\quad |G_\epsilon (x,1)-2\ln\frac{1}{\epsilon}|\leq C.$$
\end{lemma}
\begin{proof}
Let us calculate $G_\epsilon(x,x)$ for $\epsilon>0$ small enough. Recall that the non-regularized Green function $G(x,y)$ is the sum of $\ln\frac{1}{|x-y|}$ and $\ln\frac{1}{|1-x\overline{y}|}$. We have already seen that the $\epsilon$-regularization of $\ln\frac{1}{|x-y|}$ part of $G_\epsilon(x,x)$ will simply be $-\ln\epsilon$ as in the proof of proposition \ref{circlegreen}. Now for the $\ln\frac{1}{|1-x\overline{y}|}$ part, we remark a scaling relation: we can compare what is happening at $\epsilon$ with that at $\epsilon/2$ via the following observation (with $a,b>0$ both small of order $\epsilon$)
$$\ln\frac{|a/2+b/2-ab/4|}{|a+b-ab|}-\ln\frac{1}{2}=\ln\frac{|a+b-ab/2|}{|a+b-ab|}\asymp \frac{|ab/2|}{|a+b-ab|}\leq |a|$$
By taking $a=1-(x+\epsilon e^{i\theta})$ and $b=1-(\overline{x}-\epsilon e^{i\theta'})$ we can establish
$$\sup\limits_{\epsilon>0}\sup\limits_{x\in D_\epsilon}|\frac{1}{4\pi^2}\int_{\mathds{S}_1}\int_{\mathds{S}_1}\ln\frac{1}{|1-(x+\epsilon e^{i\theta})(\overline{x}+e^{i\theta'})|}d\theta d\theta' +\ln\epsilon|<+\infty$$
Together we get the first part of the lemma.\\
The first inequality in the second part of the lemma comes as a direct consequence. The second inequality can be proved using a similar scaling relation as in the above proof.
\end{proof}

Now we establish another estimate concerning the process $Y_\epsilon$. Recall that $Y_\epsilon$ is the Gaussian process defined as $Y_\epsilon(u)=X_{\partial\mathds{D},\epsilon}(1-\epsilon u)-X_{\partial\mathds{D},\epsilon}(1)$ and $D(2,1)$ is the disk centered at 2 with radius $1$.
\begin{lemma}
For all $z,z'\in D(2,1)$,
$$\mathds{E}[(Y_\epsilon(z)-Y_\epsilon(z'))^2]\leq C|z-z'|$$
uniformly in $0<\epsilon\leq 1$.
\end{lemma}
\begin{proof}
It suffices to prove that uniformly in $\epsilon$,
$$|G_\epsilon(1-\epsilon z,1-\epsilon z)-G_\epsilon(1-\epsilon z,1-\epsilon z')|\leq C|z-z'| .$$
For the $\ln\frac{1}{|x-y|}$ part of $G$, it suffices to prove that the following function is Lipschitz in $r$ for $r\in[0,2]$
$$f(r)=\frac{1}{4\pi^2}\int_{\mathds{S}_1}\int_{\mathds{S}_1}\ln\frac{1}{|e^{i\theta}-re^{i\theta'}|}d\theta d\theta'$$
notice that $f(1)=0$. But we have already seen that $f(r)=0$ when $r\leq 1$ and this implies that $f(r)=\ln r$ when $r>1$.\\
As of the $\ln\frac{1}{|1-x\overline{y}|}$ part, we will write the difference as
$$\frac{1}{4\pi^2}\int_{\mathds{S}_1}\int_{\mathds{S}_1}\ln\frac{|1-(x+\epsilon e^{i\theta})((\overline{y}-\overline{x})+\overline{x}+\epsilon e^{i\theta'})|}{|1-(x+\epsilon e^{i\theta})(\overline{x}+\epsilon e^{i\theta'})|}d\theta d\theta'$$
where $x=1-\epsilon z$ and $y=1-\epsilon z'$. Then we note $t=\overline{\frac{y-x}{\epsilon}}$ and this becomes
$$\frac{1}{4\pi^2}\int_{\mathds{S}_1}\int_{\mathds{S}_1}\ln\frac{|1/\epsilon^2-(x/\epsilon+e^{i\theta})(t+\overline{x}/\epsilon+e^{i\theta'})|}{|1/\epsilon^2-(x/\epsilon+e^{i\theta})(\overline{x}/\epsilon+e^{i\theta'})|}d\theta d\theta'$$
As the derivative with respect to $t$ is continuous and uniformly bounded in $\epsilon$ for $|t|\leq 2$, our proof is complete.
\end{proof}

\subsection{Moment estimates on Gaussian multiplicative chaos}\label{momentestimates}

\subsubsection{Moment estimate for an insertion on the boundary}
To define properly unit volume Liouville measures in Corollary~\ref{unitvolume}, we need several moment estimates on the GMC measure on the unit disk that we study in this section.\\
The techniques of proof for different values $\gamma$ are not the same: we seperate the regimes $\gamma\in]0,\sqrt{2}[$ (Lemma~\ref{lemmomapp}) and $\gamma\in[\sqrt{2},2[$ (Lemma~\ref{lemmomapp2}). The main difference from the unit volume Liouville measure on the Riemann sphere (see \cite[Section 3.4]{DKRV}) is that we have to give moment bounds of same kind of non classical GMC measure near the boundary (with or without insertions at the boundary), where, as we will explain after, the $2d$-GMC measure collapses and behaves like a $1d$-GMC measure: analysing this transition carefully is the goal of Lemma~\ref{lemmomapp} and Lemma~\ref{lemmomapp2}.\\
~\\
We first look at the regime $\gamma\in]0,\sqrt{2}[$ and give a necessary and sufficient condition for the finiteness of the GMC measure on the unit disk.
\begin{lemma}\label{lemmomapp}
Let $\gamma \in ]0,\sqrt{2}[$. Consider $\beta < Q$ and $s \in \partial \D$. Then for all $p>0$, $\delta>0$,
\begin{equation}\label{eqmomapp}
\E \left [\left ( \int_{\D\cap B(s,\delta)}e^{\gamma\frac{\beta}{2}G(\cdot,s)} e^{\gamma X_{\partial \D }-\frac{\gamma^2}{2}\E[X_{\partial \D }^2]} g_P^{\frac{\gamma^2}{4}}\,d\lambda \right )^p  \right ]<+\infty
\end{equation}
if and only if $p< \frac{2}{\gamma^2}  \wedge \frac{1}{\gamma} (Q-\beta)$.
\end{lemma}

\begin{remark}\label{halfplanekernel}
In the following of this section, we will work with an exact scale-invariant Gaussian process on $\H$ with covariance
$$G_\H(x,y)=\ln\frac{R}{|x-y||x-\bar{y}|}.$$
This is not exactly the same kernel as the process we will consider in the half-plane geometry: they differ by at most a finite constant near $0$. However, by Kahane's inequality, they possess the same moment bound for Lemma~\ref{lemmomapp} and Lemma~\ref{lemmomapp2}. Furthermore, restricted to a small compact near $0$, $G_\H$ is indeed positive definite: one can first sample a log-correlated Gaussian field on $B(0,\epsilon)\cap\C$ with correlation $\ln^{+}\frac{1}{|x-y|}$ (this is indeed a definite-positive kernel, see \cite{vincent}) for a \textbf{fixed} small $\epsilon$, then define the Gaussian field on $B(0,\epsilon)\cap\H$ by assigning $\frac{X(z)+X(\overline{z})}{\sqrt{2}}$ at point $z$: this new field has correlation $G_\H$ on $B(0,\epsilon)\cap\H$. This is a standard procedure, see for example \cite[Section~3.2]{Sheff}.
\end{remark}

\proof
By rotational invariance of the problem, we can suppose that $s=-1$. Let $\psi(z)=\frac{z-i}{z+i}$  be the Cayley transform 
which maps the upper half-plane onto the unit disk and sends $0$ to $-1$. We set $G_\H(x,y)=G(\psi(x),\psi(y))$ and $X=X_{\partial\D}\circ \psi$ which has covariance $G_\H$.\\
In light of the above Remark~\ref{halfplanekernel}, we will work with the following form of $G_\H$:
$$G_\H(x,y)=\ln\frac{1}{|x-y||x-\bar{y}|},$$
we have that for $x,y\in \H$ and all $r\in]0,1[$,
\begin{equation}\label{exactscaling}
G_\H(rx,ry)= G_\H(x,y)+2\ln\frac{1}{r},
\end{equation}
which is an exact scaling relation.\\
With these notations, by conformal invariance, it is equivalent to study whether
\begin{equation}
\E \left [ \left (\int_{S}e^{\gamma\frac{\beta}{2}G_\H(z,0)} e^{\gamma X(z) -\frac{\gamma^2}{2}\E[X^2(z)]} \frac{1}{{\rm Im}(z)^{\frac{\gamma^2}{2}}} \, \lambda(dz) \right )^p \right ]<+\infty,
\end{equation}
where $S$ denotes some small square $[-\epsilon,\epsilon]\times[0,2\epsilon]$ in $\H$ such that $G_\H$ is definite positive on $S$.\\
$\epsilon$ will be fixed in the rest of the section, it is solely chosen to ensure that $G_{\H}$ is well-defined.\\
Let $X_{1/2^n}$ denotes a $2^{-n}$-regularisation of $X$ by convolution with a mollifier of compact support (for instance the $2^{-n}$-circle average) and note
\begin{equation}
J_n(B)=\int_{B}e^{\gamma\frac{\beta}{2}G_\H(z,0)} e^{\gamma X_{1/2^n}(z) -\frac{\gamma^2}{2}\E[X_{1/2^n}^2(z)]} \frac{1}{{\rm Im}(z)^{\frac{\gamma^2}{2}}} \, \lambda(dz)
\end{equation}
for all borelians $B$ in $\H$.
We prove that:
\begin{equation}\label{mesuresup}
\limsup\limits_{n}\E\left[J_n(S)^p\right]<+\infty,
\end{equation}
if and only if $p< \frac{2}{\gamma^2}  \wedge \frac{1}{\gamma} (Q-\beta)$. One then follows \cite[Section A.]{DKRV} to finish off the proof.\\

\vspace{0.2 cm}

\noindent In the following we denote by
\begin{equation}\label{scalingexponent}
\overline{\zeta}(p)=(2+\gamma^2/2)p-\gamma^2p^2
\end{equation}
the scaling exponent for the measure $J_n$.\\
We first study the quantity in equation~(\ref{mesuresup}) in the case $\beta=0$ (i.e. no insertion at the boundary), then we pass to the case of general $\beta$ by slightly modifying an argument in \cite[Section A.]{DKRV}, for which we briefly recall the details in the following.\\
~\\
\emph{The case $\beta=0$:}\\
Remark that $\frac{2}{\gamma^2}<\frac{Q}{\gamma}$ such that in the case $\beta=0$, we only need to prove that :
\begin{equation}
\limsup\limits_{n}\E\left[J_n(S)^p\right]<+\infty \quad \text{if and only if} \quad p< \frac{2}{\gamma^2}.
\end{equation}

\noindent One checks by Fubini that with $\gamma\in ]0,\sqrt{2}[$, $J_n(S)$ possesses finite first moment
\begin{equation}
\E \left [ \int_{S}e^{\gamma X(z) -\frac{\gamma^2}{2}\E[X^2(z)]} \frac{1}{{\rm Im}(z)^{\frac{\gamma^2}{2}}} \, \lambda(dz) \right ]<+\infty,
\end{equation}
so we will only focus on values of $p>1$ in the regime $\gamma\in ]0,\sqrt{2}[$.\\
We will decompose $S$ into three parts and use scaling relations to explore the multifractal structure of $X$ and the associated mesure $J_n$. Let
$$S_1=[-\epsilon,0]\times[0,\epsilon], S_2=[0,\epsilon]\times[0,\epsilon], S_3=[-\epsilon,\epsilon]\times[\epsilon,2\epsilon]$$
and remark that, by scaling
\begin{equation}\label{scalingexact}
\E\left[J_{n+1}(S_1)^p\right]=2^{-\overline{\zeta}(p)}\E\left[J_n(S)^p\right]
\end{equation}
where $\overline{\zeta}(p)$ is the half-plane scaling exponent for the kernel $G_\H$ as defined in equation~(\ref{scalingexponent}).\\
~\\
Before entering the proof, we first provide three estimations on the correlations of $J_n$. The motivation behind these estimations is that we want to have some kind of ``decorrelation'' inequalities in order to compare some cross-terms with leading terms in polynomial expansions.\\
Let $p,q>0$ and $A=[a^l_1,a^r_1]\times[a^l_2,a^r_2]$, $B=[b^l_1,b^r_1]\times[b^l_2,b^r_2]$ two rectangles in $S$. We call $A$ and $B$ adjacent if $a^r_1=b^l_1$ (the right boundary of $A$ is at the same level as the left boundary of $B$).
\begin{lemma}[Cutoff estimation: non-adjacent case]\label{cutoffnonadjacent}~\\
Let $p,q>0$. Suppose that $a^r_1<b^l_1$, so that $dist(A,B)=\delta>0$. Then
\begin{equation}
\mathds{E}[J_n(A)^p]\mathds{E}[J_n(B)^q]\leq \mathds{E}[J_n(A)^p J_n(B)^q]\leq \delta^{-2pq\gamma^2}\mathds{E}[J_n(A)^p]\mathds{E}[J_n(B)^q].
\end{equation}
\end{lemma}
\begin{lemma}[Cutoff estimation: adjacent to non-adjacent case]\label{cutoffnontooui}~\\
Let $p,q>0$. Suppose that $a^r_1=b^l_1$. Let $B_t=[b^l_1+t,b^r_1+t]\times[b^l_2,b^r_2]$ be a translation of $B$ on the first coordinate with $t>0$ (so that we are moving away from $A$). Then there exists some constant $C$ that does not depend on $n$ such that
\begin{equation}
\mathds{E}[J_n(A)^p J_n(B_t)^q]\leq C\mathds{E}[J_n(A)^p J_n(B)^q].
\end{equation}
\end{lemma}
\begin{lemma}[Cutoff estimation: symmetric case]\label{cutoffsymmetry}~\\
Let $p,q>0$. Let $H_L=\{x\in\R^2; x_1=L\}$ the hyperplan of $\R^2$ with the first coordiante equal to $L$. Let $\overline{B}$ be the symmetry of $B$ with respect to $H=H_{a^r_1}=H_{b^l_1}$. Then
\begin{equation}
\mathds{E}[J_n(A)^p J_n(B)^q]\leq\mathds{E}[J_n(A)^p J_n(\overline{B})^q].
\end{equation}
\end{lemma}
\begin{lemma}[Cutoff estimation: adjacent case]\label{cutoffadjacent}~\\
Let $p,q>0$. Suppose that $a^r_1=b^l_1$. Let $B_t=[b^l_1+t,b^r_1+t]\times[b^l_2,b^r_2]$ a translation of $B$ on the first coordinate with $t<0$ (so that we are moving towards $A$) such that $[b^l_1+t,b^r_1+t]\subset[a^l_1,a^r_1]$. Then there exists some constant $C$ that does not depend on $n$ such that
\begin{equation}
\mathds{E}[J_n(A)^p J_n(B)^q]\leq C\mathds{E}[J_n(A)^p J_n(B_t)^q].
\end{equation}
\end{lemma}
\noindent For continuity of lecture, we postpone the proofs of these lemmas until the end of this section. We now explain how to finish the $\beta=0$ case using these estimations.\\
~\\
$\bullet$ Suppose $\limsup\limits_{n}\E\left[J_n(S)^p\right]<+\infty$ and let us prove that $p<2/\gamma^2$.\\
If $p\geq 4/\gamma^2$, we already know since Kahane that the $p$-th moment of a $2d$-GMC measure explodes. Thus, when $p\geq 4/\gamma^2$, $\E\left[J_n(S)^p\right]$ explodes in the limit since $\E\left[J_n(S_3)^p\right]$ does.\\
Let $1\leq p< 4/\gamma^2$ so that $\E\left[J_n(S_3)^p\right]$ converges to some finite quantity by Kahane. By super-additivity and scale invariance (equation~(\ref{scalingexponent})), we have
\begin{align*}
&\E\left[J_n(S)^p\right]\\
\geq&\E\left[J_n(S_1)^p\right]+\E\left[J_n(S_2)^p\right]+\E\left[J_n(S_3)^p\right]\\
\geq&2^{1-\overline{\zeta}(p)}\E\left[J_{n-1}(S)^p\right]+C
\end{align*}
for some constant $C>0$ independent of $n$. Since $1-\overline{\zeta}(p)\geq 0$ when $p\geq 2/\gamma^2$, necessarily $p<2/\gamma^2$ if $\E\left[J_n(S)^p\right]$ is uniformly bounded in $n$.\\
~\\
$\bullet$ Suppose $p<2/\gamma^2$ and prove $\limsup\limits_{n}\E\left[J_n(S)^p\right]<+\infty$.\\
We will show this by induction on $p$: suppose $\limsup\limits_{n}\E\left[J_n(S)^k\right]<+\infty$ for some integer $k$ and we prove $\limsup\limits_{n}\E\left[J_n(S)^p\right]<+\infty$ for all $p<2/\gamma^2$ such that $k<p\leq k+1$. We want to point out that this is a refinement of the arguments in \cite{KahPey}.\\
First we use Muirhead's inequality (see Corollary~\ref{ExpansionInequality}) to write
\begin{align*}
&\E\left[J_n(S_1+S_2)^p\right]\\
\leq&\E\left[J_n(S_1)^p\right]+\E\left[J_n(S_2)^p\right]+C(\E[J_n(S_1)^k J_n(S_2)^{p-k}]+\E[J_n(S_1)^{p-k} J_n(S_2)^{k}])\\
={}&2\E\left[J_n(S_1)^p\right]+C\E[J_n(S_1)^k J_n(S_2)^{p-k}]
\end{align*}
where we used the symmetry of $J_n$ on $S_1$ and $S_2$. Now consider some large integer that we denote by $\delta^{-1}$ (so that $\delta$ is small), and cut $S_1$ into $\delta^{-1}$ equal parts of the form
$$S^i_1=[-i\delta\epsilon,-(i-1)\delta\epsilon]\times[0,\epsilon],\quad i=1,\dots,\delta^{-1}.$$
We also cut $S_2$ into two parts
$$S^l_2=[0,\delta\epsilon]\times[0,1],\quad S^r_2=[\delta\epsilon,\epsilon]\times[0,\epsilon].$$
Notice that $S^l_2$ and $S_1$ are adjacent in the sense of Lemma~\ref{cutoffadjacent}. It follows from the same lemma, since $S^i_1$ is a horizontal translation of $S^l_2$, that for all $i=1,\dots,\delta^{-1}$,
$$\E[J_n(S_1)^k J_n(S^l_2)^{p-k}]\leq C\E[J_n(S_1)^k J_n(S^i_1)^{p-k}]$$
for some $C$ independent of $n$.\\
Since $0<p-k\leq 1$, using Jensen's inequlity for the concave function $x\mapsto x^{p-k}$,
\begin{align*}
&\E[J_n(S_1)^p]\\
\geq&\E[J_n(S_1)^k J_n(S_1)^{p-k}]\\
\geq&\delta^{1-(p-k)}\sum\limits_{i}\E[J_n(S_1)^k J_n(S^i_1)^{p-k}]\\
\geq&C\delta^{-(p-k)}\E[J_n(S_1)^k J_n(S^l_2)^{p-k}].
\end{align*}
On the other hand, we can compare $\E[J_n(S_1)^k J_n(S^r_2)^{p-k}]$ and $\E[J_n(S_1)^p]$ via Lemma~\ref{cutoffnonadjacent}: applying this lemma to $S^r_2$ and $S_1$,
\begin{align*}
&\E[J_n(S_1)^k J_n(S^r_2)^{p-k}]\\
\leq&C\delta^{-2p(p-k)\gamma^2}\E[J_n(S_1)^k]\E[J_n(S^r_2)^{p-k}]\\
\leq&C\delta^{-2p(p-k)\gamma^2}.
\end{align*}
where we used the finiteness for the $k$-th moment and the $(p-k)$-th moment thanks to induction hypothesis, and $C$ always independent of $n$.\\
Gathering all pieces of information above, we get by sub-additivity,
\begin{align*}
&\E[J_n(S_1+S_2)^p]\\
\leq&2\E\left[J_n(S_1)^p\right]+C\E[J_n(S_1)^k J_n(S_2)^{p-k}]\\
\leq&2\E\left[J_n(S_1)^p\right]+C\E[J_n(S_1)^k J_n(S^l_2)^{p-k}]+C\E[J_n(S_1)^k J_n(S^r_2)^{p-k}]\\
\leq&(2+C\delta^{p-k})\E\left[J_n(S_1)^p\right]+C\delta^{-2p(p-k)\gamma^2}
\end{align*}
where $C$ changes from line to line but is always independent of $n$.\\
Since $1\leq p<2/\gamma^2$, one checks that it is possible to choose (independently of $n$) a $\delta$ such that $\alpha=2^{-\overline{\zeta}(p)}(2+C\delta^{p-k})<1$. With this choice, we get
\begin{equation}
\E[J_n(S_1+S_2)^{p}]\leq\alpha\E[J_{n-1}(S)^p]+C.
\end{equation}
Finally, applying Minkowski's inequality to $J_n(S)$ with the cutting $S_1,S_2,S_3$, we get, for $1\leq p<2/\gamma^2$,
\begin{align*}
&\E[J_n(S)^p]^{1/p}\\
\leq&(\E\left[J_n(S_1+S_2)^p\right])^{1/p}+\E[J_n(S_3)^p]^{1/p}\\
\leq&\left(\alpha\E\left[J_{n-1}(S)^p\right]+C\right)^{1/p}+C
\end{align*}
again, the estimate on $S_3$ is classical since Kahane, and $C>0$ independent of $n$.\\
By sub-additivity,
\begin{align*}
&\E[J_n(S)^p]^{1/p}\\
\leq&\alpha^{1/p}\E\left[J_{n-1}(S)^p\right]^{1/p}+C^{1/p}+C.
\end{align*}
With $\alpha<1$, this ensures that the sequence $\E[J_n(S)^p]^{1/p}$ is uniformly bounded in $n$ (and so is $\E[J_n(S)^p]$), thus completes our proof for the $\beta=0$ case.\\
~\\
\emph{The case $\beta$ general:}
We now add an insertion of weight $\beta$ at the boundary of $\H$ at $0$. We study directly the field $X$ and exploit its exact scaling relation as in \cite[Section A.]{DKRV}.\\
Let
\begin{equation}
I(B)=\E \left [\left ( \int_{B} e^{\gamma\frac{\beta}{2}G_{\H}(\cdot,s)} e^{\gamma X_{\partial \D }-\frac{\gamma^2}{2}\E[X_{\partial \D }^2]} g_P^{\frac{\gamma^2}{4}}\,d\lambda \right )^p  \right ]
\end{equation}
for borelians $B$ in $\H$.\\
Let us fix a small $\epsilon$ such that $G_\H$ on $B(0,\epsilon)$ is well-defined.\\
Let $A=\H\cap B(0,\epsilon)$ and $A_i=\H\cap(B(0,2^{-i}\epsilon)\backslash B(0,2^{-n-1}\epsilon))$ for all $i=0,1,\dots$, a partition of $A$ into half-annulis. We know by the case $\beta=0$ that for all $i$, $I(A_i)<\infty$ iff $p<\frac{2}{\gamma^2}$. Furthermore, by exact-scaling,
$$I(A_i)=2^{p\gamma\beta}2^{-\overline{\zeta}(p)}I(A_{i-1})$$
and remark that $f(p)=p\gamma\beta-\overline{\zeta}(p)$ changes sign at $p=\frac{Q-\beta}{\gamma}$.\\
Suppose that $p>1$, then by super-additivity,
\begin{equation}
I(A)\geq \sum\limits_{i}I(A_i),
\end{equation}
from which one deduces that $I(A)<\infty$ implies $p<\frac{Q-\beta}{\gamma}$. On the other hand, by Minkowski's inequality,
\begin{equation}
I(A)^{1/p}\leq \sum\limits_{i}I(A_i)^{1/p},
\end{equation}
from which one deduces that $p<\frac{Q-\beta}{\gamma}$ implies $I(A)<\infty$.
\qed\\
~\\
We now turn our attention to the regime $\gamma\in[\sqrt{2},2[$. Here we deal with moments of order $0<p<1$, which is something quite unusual in the GMC theory: see the remark below the following lemma.
\begin{lemma}\label{lemmomapp2}
Let $\gamma \in [\sqrt{2},2[$. Consider $\beta < Q$ and $s \in \partial \D$. Then:\\
i)) For all $p>0$, $\delta>0$,
\begin{equation*}
\E \left [\left ( \int_{\D\cap B(s,\delta)}e^{\gamma\frac{\beta}{2}G(\cdot,s)} e^{\gamma X_{\partial \D }-\frac{\gamma^2}{2}\E[X_{\partial \D }^2]} g_P^{\frac{\gamma^2}{4}}\,d\lambda \right )^p  \right ]<+\infty
\end{equation*}
if $p< \frac{2}{\gamma^2}  \wedge \frac{1}{\gamma} (Q-\beta)$.\\
ii) Conversely, if
\begin{equation*}
\E \left [\left ( \int_{\D\cap B(s,\delta)}e^{\gamma\frac{\beta}{2}G(\cdot,s)} e^{\gamma X_{\partial \D }-\frac{\gamma^2}{2}\E[X_{\partial \D }^2]} g_P^{\frac{\gamma^2}{4}}\,d\lambda \right )^p  \right ]<+\infty
\end{equation*}
then $p\leq (\frac{1}{2}+\frac{1}{\gamma^2})  \wedge \frac{1}{\gamma} (Q-\beta)$.
\end{lemma}
\begin{remark}
It is plausible to think that we can improve the bound in ii) to obtain the same moment bound as in the $\gamma\in[0,\sqrt{2}[$ regime. This would require some sharper estimations that are not studied in the classical GMC theory -- because we are dealing with moments smaller than $1$, which always exist for classical GMC measures (since the first moment is either finite (in sub-critical phase) or zero if the field degenerates (in critical or super-critical case) for classical GMC measures).
\end{remark}
\proof We put ourselves in the same settings as in the proof of the previous lemma. In particular we work with the half-plane representation. We also follow the same notations.\\
$\bullet$ We first look at assertion i).\\
\emph{The case $\beta=0$:} the bound in the $\beta=0$ is reduced to $p<\frac{2}{\gamma^2}<1$.\\
First suppose that $p<\frac{2}{\gamma^2}$ and prove that the $p$-th moment is finite. Since $p<\frac{2}{\gamma^2}\leq 1$, we can make use the sub-addivitiy inequality and exact scale-invariance to write
\begin{align*}
&\E[J_{n+1}(S)^p]\\
\leq&\E[J_{n+1}(S_1+S_2+S_3)^p]\\
\leq&\E[J_{n+1}(S_1)^p]+\E[J_{n+1}(S_2)^p]+\E[J_{n+1}(S_3)^p]\\
\leq&2^{1-\overline{\zeta}(p)}\E[J_{n}(S)^p]+C
\end{align*}
where $C$ is some finite constant independent of $n$.\\
With the assumption that $\gamma<2$ we have $\frac{1}{2}<\frac{2}{\gamma^2}$. Choose some $p'>p$ such that $\frac{1}{2}<p'<\frac{2}{\gamma^2}$ and observe that $1-\overline{\zeta}(p')<0$, thus from the above relation, the $p'$-th moment of the measure $J(S)$ exists and is finite (because we would have $2^{1-\overline{\zeta}(p')}<1$). Since $p<p'$, the $p$-th moment of the measure $J(S)$ exists and is finite.\\
\emph{The case $\beta$ general:}\\
For completeness, we sketch a proof here although it is essentially the same as in \cite[Section A.]{DKRV} and as the proof for general boundary $\beta$-insertions in the $\gamma\in(0,\sqrt{2})$ regime above.\\
Let $p<1$ and
\begin{equation}
I(B)=\E \left [\left ( \int_{B} e^{\gamma\frac{\beta}{2}G_{\H}(\cdot,s)} e^{\gamma X_{\partial \D }-\frac{\gamma^2}{2}\E[X_{\partial \D }^2]} g_P^{\frac{\gamma^2}{4}}\,d\lambda \right )^p  \right ]
\end{equation}
for borelians $B$ in $\H$.\\
Let us fix a small $\epsilon$ such that $G_\H$ on $B(0,\epsilon)$ is well-defined.\\
Let $A=\H\cap B(0,\epsilon)$ and $A_i=\H\cap(B(0,2^{-i}\epsilon)\backslash B(0,2^{-n-1}\epsilon))$ for all $i=0,1,\dots$, a partition of $A$ into half-annulis. We know by the case $\beta=0$ that for all $i$, $I(A_i)<\infty$ if $p<\frac{2}{\gamma^2}$. Furthermore, by exact-scaling,
$$I(A_i)=2^{p\gamma\beta}2^{-\overline{\zeta}(p)}I(A_{i-1})$$
and remark that $f(p)=p\gamma\beta-\overline{\zeta}(p)$ changes sign at $p=\frac{Q-\beta}{\gamma}$.\\
Suppose that $p<1$, then by sub-additivity,
\begin{equation}
I(A)\leq \sum\limits_{i}I(A_i),
\end{equation}
from which one deduces that $p<\frac{Q-\beta}{\gamma}$ implies $I(A)<\infty$.\\
~\\
$\bullet$ We now turn our attention to ii).\\
\emph{The case $\beta=0$:}\\
Since $p<\frac{2}{\gamma^2}\leq 1$, we can make use the sub-addivitiy inequality, exact scale-invariance, then Jensen's inequality to write
\begin{align*}
&\E[J_{n+1}(S)^p]\\
\geq&\E[J_{n+1}(S_1+S_2)^p]\\
\geq&2^{p-1}(\E[J_{n+1}(S_1)^p]+\E[J_{n+1}(S_2)^p])\\
\geq&2^p2^{-\overline{\zeta}(p)}\E[J_{n}(S)^p]
\end{align*}
Studying the sign of $p-\overline{\zeta}(p)$ gives us the bound $p\leq\frac{1}{2}+\frac{1}{\gamma^2}$ if the sequence $\E[J_{n}(S)^p]$ is uniformly bounded in $n$.\\
\emph{The case $\beta$ general:}\\
One can check that the same argument above as in assertion i) goes through by replacing sub-additivity by Minkowski's inequality for $p<1$.
\qed\\
~\\
We now gather the sufficient conditions in Lemma~\ref{lemmomapp} and Lemma~\ref{lemmomapp2} to define the unit volume measure of the disk.
\begin{corollary}\label{relaxseiberg}
Suppose $\alpha_i,\beta_j<Q$ for all $i,j$. The random variable $Z_0(\D)$ has a moment of order $\frac{Q-\sum_i\alpha_i-\frac{1}{2}\sum_j\beta_j}{\gamma}$ if 
\begin{equation}
Q-\sum_i\alpha_i-\frac{1}{2}\sum_j\beta_j<\frac{2}{\gamma}\wedge2\min\limits_{i}(Q-\alpha_i)\wedge\min\limits_{j}(Q-\beta_j).
\end{equation}
In particular, under these conditions, the unit volume measure given by Corollary \ref{unitvolume} is well defined.
\end{corollary}
\proof Let
$$p=\frac{Q-\sum_i\alpha_i-\frac{1}{2}\sum_j\beta_j}{\gamma}.$$
From Lemma~\ref{lemmomapp} and Lemma~\ref{lemmomapp2}, the random variable $Z_0(S)$ has a moment of order $p$ if $S$ is a small open neighborhood of any boundary point, not containing any bulk insertions, under the condition
$$p<\frac{2}{\gamma^2}\wedge \frac{1}{\gamma}\min\limits_{j}(Q-\beta_j).$$
Since the boundary of $\D$ is finite, one can find a covering of a small open neighborhood of $\partial\D$ that does not contain any bulk insertions using a finite number of such $S$. On the complementary of this neighborhood, from the results of unit volume Liouville measure on the Riemann sphere (\cite[Section 3.4]{DKRV}), we know that $Z_0$ has a moment of order $p$ under the condition
$$p<\frac{4}{\gamma^2}\wedge \frac{2}{\gamma}\min\limits_{i}(Q-\alpha_i).$$
Combining these two considerations we get that $Z_0(\D)$ has a moment of order $p$ if
$$p<\frac{2}{\gamma^2}\wedge \frac{2}{\gamma}\min\limits_{i}(Q-\alpha_i)\wedge \frac{1}{\gamma}\min\limits_{j}(Q-\beta_j).$$
Replacing $p$ by $\frac{Q-\sum_i\alpha_i-\frac{1}{2}\sum_j\beta_j}{\gamma}$ yields the corollary.
\qed\\
~\\
We now give proofs of the cutoff estimations (Lemma~\ref{cutoffnonadjacent} to Lemma~\ref{cutoffadjacent}). The idea is to construct auxilary Gaussian processes and use a continous version of Slepian's lemma (Proposition~\ref{SlepianLemma} below, see \cite[Theorem~3]{Zeitouni} for a general formulation) to compare their functionals.\\
\begin{figure}[h]
\centering
\begin{tikzpicture}
\draw (0,0) -- (0,2)  -- (2,2) -- (2,0) -- cycle;
\fill [color=black] (2,1) -- (3,1)  -- (3,3) -- (2,3) -- cycle;
\fill [color=gray] (2,1) -- (1,1)  -- (1,3) -- (2,3) -- cycle;

\draw (5,0) -- (5,2)  -- (7,2) -- (7,0) -- cycle;
\fill [color=black] (7,1) -- (8,1)  -- (8,3) -- (7,3) -- cycle;
\fill [color=gray] (8.1,1) -- (9.1,1)  -- (9.1,3) -- (8.1,3) -- cycle;

\draw (11,0) -- (11,2)  -- (13,2) -- (13,0) -- cycle;
\fill [color=black] (13,1) -- (14,1)  -- (14,3) -- (13,3) -- cycle;
\fill [color=gray] (11.5,1) -- (12.5,1)  -- (12.5,3) -- (11.5,3) -- cycle;
\end{tikzpicture}
\caption{Cut-off estimations: $A$ is illustrated by a white box, $B$ by a black box. The grey box denotes the box $B$ after some linear transformation.\\Left: the symmetry case. Middle: $B$ is moving away from $A$. Right: $B$ is moving inside of $A$.}
\end{figure}
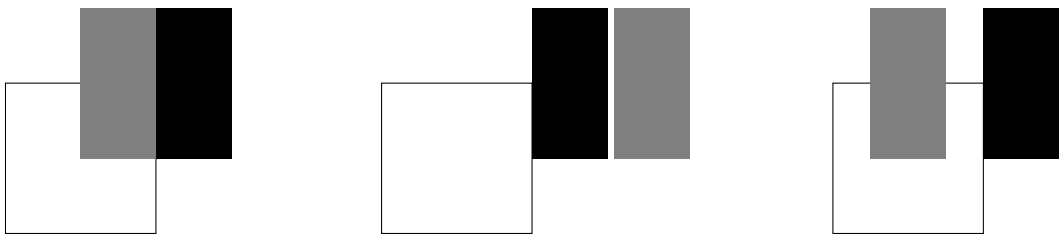

\begin{proof}[Proof of Lemma~\ref{cutoffnonadjacent}]~\\
Let $p,q>0$. Let $(\mathbf{x},\mathbf{y})$ be a Gaussian vector and define $$f\left(x_1,\dots,x_k,y_1\dots,y_l\right)=\left(\sum\limits_{1\leq i\leq k} p_i e^{\gamma x_i}\right)^p\left(\sum\limits_{1\leq j\leq l} q_j e^{\gamma y_j}\right)^q$$
with $p_i,q_j$ some non-negative weights. One verifies that
$$\partial_{x_i,y_j}f\geq 0$$
for all couples $(x_i,y_j)$.\\
Let $\overline{X}$ be a Gaussian field defined on the same set with the same kernel as $X$, independent of $X$. Let $\overline{M}_n$ be the regularised GMC measure associated with $\overline{X}$. Let $(\alpha_i)$ (resp. $(\beta_j)$) be a discret net approximating $A$ (resp. $B$).\\
Now consider two sets of Gaussian vectors:\\
-- $(\mathbf{x},\mathbf{y})$ with $x_i=X(\alpha_i)$, $y_j=X(\beta_j)$;\\
-- $(\mathbf{\overline{x}},\mathbf{\overline{y}})$ with $\overline{x}_i=X(\alpha_i)$, $\overline{y}_j=\overline{X}(\beta_j)$.\\
The difference between these two vectors is that by changing the $y$ component, the $A$ part and the $B$ part of the second vector becomes independent. It is clear that these two vectors possess the same correlation structure except that for all couples $(x_i,y_j)$ (resp. $(\overline{x}_i,\overline{y}_j)$), and we have
$$\mathds{E}[\overline{x}_i\overline{y}_j]\leq\mathds{E}[x_iy_j].$$
This allows us to apply Slepian's lemma (see Proposition~\ref{SlepianLemma}) to conclude that
$$\mathds{E}[f(\overline{x}_i,\overline{y}_j)]\leq\mathds{E}[f(x_i,y_j)].$$
Taking the Riemann integral limit, we obtain the first inequality in Lemma~\ref{cutoffnonadjacent}.\\
~\\
For the other inequality, we apply the same idea: choose three mutually independent standard normal Gaussian distributions $N, N', \overline{N}$, independent of $X$ and $\overline{X}$ and consider two sets of Gaussian vectors:\\
-- $(\mathbf{x},\mathbf{y})$ with $x_i=X(\alpha_i)+rN$, $y_j=X(\beta_j)+rN'$;\\
-- $(\mathbf{\overline{x}},\mathbf{\overline{y}})$ with $\overline{x}_i=X(\alpha_i)+r\overline{N}$, $\overline{y}_j=\overline{X}(\beta_j)+r\overline{N}$;\\
where $r$ is a constant we choose later. Observe that if $r^2\geq-2\ln\delta$, then these two vectors possess the same correlation structure except that for all couples $(x_i,y_j)$ (resp. $(\overline{x}_i,\overline{y}_j)$), we have
$$\mathds{E}[x_iy_j]\leq\mathds{E}[\overline{x}_i\overline{y}_j].$$
Now we apply Slepian's lemma (see Proposition~\ref{SlepianLemma}) to conclude that
$$\mathds{E}[f(x_i,y_j)]\leq \mathds{E}[f(\overline{x}_i,\overline{y}_j)].$$
Taking the Riemann integral limit, we obtain
$$\mathds{E}[e^{pr\gamma N}e^{qr\gamma N'}M_n(A)^p M_n(B)^q]\leq\mathds{E}[e^{(p+q)r\gamma\overline{N}}M_n(A)^p\overline{M}_n(B)^q].$$
Choosing $r^2=-2\ln\delta$, we get
$$\mathds{E}[M_n(A)^p M_n(B)^q]\leq \delta^{-2pq\gamma^2}\mathds{E}[M_n(A)^p]\mathds{E}[M_n(B)^q].$$
This gives the second inequality in Lemma~\ref{cutoffnonadjacent}.
\end{proof}

\begin{proof}[Proof of Lemma~\ref{cutoffnontooui}]~\\
We apply the same idea as in the proof of Lemma~\ref{cutoffnonadjacent} with the same function $f$. We only give hints about the Gaussian vectors we consider: the rest follows the same lines as in the above proof.\\
Suppose $(\alpha_i)$ approximate $A$ and $(\beta_j)$ approximate $B$. Let $\beta^t_j=\beta_j+te_1$ where $e_1$ is the first coordinate. Consider two sets of Gaussian vectors:\\
-- $(\mathbf{x},\mathbf{y})$ with $x_i=X(\alpha_i)$, $y_j=X(\beta_j)$;\\
-- $(\mathbf{\overline{x}},\mathbf{\overline{y}})$ with $\overline{x}_i=X(\alpha_i)$, $\overline{y}_j=X(\beta^t_j)$.\\
These two vectors possess the same correlation structure except that for all couples $(x_i,y_j)$ (resp. $(\overline{x}_i,\overline{y}_j)$), and we have
$$\mathds{E}[\overline{x}_i\overline{y}_j]\leq\mathds{E}[x_iy_j].$$
In particular, the constant can be chosen to be $C=1$ for all $t>0$.
\end{proof}

\begin{proof}[Proof of Lemma~\ref{cutoffsymmetry}]~\\
We apply the same idea as in the proof of Lemma~\ref{cutoffnonadjacent} with the same function $f$. We only give hints about the Gaussian vectors we consider: the rest follows the same lines as in the above proof.\\
Suppose $(\alpha_i)$ approximate $A$ and $(\beta_j)$ approximate $B$. Let $\overline{\beta}_j$ be the symmetry of $\beta_j$ with respect to the hyperplan $H=H_{a^r_1}=H_{b^l_1}$. Consider two sets of Gaussian vectors:\\
-- $(\mathbf{x},\mathbf{y})$ with $x_i=X(\alpha_i)$, $y_j=X(\beta_j)$;\\
-- $(\mathbf{\overline{x}},\mathbf{\overline{y}})$ with $\overline{x}_i=X(\alpha_i)$, $\overline{y}_j=X(\overline{\beta}_j)$.\\
These two vectors possess the same correlation structure except that for all couples $(x_i,y_j)$ (resp. $(\overline{x}_i,\overline{y}_j)$), and we have
$$\mathds{E}[\overline{x}_i\overline{y}_j]\geq\mathds{E}[x_iy_j].$$
Applying Slepian's lemma gives us the desired inequality.
\end{proof}
\begin{proof}[Proof of Lemma~\ref{cutoffadjacent}]~\\
Suppose $t<0$ and such that $[b^l_1+t,b^r_1+t]\subset[a^l_1,a^r_1]$. For simplicity we left out the second coordinate (i.e. $a_2, b_2$) in the following calculations: for example, the rectangle $A=[a^l_1,a^r_1]\times[a^l_2,a^r_2]$ will be simply denoted by $[a^l_1,a^r_1]$. We do not change the second cooordinate of $A$ (resp. $B$) in the following calculation where all linear shifts and symmetries preserve the second coordinate.\\
Because for all $x,y\geq 0$, $x^p+y^p$ and $(x+y)^p$ only differ by some positive multiplicative constant (which only depends on $p$), we have, with $C$ denoting some constant independent of $n$ which might change from line to line:
\begin{align*}
&\mathds{E}[M_n(A)^p M_n(B_t)^q]\\
\stackrel{\textcircled{\small a}}{\geq}&C(\mathds{E}[M_n([a^l,b^l+t])^p M_n([b^l+t,b^r+t])^q]+\mathds{E}[M_n([b^l+t,a^r])^p M_n([b^l+t,b^r+t])^q])\\
\stackrel{\textcircled{\small b}}{=}{}&C(\mathds{E}[M_n([a^l-t,b^l])^p M_n([b^l,b^r])^q]+\mathds{E}[M_n([b^l+t,a^r])^p M_n([a^r+b^l-b^r,a^r])^q])\\
\stackrel{\textcircled{\small c}}{\geq}&C(\mathds{E}[M_n([a^l-t,a^r])^p M_n([b^l,b^r])^q]+\mathds{E}[M_n([b^l+t,a^r])^p M_n([b^l,b^r])^q])\\
\stackrel{\textcircled{\small d}}{\geq}&C(\mathds{E}[M_n([a^l-t,a^r])^p M_n([b^l,b^r])^q]+\mathds{E}[M_n([a^l,a^l-t])^p M_n([b^l,b^r])^q])\\
\stackrel{\textcircled{\small e}}{\geq}&C\mathds{E}[M_n([a^l,a^r])^p M_n([b^l,b^r])^q]\\
={}&C\mathds{E}[M_n(A)^p M_n(B)^q]
\end{align*}
where\\
-- \textcircled{\small a} comes from the comparaison between $x^p+y^p$ and $(x+y)^p$ above;\\
-- \textcircled{\small b} comes from the horizontal translation invariance of $X$ (for the first term) and symmetry of $X$ (for the second term);\\
-- \textcircled{\small c} comes from the symmetry case Lemma~\ref{cutoffsymmetry} (applied to the second term);\\
-- \textcircled{\small d} comes from the adjacent-to-non-adjacent case Lemma~\ref{cutoffnontooui} (applied to the second term);\\
-- \textcircled{\small e} comes from (again) the comparaison between $x^p+y^p$ and $(x+y)^p$ above.\\
This concludes Lemma~\ref{cutoffadjacent}.\\
One can find a pictural representation of each inequality above in Figure~\ref{lolpic} below.
\end{proof}
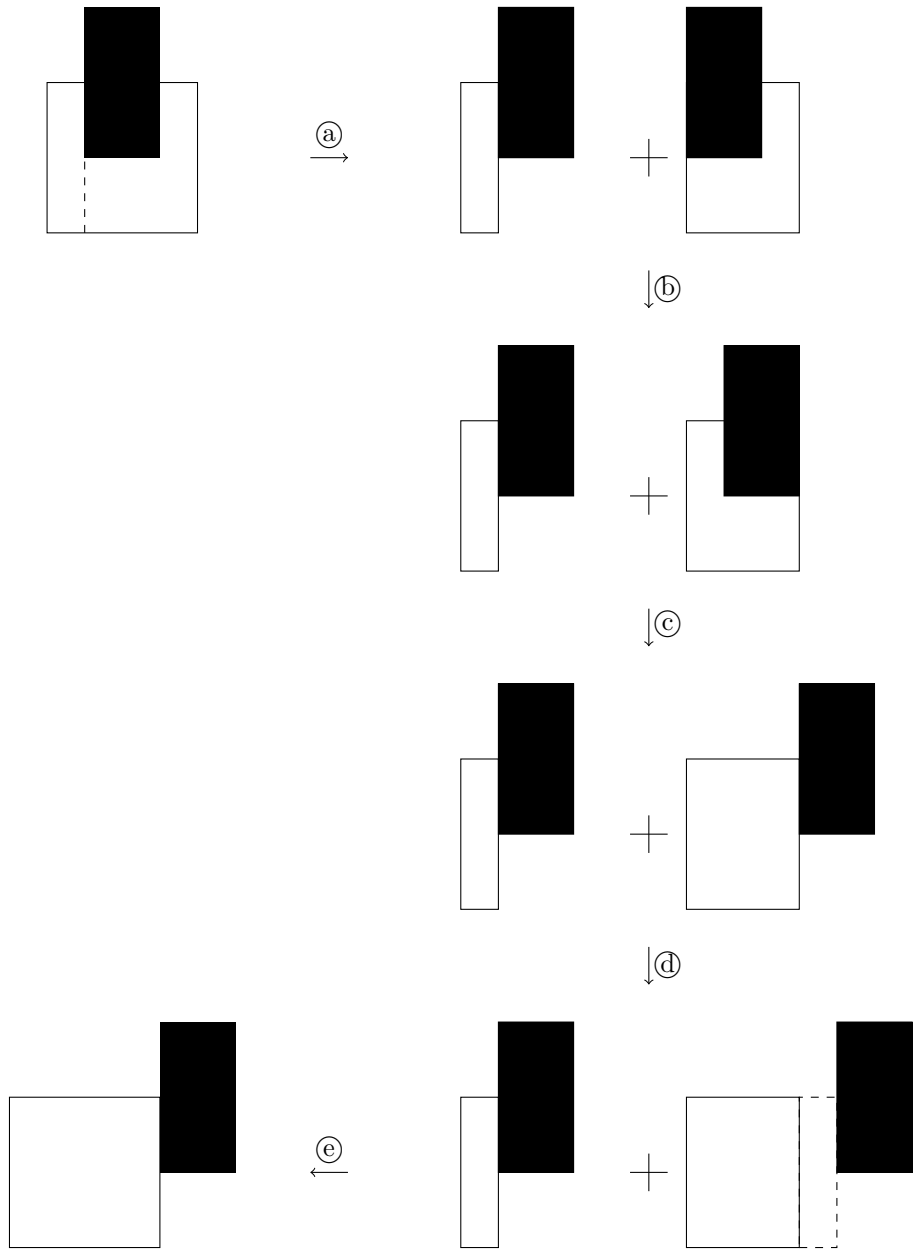
\begin{figure}[p]
\centering
\begin{tikzpicture}
\draw (-2,0) -- (-2,2)  -- (0,2) -- (0,0) -- cycle;
\draw [dashed] (-1.5,0) -- (-1.5,2);
\fill [color=black] (-1.5,1) -- (-0.5,1)  -- (-0.5,3) -- (-1.5,3) -- cycle;

\draw [->] (1.5,1) -- (2,1);
\node[draw=none] at (1.75,1.25) {\textcircled{\small a}};

\draw (6.5,0) -- (6.5,2)  -- (8,2) -- (8,0) -- cycle;
\draw [fill=black] (6.5,1) -- (7.5,1)  -- (7.5,3) -- (6.5, 3) -- cycle;
\draw (3.5,0) -- (3.5,2)  -- (4,2) -- (4,0) -- cycle;
\draw [fill=black] (4, 1) -- (5, 1)  -- (5, 3) -- (4, 3) -- cycle;
\draw (6, 1.25) -- (6,0.75);
\draw (5.75,1) -- (6.25,1);

\draw [->] (6,-0.5) -- (6,-1);
\node[draw=none] at (6.25,-0.75) {\textcircled{\small b}};

\draw (6.5, -4.5) -- (6.5, -2.5)  -- (8, -2.5) -- (8, -4.5) -- cycle;
\draw [fill=black] (7, -3.5) -- (8, -3.5)  -- (8, -1.5) -- (7, -1.5) -- cycle;
\draw (3.5, -4.5) -- (3.5, -2.5)  -- (4, -2.5) -- (4, -4.5) -- cycle;
\draw [fill=black] (4, -3.5) -- (5, -3.5)  -- (5, -1.5) -- (4, -1.5) -- cycle;
\draw (6,-3.25) -- (6,-3.75);
\draw (5.75,-3.5) -- (6.25,-3.5);

\draw [->] (6,-5) -- (6,-5.5);
\node[draw=none] at (6.25,-5.25) {\textcircled{\small c}};

\draw (6.5, -9) -- (6.5, -7)  -- (8, -7) -- (8, -9) -- cycle;
\draw [fill=black] (8, -8) -- (9, -8)  -- (9, -6) -- (8, -6) -- cycle;
\draw (3.5, -9) -- (3.5, -7)  -- (4, -7) -- (4, -9) -- cycle;
\draw [fill=black] (4, -8) -- (5, -8)  -- (5, -6) -- (4, -6) -- cycle;
\draw (6,-7.75) -- (6,-8.25);
\draw (5.75,-8) -- (6.25,-8);

\draw [->] (6,-9.5) -- (6,-10);
\node[draw=none] at (6.25,-9.75) {\textcircled{\small d}};

\draw (6.5, -13.5) -- (6.5, -11.5)  -- (8, -11.5) -- (8, -13.5) -- cycle;
\draw [fill=black] (8.5, -12.5) -- (9.5, -12.5)  -- (9.5, -10.5) -- (8.5, -10.5) -- cycle;
\draw (3.5, -13.5) -- (3.5, -11.5)  -- (4, -11.5) -- (4, -13.5) -- cycle;
\draw [fill=black] (4, -12.5) -- (5, -12.5)  -- (5, -10.5) -- (4, -10.5) -- cycle;
\draw (6,-12.25) -- (6,-12.75);
\draw (5.75,-12.5) -- (6.25,-12.5);
\draw [dashed] (8, -13.5) -- (8, -11.5)  -- (8.5, -11.5) -- (8.5, -13.5) -- cycle;

\draw [->] (2,-12.5) -- (1.5,-12.5);
\node[draw=none] at (1.75,-12.25) {\textcircled{\small e}};

\draw (-2.5,-13.5) -- (-2.5,-11.5)  -- (-0.5,-11.5) -- (-0.5,-13.5) -- cycle;
\fill [color=black] (-0.5,-12.5) -- (0.5,-12.5)  -- (0.5,-10.5) -- (-0.5,-10.5) -- cycle;

\node[draw=none] at (0,-14.5) {};

\end{tikzpicture}
\caption{Proof of Lemma~\ref{cutoffadjacent}, a graphical representation.}\label{lolpic}
\end{figure}

\subsubsection{Muirhead's inequality}
\begin{proposition}[Simple case of Muirhead's inequality]\label{SimpleMuirhead}~\\
Let $0<p_1\leq p_2\leq\frac{p}{2}$. Then for $a,b\geq 0$,
$$a^{p_1}b^{p-p_1}+a^{p-p_1}b^{p_1}\geq a^{p_2}b^{p-p_2}+a^{p-p_2}b^{p_2}.$$
\end{proposition}
\begin{proof}[Proof of Proposition~\ref{SimpleMuirhead}]~\\
We can rewrite the inequality as
$$a^{p_1}b^{p_1}(a^{p_2-p_1}-b^{p_2-p_1})(a^{p-p_2-p_1}-b^{p-p_2-p_1})\geq 0,$$
and this is true because by assumption, $p_2\geq p_1$, $p\geq p_1+p_2$.
\end{proof}
\noindent Inspired by \cite{KahPey}, we recall a consequence of this inequality:
\begin{corollary}[Expansion inequality]\label{ExpansionInequality}~\\
Let $k\in\N^{*}$ and $k\leq p\leq k+1$. Then for all $x,y\geq 0$, we have
\begin{equation}\label{eq:expansioninequality}
(x+y)^p\leq x^p+y^p+C(x^k y^{p-k}+x^{p-k}y^k)
\end{equation}
where $C$ is a constant depending only on $p$.
\end{corollary}
\begin{proof}[Proof of Corollary~\ref{ExpansionInequality}]~\\
We have, by sub-additivity of $x\mapsto x^{\frac{p}{k+1}}$,
\begin{align*}
&(x+y)^p\\
=&((x+y)^{p/k+1})^{k+1}\\
\leq&(x^{p/k+1}+y^{p/k+1})^{k+1}\\
\leq&x^p+y^p+\sum\limits_{i=1}^{k}C^{i}_{k+1}x^{\frac{ip}{k+1}}y^{\frac{(k+1-i)p}{k+1}}.
\end{align*}
By Proposition~\ref{SlepianLemma}, for each pair of crossterms with index $(i,k+1-i)$, we have (with $1\leq i\leq k$)
$$x^{\frac{ip}{k+1}}y^{\frac{(k+1-i)p}{k+1}}+x^{\frac{(k+1-i)p}{k+1}}y^{\frac{ip}{k+1}}\leq x^k y^{p-k}+x^{p-k}y^k.$$
Since the number of cross-terms is finite, one can choose $C$ big enough (say $C=2^{k+1}$) to conclude.
\end{proof}

\subsubsection{Slepian's lemma}
\begin{proposition}[Slepian's lemma]\label{SlepianLemma}
Let $T=\{1,\dots,n\}$. Let $\mathbf{X}$ and $\mathbf{Y}$ be two $n$-dimensional centered Gaussian vectors. Assume the existence of some subset $\Omega\subset T\times T$ such that:
\begin{align*}
&\mathds{E}[X_iX_j]\leq\mathds{E}[Y_iY_j],\quad (i,j)\in \Omega.\\
&\mathds{E}[X_iX_j]=\mathds{E}[Y_iY_j],\quad (i,j)\notin \Omega.
\end{align*}
Suppose $f:\mathds{R}^n\to\mathds{R}$ is smooth, with appropriate growth at infinity (exponential growth is fine) as well as its first and second derivative, and
$$\partial_{ij}f\geq 0\quad (\text{resp.}\leq 0),\quad (i,j)\in \Omega.$$
Then $\mathds{E}[f(X)]\leq\mathds{E}[f(Y)]$ (resp. $\mathds{E}[f(X)]\geq\mathds{E}[f(Y)]$).
\end{proposition}
\begin{proof}[Proof of Proposition~\ref{SlepianLemma}]~\\
See \cite[Theorem~3]{Zeitouni} for a slightly stronger form.
\end{proof}

\subsection{Proof of Lemma \ref{decomp}}\label{proof:decomp}
We introduce the Fourier coefficients $\alpha_{v}(n)\geq 0$ for $n \in \Z,v\in [1,\infty[ $ given by
\begin{equation*}
|\alpha_{v}(n)|^2=\frac{1}{2\pi} \int_{0}^{2\pi}  e^{-i n \theta}  (1-|v(e^{i \theta}-1)|^{\frac{1}{2}})_{+}   d \theta.
\end{equation*}

We consider a standard white noise $W$ on $[1,\infty[ \times \partial \D$ and we set
$$\bar{X}_\epsilon(e^{i \theta})=\sum_{n\in\Z} \alpha_{v}(n)e^{in\theta} \frac{1}{\sqrt{2\pi}}\int_1^{\frac{1}{\epsilon}}  \int_{0}^{2\pi} \frac{ e^{-i n u}}{\sqrt{2\pi v}} W(dv,d u)$$
Observe that $\alpha_{v}(n)=\alpha_{v}(-n)$ for $n\geq 0$ in such a way that $\bar{X}_\epsilon$ is real-valued. Then we can check that
\begin{equation*}
 \E[ \bar{X}_\epsilon (e^{i \theta}) \overline{\bar{X}_{\epsilon'} (e^{i \theta})}  ]  
 =   \sum_{n \in \Z} e^{i n (\theta-\theta')} \int_1^{\frac{1}{\epsilon}} |\alpha_v(n)|^2 \frac{dv}{v}  
= \int_1^{\frac{1}{\epsilon}}  (1-|v(e^{i \theta}-e^{i \theta'})|^{\frac{1}{2}})_+ \frac{dv}{v}.
\end{equation*}
Also, notice that we have
\begin{equation*} 
\bar{X} (e^{i \theta})=\sum_{n\in\Z} \alpha_{v}(n)e^{in\theta} \int_1^{\infty} \int_{0}^{2\pi} \frac{ e^{-i n u}}{\sqrt{2\pi v}} W(dv,d u).
\end{equation*}

Now we compute the correlations between the family $(\bar{X}_\epsilon)_\epsilon$ and the white noise $W$. We consider a smooth function $H:[1,+\infty[\to \R$ with compact support and a smooth function $f$ on $\partial \D$: we set $F=H\otimes f$ and
$$W(F)=\frac{1}{\sqrt{2\pi}}\int_{[1,+\infty[\times [0,2\pi]}H(v)f(e^{iu})W(dv,du).$$
Therefore, by considering the Fourier coefficients $(c_n(f))_n$ of $f$, we obtain 
\begin{align*}
T_\epsilon(F)(e^{i\theta})=&\E[ \bar{X}_\epsilon (e^{i \theta})W(H\otimes f)]\\
=&\sum_n\frac{1}{2\pi} \int_{[1,1/\epsilon[\times [0,2\pi]}\frac{\alpha_v(n)}{\sqrt{v}}e^{in\theta}f(e^{iu})e^{-inu}H(v)\,dvdu \\
=&\sum_n c_n(f)e^{in\theta} \int_{[1,1/\epsilon[}\frac{\alpha_v(n)}{\sqrt{v}} H(v)\,dv. 
\end{align*}
Because $H$ has compact support, it is readily seen that this series defines a continuous function of $\theta$, which converges uniformly as $\epsilon\to 0$ towards a continuous function given by
$$T(F)(e^{i\theta})=\sum_n c_n(f)e^{in\theta} \int_{[1,\infty[}\frac{\alpha_v(n)}{\sqrt{v}} H(v)\,dv.$$

\subsection{Backgrounds on fractional Brownian sheet}
We look at the main theorem in \cite{acosta} and we slightly modify the hypothesis (1.2).\\
Let $\{(Y_\epsilon^x : x\in[0,1]^d\}_{\epsilon>0}$ be a family of centerd Gaussian fields indexed by $[0,1]^d$ where $d$ is the dimension of the space. We suppose that for some constant $0<C_Y<\infty$,
\begin{equation}
\forall x,y\in[0,1]^d, \forall\epsilon>0, |Cov(Y_\epsilon^x,Y_\epsilon^y)+\log(\max\{\epsilon,|x-y|\})|\leq C_Y
\end{equation}
\begin{equation}\label{new1.2}
\mathds{E}[(Y_\epsilon^x - Y_\epsilon^y)^2]\leq C_Y \epsilon^{-1/2}|x-y|^{1/2} ~\text{if}~ |x-y|\leq\epsilon
\end{equation}
where $|\cdot|$ is the Euclidean distance.\\
We claim that
\begin{theorem}
There exist constants $0<c,C<\infty$ and a small $\epsilon_0>0$ (all depending of $C_Y$ and $d$) such that for all $0<\epsilon\leq\epsilon_0$ and all $\lambda\geq 0$,
$$\mathds{P}[|\max\limits_{x\in[0,1]^d}Y_\epsilon^x-m_\epsilon|\geq\lambda]\leq Ce^{-c\lambda}$$
\end{theorem}
We adapt the proof by introducing the fractional Brownian sheet. Recall that a (one-dimensional) fractional Brownian sheet $B^{H}_0=\{B^{H}_0, t\in\mathds{R}^N\}$ with Hurst index $H=(H_1,\dots,H_N), 0<H_j<1$ is a real-valued centered Gaussian field with covariance structure
$$\mathds{E}[B_0^H(s)B_0^{H}(t)]=\prod\limits_{j=1}^{N}\frac{1}{2}(|s_j|^{2H_j}+|t_j|^{2H_j}-|s_j-t_j|^{2H_j}), s,t\in\mathds{R}^N.$$\\
In particular $B_0^{H}$ is self-similar, i.e. for all constants $c>0$,
$$\big\lbrace B_0^{H}(ct), t\in\mathds{R}^N \big\rbrace=\big\lbrace c^{\sum_{j=1}^{N}H_j}B_0^{H}(t),t\in\mathds{R}^N \big\rbrace$$
in distribution.\\
In view of comparing with equation (\ref{new1.2}), we will choose a $d$-dimensional vector $H$ with all $H_j$ equal to $1/4$. Let us denote this particular fractional Brownian sheet by $\Phi$.\\
We now define the field $\Phi_\epsilon$ on $[0,\epsilon[^d$ by linearly shrinking the region $[p,2p[^d$ of $\Phi$ onto $[0,\epsilon[^d$, that is, $(\Phi_\epsilon(x),x\in[0,\epsilon[^d)=(\Phi(l(x)),l(x)\in[p,2p[^d)$ where $l$ is the affine map from $[0,\epsilon[^d$ to $[p,2p[^d$. Notice that $\Phi_\epsilon$ depends on the choice of $p$, and $p$ can be chosen as large as desired.\\
Let us recall two estimations that are useful for the proof (compare with equations (2.7) and (2.8) in \cite{acosta}):\\
Following the definition of fractional Brownian sheet:
\begin{equation}\label{new2.7}
p^{d/2}\leq Var(\Phi_\epsilon(x))\leq (2p)^{d/2}
\end{equation}
Combine self-similarity of $\Phi$ with lemma 3.4 from \cite{ayachexiao5} we deduce that there exist $c,C>0$ such that
\begin{equation}
cp^{d/2}\epsilon^{-1/2}|x-y|_2^{1/2}\leq\mathds{E}[(\Phi_\epsilon(x)-\Phi_\epsilon(y))^2]\leq Cp^{d/2}\epsilon^{-1/2}|x-y|_2^{1/2}
\end{equation}
where $|\cdot|_2$ is the $2$-norm, which is equivalent to the Euclidean norm.\\
New following \cite{acosta} we will divide $[0,1[^d$ into boxes of side length $\epsilon>0$ and assign values to each box using independent copies of $\Phi_\epsilon$.\\
We first recover Lemma~2.2 in \cite{acosta}. We claim that
\begin{lemma}
There exist constants $0<c,C<\infty$ (depending on $p$ and $d$) such that
\begin{equation}
\sup\limits_{v\in V_\epsilon}\mathds{P}(\sup\limits_{x\in\Box_\epsilon^v}\Phi_\epsilon(x)\geq\lambda)\leq Ce^{-c\lambda^2}
\end{equation}
\end{lemma}
To prove this lemma we use Fernique's majorizing measure argument. Notice that
$$B(x,r):=\big\lbrace y\in \Box_\epsilon^v : \mathds{E}[(\Phi_\epsilon(x)-\Phi_\epsilon(y))^2]\leq r^2 \big\rbrace\supset \big\lbrace y\in \Box_\epsilon^v : C p^{d/2}\epsilon^{-1/2}|y-x|_2^{1/2}\leq r^2 \big\rbrace$$
so that
$$\mu(B(x,r))\geq C r^{4d}$$
for some $C>0$ depending on $p$ and $d$.\\
Applying the majorizing measure technique we obtain
$$\mathds{E}[\sup\limits_{x\in\Box_\epsilon^v}\Phi_\epsilon(x)]\leq C\int_{0}^\infty \sqrt{-\log(cr^{4d})}dr\leq C<\infty$$
then we complete the proof of Lemma~2.2 by invoking Borell's inequality:
$$\mathds{E}[\sup\limits_{x\in\Box_\epsilon^v}\Phi_\epsilon(x)\geq C+\lambda]\leq e^{-\lambda^2/2(2p)^{d/2}}$$
the quantity on the right results from (\ref{new2.7}).\\
We then follow exactly the same steps as in \cite{acosta} (the only difference is to replace some $d$'s by $d/2$'s because of (\ref{new2.7})) to recover the main theorem.


\hspace{10 cm}

 \end{document}